\documentclass[12pt,twoside,reqno]{amsart}

\usepackage{amsmath, amssymb, amsthm, enumitem, esint, hyperref}

\newcommand{\IR}{\mathbb{R}}

\newcommand{\IN}{\mathbb{N}}

\newcommand{\LL}{\mathcal{L}}
\newcommand{\RR}{\mathcal{R}}
\newcommand{\XX}{\mathcal{X}}
\newcommand{\eps}{\varepsilon}
\newcommand{\ov}[1]{\overline{#1}}
\newcommand{\td}[1]{\widetilde{#1}}

\newcommand{\textQQqq}[1]{\qquad \text{#1} \qquad}

\newcommand{\textQQq}[1]{\qquad \text{#1} \quad}
\newcommand{\textQq}[1]{\quad \text{#1} \quad}

\DeclareMathOperator{\Ric}{Ric}
\DeclareMathOperator{\Rm}{Rm}

\DeclareMathOperator{\length}{length}
\newcommand{\rrm}{r_{\Rm}}

\newcommand{\rrmm}{r_{\Rm, 1}}

\newcommand{\EMPTY}[1]{}

\newtheorem*{Question1}{Question 1}
\newtheorem*{Question2}{Question 2}
\newtheorem*{HTConj}{Hamilton-Tian Conjecture}
\newtheorem*{Claim}{Claim}

\newtheorem{Theorem}{Theorem}[section]
\newtheorem{Lemma}[Theorem]{Lemma}
\newtheorem{Corollary}[Theorem]{Corollary}
\newtheorem{Proposition}[Theorem]{Proposition}
\newtheorem{Definition}[Theorem]{Definition}
\numberwithin{equation}{section}

\title{Convergence of Ricci flows with bounded scalar curvature}
\author{Richard H Bamler}
\address{Department of Mathematics, UC Berkeley, CA 94720, USA}
\email{rbamler@math.berkeley.edu}
\date{\today}

\begin{document}

\begin{abstract}
In this paper we prove convergence and compactness results for Ricci flows with bounded scalar curvature and entropy.
More specifically, we show that Ricci flows with bounded scalar curvature converge smoothly away from a singular set of codimension $\geq 4$.
We also establish a general form of the Hamilton-Tian Conjecture, which is even true in the Riemannian case.

These results are based on a compactness theorem for Ricci flows with bounded scalar curvature, which states that any sequence of such Ricci flows converges, after passing to a subsequence, to a metric space that is smooth away from a set of codimension $\geq 4$.
In the course of the proof, we will also establish $L^{p < 2}$-curvature bounds on time-slices of such flows.
\end{abstract}

\maketitle
\tableofcontents

\section{Introduction and statement of the main results}
\subsection{Introduction} \label{subsec:Introduction}
There has lately been a lot of progress in the study of compactness and degeneration behaviors of solutions to geometric equations, such as Einstein metrics, minimal surfaces or mean curvature flow (see for example \cite{MR999661,MR1074481, Cheeger-Colding-Cone, MR1937830, Cheeger-Naber-quantitative, Cheeger-Naber-Codim4,MR684900,MR1465365, MR3043387, MR3061773}).
These studies usually follow a common approach, which can be summarized as follows: One first specifies a topology, in which such solutions converge to a possibly singular space, after passing to a subsequence.
Secondly, one devises a partial regularity and structure theory for the limiting space.
This theory usually implies that the limiting space is smooth away from a singular set of small codimension and provides a characterizations of the tangent cones at the singular points.
Unfortunately, a theory of this kind has not been available for the Ricci flows, despite several interesting attempts or partial results (see \cite{ MR2846384, MR3373942, MR2666905, MR2572247, Sturm-super-RF}).
The goal of this paper is to carry out such a partial regularity and structure theory for Ricci flows that satisfy an additional scalar curvature bound.
Using this theory, we will characterize the formation of finite-time singularities of such flows.
As a corollary, we will obtain a general, Riemannian form of the Hamilton-Tian Conjecture.

Understanding the formation of finite-time singularities is an important goal in Ricci flow.
In dimensions 2 and 3, finite-time singularities are reasonably well understood.
In these dimensions, the maximum of the scalar curvature diverges at a singular time (see \cite{MR664497,MR1375255}) and the geometry of the singularity can be analyzed by a blow-up procedure.
More specifically, after normalizing the scalar curvature at a sequence of basepoints via parabolic rescaling, the flow subsequentially converges to a smooth singularity model.
In dimension 2, Hamilton and Chow showed (see \cite{MR954419,MR1094458}) that the only such singularity models are the round sphere and projective space, which is equivalent to saying that the flow becomes asymptotically round at a finite-time singularity.
In dimension~3, Perelman proved (see \cite{PerelmanI}) that the singularity models are $\kappa$-solutions, which he then classified in a qualitative way.
This classification was the basis of the construction of Ricci flows with surgery, which led to the resolution of the Poincar\'e and Geometrization Conjectures (see also \cite{PerelmanII}).
In higher dimensions, similar characterizations have only been obtained in relatively restrictive settings.
For example, it was shown by Sesum and Enders, M\"uller, Topping (see \cite{MR2255013, MR2886712}) that under a Type I bound on the entire Riemannian curvature tensor --- of the form $|{\Rm}_t| < C (T-t)^{-1}$ --- singularity models exist and are gradient shrinking solitons.
Apart from these results, the characterization and classification of finite-time singularities in Ricci flows has been largely open.

A main difficulty in the analysis of finite-time singularities of higher dimensional Ricci flows comes from the fact that --- at least a priori --- we may observe different types of singularity formation at different scales, leading to a bubble-tree-like structure.
At the smallest scale, such singularities can be described by smooth singularity models.
These models arise as blow-up limits due to a compactness theorem of Hamilton (see \cite{MR1333936}), which requires uniform bounds on the Riemannian curvature.
By contrast, singularity models describing the flow at larger scales are expected to be singular, as they may arise as limits of flows with unbounded curvature.
The analysis of the flow at such larger scales requires a reasonable compactness, partial regularity and structure theory, which has been missing so far.
The theory developed in this paper will partially fill this void and therefore allow the analysis of singularities at any scale below the scalar curvature scale (i.e., the scale that results in blow-up sequences with bounded scalar curvature).

The main result of this paper is a compactness and partial regularity theorem, which states that every non-collapsed sequence of Ricci flows with uniformly bounded scalar curvature converges, after passing to a subsequence, to a space that is smooth away from a singular set of codimension at least 4.
We refer to subsection~\ref{subsec:furtherresults} for a precise statement.
We will also derive various other structural properties of this limiting space, which, combined with earlier results (see \cite{Bamler-CGN}), imply that all tangent cones of the limiting space are metric cones.
As such, our characterizations of the limiting space are comparable to the partial regularity and structure theorems obtained for non-collapsed limits of Einstein manifolds or spaces with bounded Ricci curvature (see \cite{Cheeger-Colding-Cone, MR1937830, Cheeger-Naber-Codim4}).
In fact, as Einstein metrics can be trivially evolved into a Ricci flow with bounded scalar curvature, our theorem in some way generalizes these results; however, it does not provide an alternative proof, since our methods rely on generalizations of these results to the singular setting.

Let us now discuss in some more detail the structural results on the formation of finite-time singularities that follow from our main theorem.
Consider a Ricci flow $(g_t)_{t \in [0,T)}$
\[ \partial_t g_t = - 2 \Ric_{g_t} \]
on a manifold $M$, which possibly develops a singularity at some finite $T < \infty$.
We will answer the following questions:

\begin{Question1}
Suppose that the scalar curvature satisfies the bound $R < C$ on $M \times [0,T)$ for some constant $C < \infty$.
What can be said about the behavior of the metric $g_t$ as $t \nearrow T$?
\end{Question1}

\begin{Question2}
Suppose that there is some constant $C < \infty$ such that the scalar curvature satisfies $R(\cdot, t) < C (T-t)^{-1}$ for all $t \in [0,T)$.
What can be said about the behavior of the rescaled metric $(T-t)^{-1} g_t$ as $t \nearrow T$?
\end{Question2}

We will show that in the settings of both questions, the (rescaled) metric converges (subsequentially) to a singular space that possesses the partial regularity and structural properties as explained earlier.
More specifically, we obtain that this space is smooth away from a singular set of codimension at least 4 and that all its tangent cones are metric cones.
In the setting of Question~2, we moreover obtain that the limit is a gradient shrinking soliton on its regular part.
We refer to subsection~\ref{subsec:statement-main-results} for the precise statements of these results.

Question 1 is related to a famous conjecture that the scalar curvature near any finite-time singularity in Ricci flow must blow up.
This conjecture is equivalent to the conjecture that in the setting of Question 1 the metric $g_t$ converges to a smooth metric $g_T$ as $t \nearrow T$ and hence that the flow $(g_t)_{t \in [0,T)}$ can be extended past time $T$.
The conjecture is true in dimensions $n = 2, 3$, as pointed out earlier, and in the case in which $(g_t)_{t \in [0,T)}$ is a K\"ahler-Ricci flow (see \cite{ZZhang:2010}).
Moreover, if we replace the assumption that $R < C$ on $M \times [0,T)$ by the stronger assumption that $|{\Ric}| < C < \infty$ on $M \times [0,T)$, then the conjecture holds (see \cite{Sesum:2005}).
Even in dimension 4, a complete answer to Question 1 is still unknown.
In this dimension, it has, however, been proven recently that the $L^2$-norm of the Riemannian curvature tensor, $\int_M |{\Rm}|^2 dg_t$, remains uniformly bounded as $t \nearrow T$ and that the metric $g_t$ converges to a $C^0$-orbifold (see \cite{Bamler-Zhang-Part1, Simon:2015-bounds, Simon:2015-extending}).
The techniques used to obtain these results are very specific to dimension 4 and cannot be generalized to higher dimensions.
It was moreover shown in \cite{Simon:2015-extending} that in dimension 4, despite possible singularities, the flow $(g_t)_{t \in [0,T)}$ can still be extended past time $T$ by a Ricci flow on the limiting orbifold.
In other words, the flow can be continued if we allow the underlying manifold to change its topology.
This insight raises the question of whether a similar extension can be constructed in higher dimensions.
So one may wonder whether in the context of Question~1, the singular time-$T$-slice is ``regular enough'' such that the flow can be continued past time $T$, possibly via a singular flow.

The setting of Question 2 is a generalization of the Type I condition and it occurs naturally in the study of K\"ahler-Ricci flows on Fano manifolds (see \cite{Sesum-Tian:2008}).
By the recent resolution of the Yau-Tian-Donaldson Conjecture (see \cite{SDC2015-I, SDC2015-II, SDC2015-III, Tian:2012mq}), Fano manifolds admit K\"ahler-Einstein metrics under certain algebro-geometric conditions.
If these conditions are fulfilled, then the rescaled metric $(T-t)^{-1} g_t$ smoothly converges to one of the predicted K\"ahler-Einstein metrics as $t \nearrow T$ (see \cite{Tian-Zhu:2013}).
If these conditions are not assumed, then we have the following conjecture:

\begin{HTConj}
If $(g_t)_{t \in [0,T)}$ is a K\"ahler-Ricci flow on a Fano manifold $M$, then $(T-t)^{-1} g_t$ subsequentially converges to a compact K\"ahler-Ricci soliton, possibly away from a singular set of codimension $\geq 4$ as $t \nearrow T$.
\end{HTConj}

Progress towards the Hamilton-Tian Conjecture has been made in \cite{Tian-ZLZhang:2013-announcement, Tian-Zhang:2013, Chen-Wang-II}.
The approach in \cite{Chen-Wang-II} uses the Bergman kernel in a crucial way, which is only available in the K\"ahler setting.
This technique is non-standard in proving compactness of geometric equations and very different from the techniques used in this paper.

Our answer to Question~2, restricted to the K\"ahler case, implies the Hamilton-Tian Conjecture.
Our proof is purely Riemannian and does not require any tools from K\"ahler geometry.

Our structure theorem will follow by analyzing Ricci flows with bounded scalar curvature at different scales, using a blow-up argument.
Due to an earlier estimate of Zhang (see \cite{MR2923189}) on the volumes of geodesic balls in such flows, it is known that such blow-up sequences converge to a possibly singular metric space in the Gromov-Hausdorff topology.
Our goal in this paper will be to derive several analytic and geometric properties of these blow-up limits, which will imply that the blow-up limits are Ricci flat away from a well-behaved singular set.
These properties will then allow us to obtain further structural information, using a generalization of the theory of Cheeger, Colding and Naber to the singular setting, which was developed by the author in \cite{Bamler-CGN}.
The application of this theory is quite subtle, because we do not obtain a (synthetic) Ricci curvature bound on the singular set.
The proofs of the following two properties of the blow-up limit will occupy the greater part of this paper:
\begin{itemize}
\item A type of \emph{weak convexity property} of the set of regular points (called \emph{mildness of the singular set}), stating that almost every pair of regular points can be connected by a minimizing geodesic consisting only of regular points.
This property will be a consequence of a new regularity result for Ricci flows with bounded scalar curvature, which states that, under certain assumptions, almost every pair of points in a time-slice can be connected by an almost geodesic that avoids high curvature regions.
This regularity result will follow from a combination of heat kernel estimates and integral estimates along families $\mathcal{L}$-geodesics.
For more details see section~\ref{sec:firstconvsing}.
\item An \emph{$\eps$-regularity theorem}, asserting a curvature bound at centers of balls \emph{in the same time-slice} whose volume is sufficiently close to the corresponding Euclidean volume.
This theorem will follow from an analogous and new $\eps$-regularity theorem for Ricci flows with bounded scalar curvature.
We remark that similar $\eps$-regularity theorems for Ricci flows have been deduced in \cite{MR3245102,MR2328895}.
However, these results are not applicable here, as they impose other geometric assumptions and they don't hold in a single time-slice.
The proof of our $\eps$-regularity theorem relies, among other things, on a segment inequality in the context of Ricci flows.
We refer to section~\ref{sec:epsregularity} for more details.
\end{itemize}
Throughout the entire paper, we will moreover use a number of analytic and geometric estimates for Ricci flows with bounded scalar curvature that were developed by the author and Zhang in \cite{Bamler-Zhang-Part1, Bamler-Zhang-Part2}.
For further details on the proof see subsection~\ref{subsec:outline}

The main results of this paper motivate a number of interesting questions.

First, it seems desirable to obtain a more detailed characterization of the limiting flows of sequences of Ricci flows with bounded scalar curvature (not only of limits of time-slices).
By an earlier result by Zhang and the author in \cite{Bamler-Zhang-Part2}, such limiting flows evolve continuously with respect to the Gromov-Hausdorff distance.
However, the set of singular points may, a priori, vary in time, and even change its topology.
Due to the availability of heat kernel and distance distortion estimates, it is likely that such flows offer a useful model case for a synthetic definition of Ricci flows --- perhaps via optimal transport, generalizing the approach of Sturm (see \cite{Sturm-super-RF}).

Second, our results can be seen as a first step towards a general partial regularity and structure theory for non-collapsed Ricci flows.
Hence, it is an interesting question whether our theory can be generalized to a broader class of Ricci flows.
Even though our proof seems to rely on the scalar curvature bound at a number of steps, it is likely that this bound is crucial in only two steps: in the proofs of a lower (Gaussian) heat kernel bound and an upper bound on the distance distortion.
So one may wonder whether a similar structure theory can be developed under weaker assumptions that still guarantee these two bounds, or similar estimates.
In addition, our arguments seem to offer a certain amount of flexibility.
For example, several steps in our proof produce stronger estimates than what is needed in subsequent steps.
It is also interesting to observe that the generalization of the theory of Cheeger, Colding and Naber (see \cite{Bamler-CGN}), which is used in our proof, does not require a (synthetic) curvature characterization on the singular set.
Therefore, it may be possible that, in a more general setting, our arguments are  robust towards the loss of certain geometric or analytic control, as long as this loss occurs in an almost singular region of the flow.

\subsection{Statement of the main results --- Structure of singularities} \label{subsec:statement-main-results}
We will now state in detail the main results of this paper that describe the structure of finite-time singularities of Ricci flows with bounded scalar curvature.
That is, we will answer Questions 1 and 2 from the previous subsection.

Let us first consider Question 1.
In the setting of this question, the author showed previously in collaboration with  Zhang (see \cite{Bamler-Zhang-Part2}) that the induced length metric $d_t : M \times M \to [0, \infty)$ converges uniformly to a pseudometric $d_T : M \times M \to [0,\infty)$ as $t \nearrow T$, which is a metric with the exception that the distance between some distinct points is allowed to be zero.
In dimension $n=4$, the limiting metric is a $C^0$-orbifold. Away from the singular points, the metric and the convergence to this metric is smooth (see \cite{Bamler-Zhang-Part1, Simon:2015-extending}).
The first main result of this paper generalizes this fact to higher dimensions.
More specifically, it states that the convergence to the limiting metric is smooth away from a small set of Minkowski dimension $\leq n-4$:

\begin{Theorem}[Evolution of the flow under a uniform scalar curvature bound] \label{Thm:AnswerQuestion1}
Suppose that $R < C$ on $M \times [0, T)$ for some $C < \infty$ and let $d_T : M \times M \to [0, \infty)$ be the limiting pseudometric on $M$.

Then there is an open subset $\RR \subset M$ on which the metric $g_t$ smoothly converges to a smooth Riemannian metric $g_T$.
The complement of this subset, $M \setminus \RR$, has Minkowski dimension $\leq n-4$ with respect to $d_T$.
Moreover, the induced metric space $(M / \sim, d_T)$ is isometric to the completion of the length metric of the (incomplete) Riemannian manifold $(\RR, g_T)$.
\end{Theorem}

Similarly, in the setting of Question 2, we obtain smooth subconvergence of the rescaled flow to a singular gradient shrinking Ricci soliton:

\begin{Theorem}[Degeneration towards a singular gradient shrinking soliton] \label{Thm:AnswerQuestion2}
Assume that for some constant $C < \infty$ we have
\[ R(\cdot, t) < C (T-t)^{-1} \textQQqq{for all} t \in [0,T). \]
Then for any point $q \in M$ and any sequence of times $t_i \nearrow T$ we can choose a subsequence such that $(M, (T-t_i)^{-1} g_{t_i}, q)$ converges to a pointed, singular space $(\mathcal{X}, q_\infty)  = (X, d, \RR, g, q_\infty)$ that has singularities of codimension~4 in the sense of Definition \ref{Def:codimensionsingularities} (this implies that the singular set $X \setminus \RR$ has Minkowski dimension $\leq n-4$) and that is $Y$-regular at scale $1$ in the sense of Definition \ref{Def:Yregularity} for some $Y < \infty$ which only depends on $g_0$ and $C$.

Moreover, $\mathcal{X}$ is a shrinking gradient soliton in the following sense:
There is a smooth and bounded function $f_\infty \in C^\infty (\RR)$ that satisfies the shrinking soliton equation
\[ \Ric_g + \nabla^2 f_\infty = \tfrac{1}2 g \textQQqq{on} \RR. \]
\end{Theorem}

A precise definition of a singular space $\XX$ can be found in subsection \ref{subsec:terminology} (see Definition \ref{Def:singlimitspace}).
In a nutshell, a singular space is a metric space whose geometry is given by a smooth Riemannian metric on a generic subset.
The notion of ``convergence to a singular space'' is also made more concrete in subsection \ref{subsec:terminology} (see Definition \ref{Def:convergencescheme}).

In the K\"ahler-Fano case, Theorem \ref{Thm:AnswerQuestion2} implies the Hamilton-Tian Conjecture:

\begin{Corollary} \label{Cor:AnswerHT}
Let $(M^{2n}, (g_t)_{t \in [0,T)})$ be a K\"ahler-Ricci flow on a Fano manifold $M^{2n}$.
Then for any sequence of times $t_i \nearrow T$ we can find a subsequence such that $(M, (T-t_i)^{-1} g_{t_i})$ converges to a compact singular space $\mathcal{X}  = (X, d, \RR, g)$ that has singularities of codimension 4, that is $Y$-regular at scale $1$ for some $Y < \infty$ and that is a shrinking soliton in the sense of Theorem~\ref{Thm:AnswerQuestion2}.
\end{Corollary}

\subsection{Statement of the main results --- Compactness Theorem and further results} \label{subsec:furtherresults}
Theorems \ref{Thm:AnswerQuestion1}, \ref{Thm:AnswerQuestion2}, and by proxy Corollary \ref{Cor:AnswerHT}, will follow from a unifying compactness and partial regularity result for sequences of Ricci flows with bounded scalar curvature.
This compactness result states that for \emph{any} sequence of Ricci flows with bounded scalar curvature and uniformly bounded entropy, the final time-slices subconverge to a singular space, which is smooth away from a set of Minkowski dimension $\leq n-4$.

\begin{Theorem}[Compactness of Ricci flows with bounded scalar curvature] \label{Thm:maincompactness}
Let $(M_i, (g^i_t)_{t \in [-2,0]})$ be a sequence of Ricci flows on compact, $n$-dimensional manifolds $M_i$ and assume that there is a uniform constant $C < \infty$ such that the following holds:
\begin{enumerate}[label=(\roman*)]
\item The scalar curvature satisfies the uniform bound
\[ |R| < \rho_i < C \textQQqq{on}  M_i \times [-2,0] \]
for some sequence $\rho_i$.
\item Perelman's entropy satisfies the uniform lower bound
\[ \nu [g_{-2}^i, 4] := \inf_{0 < \tau < 4} \mu [g_{-2}^i, \tau] > - C. \]
(For a definition of $\mu [ g^i_{-2}, \tau]$ see subsection \ref{subsec:terminology}.)
\end{enumerate}
Let $q_i \in M_i$ be a sequence of basepoints.
Then, after passing to a subsequence, there is an $n$-dimensional, pointed singular space $(\mathcal{X}, q_\infty)$ (in the sense of Definition \ref{Def:singlimitspace}) such that the sequence of pointed Riemannian manifolds $(M_i, g_0^i, q_i)$ converges to $(\mathcal{X}, q_\infty)$ (in the sense of Definition \ref{Def:convergencescheme}).
Moreover, the singular space $\mathcal{X}$ has singularities of codimension $4$ (in the sense of Definition \ref{Def:codimensionsingularities}) and $\mathcal{X}$ is $Y$-regular at scale $1$ for some $Y = Y(n, C) < \infty$, which only depends on $n$ and $C$.
Lastly, if $\rho_i \to 0$, then $\mathcal{X}$ is Ricci flat (in the sense of Definition \ref{Def:singlimitspace}) and has mild singularities (in the sense of Definition \ref{Def:mild}).
\end{Theorem}

We remark that Theorem~\ref{Thm:maincompactness} shows that in the case $\rho_i \to 0$ the limiting space $\XX$ is regular enough to apply the generalization of the theory of Cheeger, Colding and Naber from \cite{Bamler-CGN}.
This theory gives us further structural characterizations on $\XX$ and its blow-up limits, and it also holds in the settings of Theorems~\ref{Thm:AnswerQuestion1} and \ref{Thm:AnswerQuestion2}.
For example, Theorem~\ref{Thm:maincompactness} combined with \cite[Theorems~1.5, Proposition~4.1]{Bamler-CGN} implies:

\begin{Corollary}
In the settings of Theorem~\ref{Thm:maincompactness}, Theorems~\ref{Thm:AnswerQuestion1} and \ref{Thm:AnswerQuestion2}, all tangent cones of the limiting space $\XX$ are metric cones.
\end{Corollary}

In the setting of Theorem~\ref{Thm:maincompactness}, we also obtain further characterizations of the behavior of the flow as $\rho_i \to 0$, which will be discussed in a subsequent paper.


Theorem \ref{Thm:maincompactness} and its proof imply an important geometric bound for Ricci flows with bounded scalar curvature.
Before introducing this bound, we need to recall the following terminology:

\begin{Definition}[Curvature radius] \label{Def:curvradius}
Let $(M, g)$ be a (not necessarily complete) Riemannian manifold and let $x \in M$ be a point.
Then we define the \emph{curvature radius $\rrm (x)$ at $x$} to be the supremum over all $r > 0$ such that the ball $B(x,r)$ is relatively compact in $M$ and such that $|{\Rm}| < r^{-2}$ on $B(x,r)$.

If $(g_t)_{t \in I}$ is a Ricci flow on $M$, then we denote by $\rrm (x,t)$ the curvature radius $\rrm(x)$ with respect to the metric $g_t$.

We will often denote by $\{ a < \rrm < b \}$ (in the static case) or $\{ a < \rrm (\cdot, t) < b \}$ (in the dynamic case) the set of all points $x \in M$ such that $a < \rrm (x) < b$ or $a < \rrm (x,t) < b$, respectively.
\end{Definition}

In the next result, we control the inverse of the curvature radius, $(\rrm (\cdot, t))^{-1}$, in the local $L^{p < 4}$ sense on each time-slice of a Ricci flow, in terms of a constant that only depends on an upper bound on the scalar curvature and a lower bound on the entropy.
We furthermore obtain $L^{p<2}$-bounds on the Riemannian curvature tensor.

\begin{Theorem}[Bounds on the curvature radius] \label{Thm:mainrrmbound}
For any $A < \infty$, $\eps > 0$ and $n$ there is an $C = C(A, \eps, n) < \infty$ such that the following holds:

Let $(M, (g_t)_{t \in [-2, 0]})$ be a Ricci flow on a compact $n$-dimensional manifold such that the following holds:
\begin{enumerate}[label=(\roman*)]
\item \label{(i)-Thm:mainrrmbound} The scalar curvature satisfies the uniform upper bound
\[ R < A \textQQqq{on}  M \times [-2,0]. \]
\item \label{(ii)-Thm:mainrrmbound} Perelman's entropy satisfies the uniform lower bound
\[ \nu [g_{-2}, 4] = \inf_{0 < \tau \leq 4} \mu [g_{-2}, \tau] > - A. \]
\end{enumerate}
Then for any $0 < r < 1$, $t \in [-1,0]$ and $x \in M$ we have
\[ \int_{B(x,t,r)} |{\Rm} (\cdot , t)|^{2-\eps} dg_t  \leq \int_{B(x,t,r)} \big( \rrm(\cdot, t) \big)^{-4+2\eps}  dg_t < C r^{n-4 + 2\eps}. \]
\end{Theorem}

Note that Theorem \ref{Thm:mainrrmbound} can be seen as a generalization of the main result of \cite{Cheeger-Naber-Codim4}.

Lastly, we mention that the size of the time-interval $[-2,0]$ in Theorem \ref{Thm:mainrrmbound} was chosen for technical reasons and can be adjusted to any other size via parabolic rescaling (see subsection \ref{subsec:terminology} for more details).
Likewise, the interval $[-1,0]$ in the assertion of Theorem \ref{Thm:mainrrmbound} can be chosen to be larger than half of the time-interval $[-2,0]$.

\subsection{Outline of the proof} \label{subsec:outline}
In the following, we give a brief outline of the proofs of our main results.
As mentioned before, Theorems \ref{Thm:AnswerQuestion1}, \ref{Thm:AnswerQuestion2} and Corollary \ref{Cor:AnswerHT} are deduced as a consequence of the compactness result, Theorem~\ref{Thm:maincompactness}.
This compactness result and the curvature bound, Theorem~\ref{Thm:mainrrmbound}, will be proven virtually simultaneously.

Before explaining the strategy of proofs of Theorems \ref{Thm:maincompactness} and \ref{Thm:mainrrmbound}, we first make the following observation:
In order to obtain the (strong) $L^{p<4}$-bound in Theorem~\ref{Thm:mainrrmbound}, it suffices to establish a similar weak $L^{p^*}$-bound for some $p^* > p$.
For technical convenience, we will often work with this weak $L^{p^*}$-bound in this paper and for brevity, we will refer to a weak $L^{p}$-bound simply as an ``$L^{p}$-bound'' in this outline.

The strategy of our proof is the following.
Let us fix some constants $E, p, p', p''$ such that $3 < p < p' < p'' < 4$.
We first introduce an a priori assumption, which states that the inverse of the curvature radius, $\rrm^{-1}$, is locally bounded by $E$ in a certain $L^p$-sense (similar as in the statement of Theorem \ref{Thm:mainrrmbound}).
Assuming this a priori assumption, we show that $\rrm^{-1}$ is bounded in an $L^{p'}$-sense, at sufficiently small scales, by a constant $C$, whose value is independent of $E$.
It is therefore possible to choose $E \gg C$ in our a priori assumption.
Using this conclusion, we carry out an induction argument over scales, which will imply that the  $L^{p'}$-bound from above holds even if we don't impose the a priori assumption.
More specifically, we can arrange things in such a way that the $L^{p}$-bound in the a priori assumption (involving $E$) follows from the $L^{p'}$-bound (involving $C$) at smaller scales.
Therefore, if the $L^p$-bound holds at scales below some $r < 1$, then the a priori assumption, and thus also the $L^p$-bound, hold at scales $< 10 r$.
Iterating this conclusion, will then allow us to derive an $L^p$-bound at all scales.

In order to derive the $L^p$-bound on $\rrm^{-1}$ under the a priori assumption, we prove a compactness and blow-up result, which is similar to Theorem~\ref{Thm:maincompactness}.
More specifically, we show that sequences of Ricci flows that satisfy the a priori assumption and whose scalar curvature is bounded by a constant that goes to $0$ converge, after passing to a subsequence, to a singular space $\XX$ that is Ricci flat away from a singular set of codimension $> 3$.
It will turn out that these limit spaces can be analyzed using the results of \cite{Bamler-CGN}, which generalize the theory of Cheeger, Colding and Naber (see \cite{Colding-vol-conv, Cheeger-Colding-Cone, Cheeger-Naber-quantitative, Cheeger-Naber-Codim4}) to the singular setting.
As a result, we obtain an $L^{p''}$-bound on $\rrm^{-1}$ on $\XX$.
From this bound we can deduce an $L^{p'}$-bound on Ricci flows with sufficiently small scalar curvature via a covering argument.

It is important to note that the limiting space $\XX$ can only be shown to be Ricci flat on its regular part.
Unfortunately, a (synthetic) characterization of the curvature on the singular points seems to be unavailable.
Therefore, the generalization of the theory of Cheeger, Colding and Naber to the singular setting becomes quite subtle.
In lieu of a curvature condition on the singular points of $\XX$, we have to show that $\XX$ possesses a number of geometric and analytic properties, which allow us to carry out this theory whatsoever.
A sufficient set of such properties has been worked out in \cite{Bamler-CGN}.
Luckily, several of these properties follow more or less naturally from earlier work of the author and Zhang (see \cite{Bamler-Zhang-Part1, Bamler-Zhang-Part2}).
However, as discussed in subsection~\ref{subsec:Introduction}, there are two properties --- the weak convexity property of the regular set (aka ``mildness of the singular set'') and an $\eps$-regularity property --- that require new regularity results for Ricci flows with bounded scalar curvature.
The proofs of these regularity results occupy most of this paper.

The general inductive multi-scale approach of our proof is quite common in the analysis of geometric PDEs and Ricci flows in particular.
The effectiveness of this approach in Ricci flows was first demonstrated by Perelman (cf \cite{PerelmanI, PerelmanII}).
In the setting of Ricci flows with bounded scalar curvature, the approach was also used in \cite{Bamler-Zhang-Part1}.
In the K\"ahler case, it was furthermore used by Chen and Wang (cf \cite{Chen-Wang-II}).
However, in their work, the authors needed to impose several strong additional a priori assumptions that they had to verify subsequently.
In our proof, these additional properties will be derived directly, making the structure of our argument more linear and transparent.

The paper is structured as follows.
Section \ref{sec:terminologyandconventions} contains the most important definitions and conventions used throughout this paper.

In section \ref{sec:PreliminariesRF}, we review the results from \cite{Bamler-Zhang-Part1} and \cite{Bamler-Zhang-Part2} that will be needed subsequently.

In section \ref{sec:firstconvsing}, we prove a result similar to the compactness statement of Theorem~\ref{Thm:maincompactness} under the a priori $L^p$-bound on $\rrm^{-1}$.
The main part of this section is devoted to the proof of the mildness of the singularities in the limit (see Definition~\ref{Def:mild}), which is needed in order to apply \cite{Bamler-CGN}.

In section \ref{sec:epsregularity} we show that the limiting space from the previous section is $Y$-regular (in the sense of Definition \ref{Def:Yregularity}), meaning that any ball with $Y^{-1}$-almost maximal volume has bounded curvature at its center.

In section \ref{sec:mainproof}, we combine the compactness statement under the a priori assumption (from section \ref{sec:firstconvsing}) with the $Y$-regularity of the limit $\XX$ (from section \ref{sec:epsregularity}) and the $L^p$ bound on $\rrm^{-1}$ (from \cite{Bamler-CGN}) to deduce the main Theorems~\ref{Thm:maincompactness} and \ref{Thm:mainrrmbound}.
Theorems \ref{Thm:AnswerQuestion1}, \ref{Thm:AnswerQuestion2} and Corollary \ref{Cor:AnswerHT} will follow immediately.

\section{Important terminology and conventions} \label{sec:terminologyandconventions}
\subsection{Terminology} \label{subsec:terminology}
We now give a precise definition of the terminology that was used in the theorems of corollaries of the previous subsection and which will be used throughout this paper.
Let us first introduce the following notion:
Given a measurable subset $S \subset M$ of a Riemannian manifold $(M,g)$ we will denote by $|S| = |S|_g$ the Riemannian measure of $S$ with respect to the metric $g$.
If $(g_t)_{t \in M}$ is a Ricci flow, then we often write $|S|_t := |S|_{g_t}$.

Next, we review Perelman's entropy formulas (cf \cite{PerelmanI}).
For any compact, $n$-dimensional Riemannian manifold $(M, g)$, any function $C^1 (M)$ and any $\tau > 0$ we define
\[ \mathcal{W} [g, f, \tau] := \int_M \big( \tau ( |\nabla f|^2 + R) + f - n \big) (4 \pi \tau)^{-n/2} e^{-f} dg. \]
We can then derive the following functionals:
For $\tau > 0$, we define
\[ \mu [ g, \tau ] := \inf_{\substack{f \in C^1 (M) \\ \int_M (4 \pi \tau)^{-n/2} e^{-f} dg = 1}} \mathcal{W} [ g, f, \tau] \]
and
\[ \nu [ g, \tau ] := \inf_{0 < \tau' < \tau} \mu [ g, \tau' ]. \]
Note that $\nu[g, \tau] \leq 0$. 
If $(g_t)_{t \in [0,\tau)}$ is a Ricci flow, then the functions $t \mapsto \mu [g_t, \tau - t]$, $t \mapsto \nu[g_t, \tau - t]$ and $t \mapsto \nu[g_t]$ are non-decreasing.
Moreover, by replacing $f \leftarrow f  - \frac{n}2 \log \tau_1 + \frac{n}2 \log \tau_2$, we can deduce the following estimate for any $0 < \tau_1 \leq \tau_2$:
\[ \mu [ g,  \tau_2 ] \geq \mu [ g, \tau_1] + \frac{n}2 \log \tau_1 - \frac{n}2 \log \tau_2, \]
and similarly
\[ \nu [ g, \tau_1] \geq \nu [ g,  \tau_2 ] \geq \nu [ g, \tau_1] + \frac{n}2 \log \tau_1 - \frac{n}2 \log \tau_2. \]
Lastly, note that the choice $\tau = 4$ in the conditions on the entropy in the theorems and corollaries of subsection \ref{subsec:Introduction} only serves our convenience and can be modified to any other constant.

We now define what we mean by the singular spaces that appeared in Theorem~\ref{Thm:maincompactness}.
The following definition comprises the most basic notions of a metric space that is smooth on a generic subset.

\begin{Definition}[singular space] \label{Def:singlimitspace}
A tuple $\mathcal{X} = (X, d, \RR, g)$ is called an \emph{($n$-dimensional) singular space} if the following holds:
\begin{enumerate}[label=(\arabic*)]
\item $(X,d)$ is a locally compact, complete metric length space.
\item $\RR \subset X$ is an open and dense subset that is equipped with the structure of a differentiable $n$-manifold whose topology is equal to the topology induced by $X$.
\item $g$ is a smooth Riemannian metric on $\RR$.
\item The length metric of $(\RR, g)$ is equal to the restriction of $d$ to $\RR$.
In other words, $(X,d)$ is the completion of the length metric on $(\RR, g)$.
\item There are constants $0 < \kappa_1 < \kappa_2 < \infty$ such that for all $0 < r < 1$
\[ \kappa_1 r^n < | B(x,r) \cap \RR | < \kappa_2 r^n. \]
Here $| \cdot |$ denotes the Riemannian volume with respect to the metric $g$ and distance balls $B(x,r)$ are measured with respect to the metric $d$.
\end{enumerate}
If moreover $\Ric_g = 0$ everywhere on $\RR$, then $\mathcal{X}$ is said to be \emph{Ricci flat}.
Also, if $q \in X$ is a point, then the tuple $(\mathcal{X}, q)$ or $(X, d, \RR, g, q)$ is called \emph{pointed singular space}.

The subset $\RR$ is called the \emph{regular part of $\XX$} and its complement $X \setminus \RR$ the \emph{singular part of $\XX$}.
\end{Definition}

We remark that this definition is similar to the corresponding definition in \cite{Bamler-CGN} with the only difference that $g$ is assumed to be smooth in this paper.
We furthermore emphasize that a the metric $d$ on $X$ is induced by the length metric of the Riemannian metric $g$ on $\RR$ (see item (4)).
So the distance between any two points in $\RR$ can be approximated arbitrarily well by the length of a smooth connecting curve in $\RR$.
This is an important property, which will take us some effort to establish.

We can generalize the concept of curvature radius from Definition \ref{Def:curvradius} to singular spaces $\mathcal{X} = (X, d, \RR, g)$ by defining the function $\rrm : X \to [0, \infty]$ as follows: we define $\rrm |_{X \setminus \RR} \equiv 0$ and for any $x \in \RR$ we let $\rrm (x)$ be the supremum over all $r > 0$ such that $B(x,r) \subset \RR$ and $|{\Rm}| < r^{-2}$ on $B(x,r)$. 

We will now define the following properties of singular spaces:

\begin{Definition}[singularities of codimension $\mathbf{p}_0$]  \label{Def:codimensionsingularities}
A singular space $\mathcal{X} = (X, \linebreak[1] d, \linebreak[1] \RR, \linebreak[1] g)$ is said to have \emph{singularities of codimension $\mathbf{p}_0$}, for some $\mathbf{p}_0 > 0$, if for any $0 < \mathbf{p} < \mathbf{p}_0$, $x \in X$ and $r_0 > 0$ there is an $\mathbf{E}_{\mathbf{p},x,r_0} < \infty$ (which may depend on $\XX$) such that the following holds:
For any $0 < r < r_0$ and $0 < s < 1$ we have
\[
 | \{ \rrm < s r \} \cap B(x, r) \cap \RR | \leq \mathbf{E}_{\mathbf{p},x,r} s^{\mathbf{p}} r^n. 
\]
\end{Definition}

It can be seen easily that if an $n$-dimensional singular space $\XX$ has singularities of codimension $\mathbf{p}_0$ in the sense of Definition \ref{Def:codimensionsingularities}, then its singular set $X \setminus \RR$ has Minkowski dimension $\leq n - \mathbf{p}_0$.

\begin{Definition}[mild singularities] \label{Def:mild}
A singular space $\mathcal{X} = (X, d, \RR, g)$ is said to have \emph{mild singularities} if for any $p \in X$ there is a closed subset $Q_p \subset \RR$ of measure zero such that for any $x \in \RR \setminus Q_p$ there is a minimizing geodesic between $p$ and $x$ that lies in $\RR$.
\end{Definition}

The idea behind the notion of mild singularities also occurs in the work of Cheeger and Colding (see \cite[Theorem~3.9]{Cheeger-Colding-structure-II} and Chen and Wang (see \cite[Definition 2.1]{Chen-Wang-II}).

\begin{Definition}[$Y$-regularity] \label{Def:Yregularity}
A singular space $\XX$ is called \emph{$Y$-regular at scales less than $a$}, for some $a, Y > 0$, if for any $p \in X$ and $0 < r < a$ the following holds:
If 
\[ | B(p, r) \cap \RR| > (\omega_n - Y^{-1}) r^n, \]
then $p \in \RR$ and $\rrm (y) > Y^{-1} r$.
Here $\omega_n$ denotes the volume of the standard $n$-dimensional ball in Euclidean space.
The space $\XX$ is said to be \emph{$Y$-regular at all scales}, if it is $Y$-regular at scale $a$ for all $a > 0$.
\end{Definition}

The notion of $Y$-regularity is standard in the study of Einstein metrics.
A similar notion has been used in \cite{Cheeger-Colding-structure-II} and \cite{MR999661} and, in the setting of Ricci flows with bounded scalar curvature, in \cite[Definition 3.3]{Chen-Wang-II} and \cite[Theorem~2.35]{Tian-Zhang:2013}.
We also remark that Definitions \ref{Def:codimensionsingularities} and \ref{Def:Yregularity} are similar to the corresponding definitions in \cite{Bamler-CGN} with the only difference that the curvature radius, as defined in Definition \ref{Def:curvradius}, that is used in these definitions here does not involve higher derivatives of the curvature tensor.
Due to parabolic regularity and backwards pseudolocality (see Proposition \ref{Prop:Pseudoloc}), this difference will turn out to be inessential.

It can be shown that in a $Y$-regular space with singularities of codimension $\mathbf{p}_0$ (for some $\mathbf{p}_0 > 0$), any point $p \in X$ whose tangent cone is isometric to $\IR^n$, is actually contained in $\RR$.
Therefore, the regular set $\RR$ and the metric $g$ in such a space is uniquely characterized by the metric $d$.

Next, we define what we understand by convergence towards a singular space.

\begin{Definition}[convergence and convergence scheme] \label{Def:convergencescheme}
Consider a sequence $(M_i, g_i, q_i)$ of pointed $n$-dimensional Riemannian manifolds and a pointed, $n$-dimensional singular space $(\mathcal{X}, q_\infty) = (X, d, \RR, g,q_\infty)$.
Let $U_i \subset \RR$ and $V_i \subset M_i$ be open subsets and $\Phi_i : U_i \to V_i$ be (bijective) diffeomorphisms such that the following holds:
\begin{enumerate}[label=(\arabic*)]
\item $U_1 \subset U_2 \subset \ldots$
\item $\bigcup_{i=1}^\infty U_i = \RR$.
\item For any open and relatively compact $W \subset \RR$ and any $m \geq 1$ we have $\Phi_i^* g_i \to g$ on $W$ in the $C^m$-sense.
\item There exists a sequence $q^*_i \in U_i$ such that
\[ d^{M_i} ( \Phi_i(q^*_i), q_i) \to 0. \]
\item For any $R < \infty$ and $\eps > 0$ there is an $i_{R, \eps} < \infty$ such that for all $i > i_{R, \eps}$ and $x, y \in B^X (q^\infty, R) \cap U_i$ we have
\[ \big| d^{M_i} (\Phi_i(x), \Phi_i(y)) - d^X (x,y) \big| < \eps \]
and such that for any $i > i_{R, \eps}$ and $x \in B^{M_i} (q_i, R)$ there is a $y \in V_i$ such that $d^{M_i} (x,y) < \eps$.
\end{enumerate}
Then the sequence $\{ (U_i, V_i, \Phi_i) \}_{i =1}^\infty$ is called a \emph{convergence scheme for the sequence of pointed Riemannian manifolds $(M_i, g_i, q_i)$ and the pointed singular space $(\mathcal{X}, q_\infty)$}.
We say that \emph{$(M_i, g_i, q_i)$ converges to $(\mathcal{X}, q_\infty)$} if such a convergence scheme exists.
\end{Definition}

\subsection{Conventions}
In the following we will fix a dimension $n \geq 3$ and we will omit the dependence of our constants on $n$.
Note that all the theorems above trivially hold in dimension $2$.

\section{Preliminaries on Ricci flows} \label{sec:PreliminariesRF}
\subsection{Ricci flows with bounded scalar curvature}
In this subsection we review some of the previous results on Ricci flows with bounded scalar curvature that we will need in the following.

We first recall the following volume bounds for distance balls.

\begin{Proposition}[volume bound] \label{Prop:VolumeBound}
For any $A < \infty$ there is a constant $C = C(A) < \infty$ such that the following holds:

Let $(M, (g_t)_{t \in [-2,0]})$ be a Ricci flow on a compact, $n$-dimensional manifold that satisfies 
\begin{enumerate}[label=(\roman*)]
\item $\nu [ g_{-2}, 4 ] \geq - A$.
\item $|R| \leq A$ on $M \times [-2, 0]$.
\end{enumerate}
Then for any $(x,t) \in M \times [-1,0]$ and $r > 0$ we have
\[ C^{-1} \big( \min \{ 1, r \} \big)^n \leq | B(x,t,r)|_t \leq C r^n e^{Cr}. \]
\end{Proposition}

The lower volume bound is due to Perelman's No Local Collapsing Theorem (cf \cite{PerelmanI}) and the upper bound is a consequence of the non-inflating property from \cite{MR2923189} or \cite{MR3061942}, see also \cite[Lemma 2.1]{Bamler-Zhang-Part2}.

Before we move on to the next result, we recall the definition of the curvature radius from Definition~\ref{Def:curvradius} in a Ricci flow:
\[ \rrm (x,t) = \sup \big\{ r > 0 \;\; : \;\; |{\Rm}| < r^{-2} \textQq{on} B(x,t,r) \big\}. \]
Note that by definition $\rrm (\cdot, t)$ is $1$-Lipschitz, which can be seen easily by contradiction:
If $\rrm (x,t) - \rrm (y,t) > d_t (x,y)$ for two points $x, y$, then for $r := \rrm (x,t) - d_t (x,y)$ we would have $B(y,t,r) \subset B(x,t, \rrm (x,t))$, and therefore $|{\Rm}| \leq \rrm^{-2} (x,t) \leq r^{-2}$ on $B(y,t,r)$, in contradiction to $r > \rrm (y,t)$.

We now recall the Backwards and Forward Pseudolocality Theorems for Ricci flows with bounded scalar curvature.

\begin{Proposition}[Pseudolocality, cf {\cite{PerelmanI}, \cite[Theorem 1.5]{Bamler-Zhang-Part1}}] \label{Prop:Pseudoloc}
For any $A < \infty$ there is a constant $\eps = \eps (A) > 0$ such that the following holds:

Let $(M, (g_t)_{t \in [-2,0]})$ be a Ricci flow on a compact, $n$-dimensional manifold that satisfies:
\begin{enumerate}[label=(\roman*)]
\item $\nu [ g_{-2}, 4 ] \geq - A$.
\item $|R| \leq A$ on $M \times [-2, 0]$.
\end{enumerate}
Then for any $(x,t) \in M \times [-1, 0]$ and $r := \min \{ 1, \rrm (x,t) \}$ we have
\begin{equation} \label{eq:pseudolocstatement}
 \rrm > \eps r \textQQqq{on} P (x,t, \eps r, - (\eps r)^2) \cup P (x,t, \eps r, \min\{  (\eps r)^2, -t \}). 
\end{equation}
\end{Proposition}

Here, $P(x,t,r,a)$ denotes the parabolic neighborhood $B(x,t,r) \times [t,t+a]$ or $B(x,t,r) \times [t+a,t]$, depending on whether $a$ is positive or negative.

Note that the minimum in (\ref{eq:pseudolocstatement}) is placed in the second parabolic neighborhood to ensure that the parabolic neighborhood does not reach past time $0$, where the flow is not defined.

Next, we recall the distance distortion bound from \cite{Bamler-Zhang-Part2}.

\begin{Proposition}[Distance distortion, cf {\cite[Theorem 1.1]{Bamler-Zhang-Part2}}] \label{Prop:Distdistortion}
For any $A, D < \infty$ there is a constant $C = C(A,D) < \infty$ such that the following holds:

Let $(M, (g_t)_{t \in [-2,0]})$ be a Ricci flow on a compact, $n$-dimensional manifold that satisfies 
\begin{enumerate}[label=(\roman*)]
\item $\nu [ g_{-2}, 4 ] \geq - A$.
\item $|R| \leq A$ on $M \times [-2, 0]$.
\end{enumerate}
Let $t_1, t_2 \in [-2,0]$ and $x,y \in M$ such that $d_{t_1}(x,y) \leq D$.
Then
\[ \big| d_{t_1} (x,y) - d_{t_2} (x,y) \big| \leq C  \sqrt{|t_1 - t_2|}. \]
\end{Proposition}

We will also need the Gaussian heat kernel bounds from \cite{Bamler-Zhang-Part1}.

\begin{Proposition}[Gaussian heat kernel bounds, cf {\cite[Theorem 1.4]{Bamler-Zhang-Part1}}] \label{Prop:GaussianHKbounds}
For any $A < \infty$ there is a constant $C = C(A) < \infty$ such that the following holds:

Let $(M, (g_t)_{t \in [-2,0]})$ be a Ricci flow on a compact, $n$-dimensional manifold that satisfies:
\begin{enumerate}[label=(\roman*)]
\item $\nu [ g_{-2}, 4 ] \geq - A$.
\item $|R| \leq A$ on $M \times [-2, 0]$.
\end{enumerate}
Let $K(x,t;y,s)$ be the fundamental solution of the heat equation coupled with the Ricci flow and let $-2 + A^{-1} \leq s < t \leq 0$.
Then
\[ \frac{1}{C (t-s)^{n/2}} \exp \Big( { - \frac{C d^2_s(x,y)}{t-s} } \Big) < K(x,t;y,s) < \frac{C}{(t-s)^{n/2}} \exp \Big({ - \frac{d^2_s(x,y)}{C (t-s)}} \Big). \]
With the help of Proposition \ref{Prop:Distdistortion}, the time-$s$ distance $d_s(x,y)$ can also be replaced by the time-$t$ distance $d_t(x,y)$ in the formula above.
\end{Proposition}

We will sometimes use the following corollary from Propositions \ref{Prop:VolumeBound} and \ref{Prop:GaussianHKbounds}.

\begin{Corollary} \label{Cor:HKintegrals}
For any $a > 0$ and $A < \infty$ there is a constant $C = C(a, A) < \infty$ such that the following holds:

Let $(M, (g_t)_{t \in [-2,0]})$ be a Ricci flow on a compact, $n$-dimensional manifold that satisfies:
\begin{enumerate}[label=(\roman*)]
\item $\nu [ g_{-2}, 4 ] \geq - A$.
\item $|R| \leq A$ on $M \times [-2, 0]$.
\end{enumerate}
Let $K(x,t;y,s)$ be the fundamental solution of the heat equation coupled with the Ricci flow and let $-1 \leq s < t \leq 0$.
Then for all $r \geq 0$ and $x_0 \in M$
\begin{equation} \label{eq:integralKoutsideB}
  \int_{M \setminus B(x_0, t, r)} K(x_0, t; y, s)  dg_{t} (y) < C \exp \Big( {- \frac{r^2}{C (t-s)} } \Big).
\end{equation}
Moreover,
\begin{equation} \label{eq:integralKda1}
 \int_{M} K(x_0, t; y, s) \big( d_t (x_0, y) \big)^a dg_{t} (y) < C (t-s)^{a/2}
\end{equation}
and
\begin{equation} \label{eq:integralKda2}
 \int_{M} K(x_0, t; y, s) \big( d_s (x_0, y) \big)^a dg_{s} (y) < C (t-s)^{a/2}. 
\end{equation}
Similarly as in Proposition \ref{Prop:GaussianHKbounds}, the time-$t$ balls $B(x_0, t,r)$ can be replaced by the time-$s$ balls $B(x_0, s, r)$ and the time-$t$ measure $dg_t$ can be replaced by the time-$s$ measure $dg_s$ and vice versa.
\end{Corollary}

\begin{proof}
Set $\tau := t-s \leq 1$.
Let us first check (\ref{eq:integralKoutsideB}).
This inequality holds trivially for $r^2 \leq \tau$ if $C$ is large enough, as its left-hand side is bounded by $1$.
So assume that $r^2 > \tau$.
Then we have, using Propositions \ref{Prop:VolumeBound} and \ref{Prop:GaussianHKbounds}, for some generic constant $C = C(A) < \infty$
\begin{align*}
 \int_{M \setminus B(x_0, t, r)} & K(x_0, t; y, s)  dg_{t} (y) 
\leq \int_M \frac{C}{\tau^{n/2}} \exp \Big({ - \frac{d^2_t(x,y)}{C \tau}} \Big) dg_t(y) \\
&=  \frac{C}{\tau^{n/2}} \sum_{k=0}^\infty \int_{B(x_0, t, 2^{k+1} r) \setminus B(x_0, t, 2^k r)} \exp \Big({ - \frac{d^2_t(x,y)}{C \tau}} \Big) dg_t(y) \displaybreak[1] \\
&\leq \frac{C}{\tau^{n/2}} \sum_{k=0}^\infty | B(x_0, t, 2^{k+1} r) |_t \cdot \exp \Big({ - \frac{(2^k r)^2}{C \tau}} \Big) \displaybreak[1]  \\
&\leq \frac{C}{\tau^{n/2}} \sum_{k=0}^\infty  (2^{k+1} r)^n \exp \Big( C 2^{k+1} r \Big) \cdot \exp \Big({ - \frac{(2^k r)^2}{C \tau}} \Big) \displaybreak[1]  \\
&\leq \frac{C r^n}{\tau^{n/2}} \sum_{k=0}^\infty  2^{n(k+1)} \exp \Big( {C^3 2^{2(k+1)} \tau + \frac{r^2}{4 C \tau}} \Big) \cdot \exp \Big({ - \frac{ r^2}{2C \tau}  - \frac{2^{2k}}{2C }} \Big) \displaybreak[1]  \\
&\leq \frac{C r^n}{\tau^{n/2}} \exp \Big( { - \frac{r^2}{4C \tau}} \Big) \sum_{k=0}^\infty  2^{n(k+1)}  \exp \Big( {C^3 2^{2(k+1)} \cdot 1  - \frac{2^{2k}}{2C}} \Big) \displaybreak[1]  \\
&\leq \frac{C r^n}{\tau^{n/2}} \exp \Big( { - \frac{r^2}{4C \tau}} \Big) \leq C \exp \Big( n \cdot \frac{r}{\tau^{1/2}}  \Big) \cdot \exp \Big( { - \frac{r^2}{4C \tau}} \Big) \displaybreak[1]  \\
&= C \exp \Big({ n \cdot  \frac{r}{\tau^{1/2}}   - \frac{r^2}{4C \tau}} \Big) \leq C \exp \Big( {- \frac{ r^2}{8C \tau}} \Big) .
\end{align*}
(Note that in the fourth inequality we have used $\sqrt{ab} \leq \frac12 (a + b)$ for $a = 2C^3 2^{2(k+1)} \tau$ and $b = \frac{r^2}{2C \tau}$.
In the fifth inequality, we have used $(2^k r)^2 \geq \frac12 2^{2k} r^2 + \frac12 2^{2k} r^2 \geq \frac12 r^2 + \frac12 2^{2k} \tau$.)
This shows (\ref{eq:integralKoutsideB}).
The bounds (\ref{eq:integralKda1}) and (\ref{eq:integralKda2}) follow using (\ref{eq:integralKoutsideB}) and Fubini's Theorem for some generic $C = C(a,A) < \infty$.
To see (\ref{eq:integralKda1}), we argue as follows:
\begin{align*}
 \int_{M} K(x_0, t; y, s) \big( d_t (x_0, y) \big)^a dg_{t} (y) 
 &= \int_M \int_0^{d_t(x_0,y)} K(x_0, t; y,s) \cdot a r^{a-1} dr dg_t(y) \displaybreak[1] \\
 &= \int_M \int_0^\infty \chi_{r \leq d(y,x_0)} K(x_0, t; y, s) a r^{a-1} dr dg_t (y) \displaybreak[1] \\
 &= \int_0^{\infty} a r^{a-1} \int_{M \setminus B(x_0, t, r)} K(x_0, t; y,s) dg_t (y) dr \displaybreak[1] \\
 &\leq \int_0^{\infty} C \exp \Big( {- \frac{r^2}{C \tau} }\Big)  r^{a-1} dr \\
 &= \sqrt{\tau} \int_0^\infty C \exp \Big({ - \frac{u^2}{C} }\Big)  ({u \sqrt{\tau}})^{a-1} dr
 \leq C \tau^{a/2}.
\end{align*}
The bound (\ref{eq:integralKda2}) can be derived similarly.
This finishes the proof.
\end{proof}

\begin{Proposition} \label{Prop:Ricsmall}
There are constants $C_0, C_1, \ldots < \infty$, which only depend on the dimension $n$ such that the following holds:
Let $(M, (g_t)_{t \in [-2,0]})$ be a Ricci flow on an $n$-dimensional manifold .
Let $(x,t) \in M \times [-1,0]$ and $0 < r < 1$ and assume that the ball $B(x, t, r)$ is relatively compact in $M$.
Assume that $|{\Rm}| < r^{-2}$ and $|R| \leq \rho$ on the parabolic neighborhood $P(x,t,r,-r^2)$ for some $0< \rho<1$.
Then for all $m \geq 0$
\begin{equation} \label{eq:Ricboundrless}
 |{\nabla^m \Ric}|(x,t) < C_m \rho^{1/2} r^{-m-1} 
\end{equation}
and
\begin{equation} \label{eq:nabRdtRmbound}
 |\nabla R|(x,t) < C_0 \rho^{3/4} r^{-1.5} \textQQqq{and} |{\partial_t \Rm}|(x,t) < C_0 \rho^{1/2} r^{-3}. 
\end{equation}
\end{Proposition}

\begin{proof}
For (\ref{eq:Ricboundrless}) see \cite[Lemma 6.1]{Bamler-Zhang-Part1} or \cite{MR2946223}.
The second bound in (\ref{eq:nabRdtRmbound}) can also be found in \cite[Lemma 6.1]{Bamler-Zhang-Part1}.
To see the first bound in (\ref{eq:nabRdtRmbound}), use that $|\nabla^2 R|  \leq |{\nabla^2 \Ric}| < 2C_2 \rho^{1/2} r^{-3}$ and $|R| \leq \rho$ in a parabolic neighborhood of $(x,t)$ and the interpolation inequality at scale $\rho^{1/4} r^{1.5}$.
\end{proof}

\subsection{$\LL$-geometry} \label{subsec:LLgeometry}
We now recall some of the basic definitions and facts of $\LL$-geometry, as introduced in \cite{PerelmanI}.
For more detailed proofs and explanations we also refer to \cite[sec 17-23]{MR2460872}.

Let $(M, (g_t)_{[a,b]})$ be a Ricci flow, $t_0 \in (a,b]$ and $0 \leq \tau_1 < \tau_2 \leq t_0 - a$.
Consider a smooth curve $\gamma : (\tau_1, \tau_2] \to M$.
We define its \emph{$\LL$-length} $\LL_{t_0} (\gamma)$ based at time $t_0$ as follows
\[ \LL_{t_0} (\gamma) := \int_{\tau_1}^{\tau_2} \sqrt{\tau} \big( |\gamma'(\tau)|^2_{t_0 - \tau} + R(\gamma(\tau), t_0 - \tau) \big) d\tau. \]
In the following, we will often consider the case $\tau_1 = 0$.
We will also frequently omit the index $t_0$ if the base time is clear.

The $\mathcal{L}$-length functional can be viewed as a Ricci flow analogue of the Riemannian energy functional for curves.
However, while the Riemannian energy functional is applied to curves in a static Riemannian manifold, the $\LL$-length should be seen as a functional for spacetime curves, whose endpoints lie in different time-slices.
Nevertheless, many of the notions in Riemannian geometry that are related to the energy functional --- such as (minimizing) geodesics, distances, the exponential map and the cut locus --- can be translated to equivalent notions related to the $\LL$-length functional.
In the following we will briefly recall these notions and compare them with their Riemannian counterparts.

Curves that minimize of the $\LL$-length for fixed endpoints are called \emph{minimizing $\LL$-geodesics}.
Note that, unlike in the Riemannian case, such minimizing curves may not necessarily be parameterized by constant speed.
The Euler-Lagrange equation for the $\LL$-functional (i.e. the equation satisfied by curves that are critical points of the $\LL$-functional) is the \emph{$\LL$-geodesic equation}:
\begin{equation} \label{eq_LL_geod_eq}
 \nabla^{g_{t_0- \tau}}_{\gamma' (\tau)} \gamma'(\tau) - \frac12 \nabla R (\gamma(\tau), t_0 - \tau) + \frac1{2\tau} \gamma' (\tau) +2 \Ric_{g_{t_0 - \tau}} (\gamma' (\tau)) =  0. 
\end{equation}
Solutions of the $\LL$-geodesic equation are called \emph{$\LL$-geodesics}.
It follows that minimizing $\LL$-geodesics are $\LL$-geodesics and it can be shown that, as in the Riemannian case, any $\LL$-geodesic is locally minimizing.
Moreover, it is known that if the underlying manifold $M$ is compact, then any two points $(x_1, t_1)$ and $(x_2, t_2)$ with $a \leq t_2 < t_1 \leq t_0$, $x_1, x_2 \in M$ can be connected by a minimizing $\LL$-geodesic $\gamma : [\tau_1, \tau_2 ] \to M$ with $\tau_1 = t_0 - t_1$ and $\tau_2 = t_0 - t_2$.
This $\LL$-geodesic is continuous for all $\tau$ and smooth whenever $\tau > 0$.

From now on we will only consider $\LL$-geodesics based at time $t_0$ and we will omit the index $t_0$.
For any basepoint $(x_0, t_0) \in M \times [a,b]$ and $(x,t) \in M \times [a,t_0)$ we define
\[ L(x,t) = L_{(x_0,t_0)} (x,t) := \inf \big\{ \LL(\gamma) \;\, : \;\, \gamma : [0, t_0 - t] \to M, \gamma(0) = x_0, \gamma(t_0-t) = x \big\} \]
to be the \emph{$\LL$-distance between $(x,t)$ and $(x_0, t_0)$}.
Note that if $M$ is compact, then the infimum in the definition above is attained.
We also set
\[ \ov{L} (x,t) = \ov{L}_{(x_0, t_0)} (x,t) := 2 \sqrt{t_0 - t} L_{(x_0, t_0)} (x,t) \]
and
\[ l (x,t) = l_{(x_0, t_0)} (x,t) := \frac1{2 \sqrt{t_0 - t}} L_{(x_0, t_0)} (x,t). \]
It was shown by Perelman that if $M$ is compact, then
\begin{equation} \label{eq_ov_L_equation}
 \partial_t \ov{L} \geq \triangle \ov{L} - 2n 
\end{equation}
in the barrier sense (meaning that for any $(x,t) \in M \times (a, t_0)$ and $\eps > 0$ there is an open neighborhood $U \subset M$ of $x$, $\delta > 0$ and a smooth function $\phi : U \times (t - \delta, t] \to \IR$ satisfying $\partial_t \phi \geq \triangle \phi - 2n - \eps$ such that $\ov{L} \leq \phi$ with equality at $(x,t)$).
This bound can be seen as a parabolic analogue of the Laplacian comparison theorem for Riemannian manifolds with lower Ricci curvature bounds.

Next, recall the $\LL$-geodesic equation (\ref{eq_LL_geod_eq}).
Given some $x \in M$ and $\tau > 0$, this equation admits a unique solution $\gamma : (0, t_0 - a] \to M$.
Moreover, it can be seen that $x_0 := \lim_{\tau \searrow 0} \gamma (\tau)$ and $\lim_{\tau \searrow 0} \sqrt{\tau} \gamma' (\tau) \in T_{x_0} M$ exist.
Vice versa, for any point $x_0 \in M$ and vector $v \in T_{x_0}M$ there is a unique $\LL$-geodesic with these properties.
We can therefore define for any $\tau > 0$ the \emph{$\LL$-exponential map}
\[ \LL \exp_{\tau} = \LL \exp_{(x_0, t_0), \tau} : T_{x_0} M \to M \]
such that $\gamma(\tau) = \LL \exp_\tau (v)$ is an $\LL$-geodesic with $\lim_{\tau \searrow 0} \sqrt{\tau} \gamma (\tau) =x_0$ and $\lim_{\tau \searrow 0} \sqrt{\tau} \gamma' (\tau) =v$.
The $\LL$-exponential map can be viewed as a Ricci flow version of the Riemannian exponential map.
The Jacobian of the $\LL$-exponential map, with respect to the measures induced by $g_{t_0} |_{x_0}$ on $T_{x_0} M$ and $d g_{t_0 - \tau}$ on $M$ is denoted by
\[ J^\LL (\cdot, \tau) = J^\LL_{(x_0, t_0)} (\cdot, \tau) : T_{x_0} M \to \IR. \]

Observe that if $M$ is compact, then $\LL \exp_\tau$ is surjective and, more specifically, for every $x \in M$ and $\tau > 0$ there is a $v \in T_{x_0} M$ such that $\LL \exp_\tau (v) = x$ and such that $\tau' \mapsto \LL\exp_{\tau'} (v)$ is minimizing on $[0, \tau]$.
This observation motivates the definition of an analogue of the segment domain in Riemannian geometry:
\begin{multline*}
\mathcal{D}^\LL_{\tau} = \mathcal{D}^\LL_{(x_0, t_0), \tau} := \big\{ v \in T_{x_0} M \;\; : \;\; \tau' \mapsto \LL\exp_{\tau'} (v) 
\textQq{is minimizing on} \\ [0, (1+\lambda) \tau] \textQq{for some} \lambda > 0 \big\}
\end{multline*}
and
\[ \mathcal{G}^\LL_{\tau} = \mathcal{G}^\LL_{(x_0, t_0), \tau} := \LL\exp_{(x_0,t_0), \tau} \big( \mathcal{D}^\LL_{(x_0, t_0), \tau}  \big) . \]
It is known that $\mathcal{D}^\LL_\tau$ and $\mathcal{G}^\LL_\tau$ are open and that
\[ \LL\exp_\tau : \mathcal{D}^\LL_\tau \to \mathcal{G}^\LL_\tau \]
is a diffeomorphism and that the complement $M \setminus \mathcal{G}^\LL_\tau$ has measure zero.
This complement can be viewed as an analogue of the Riemannian cut locus.

Finally, we mention an important result due to Perelman's, which states that along any minimizing geodesic $\gamma : [0, \tau] \to M$, $\gamma(\tau') = \LL \exp_{\tau'} (v)$, the quantity $\tau^{-n/2} e^{-l(\gamma(v), t_0-\tau)} J^\LL (v, \tau)$ is non-increasing in $\tau$.
Therefore, if $M$ is compact, then the \emph{reduced volume}
\begin{multline*} 
 \td{V} (\tau) = \td{V}_{(x_0, t_0)} (\tau) := \int_M (4 \pi \tau)^{-n/2} e^{- l (\cdot, t_0 -\tau)} dg_{t_0 - \tau} \\
 = \int_{\mathcal{D}^\LL_\tau} (4 \pi \tau)^{-n/2} e^{- l (\LL\exp_{\tau} (v), t_0- \tau)} J^\LL (v, \tau) dv 
\end{multline*}
is non-increasing in $\tau$ as well.
This monotonicity and (\ref{eq_ov_L_equation}) were used by Perelman to give an alternative proof of the no local collapsing Theorem.

\section{Compactness of Ricci flows under a priori assumptions} \label{sec:firstconvsing}
\subsection{Statement of the main result}
The goal of this section is to prove a compactness result for sequences of Ricci flows with bounded scalar curvature that satisfy an additional a priori uniform $L^p$-curvature bound.
We will show that such a sequence subconverges towards a singular space.
Moreover, if the sequence is obtained via a blow-up process, then the limiting singular space is Ricci flat and has certain regularity properties.
For example, it has mild singularities and is $Y$-tame at all scales, where $Y$ does not depend on the a priori $L^p$-curvature bound.
This fact will become important in the proof of Theorem \ref{Thm:mainrrmbound}.
The regularity conditions and the $Y$-tameness property will enable us later to carry out some of the steps in the theory of Cheeger, Colding and Naber (cf \cite{Colding-vol-conv,Cheeger-Colding-Cone,Cheeger-Naber-quantitative,Cheeger-Naber-Codim4}) on the limiting singular space in such a way that the constants involved in this theory do not depend on the a priori $L^p$-curvature bound.

The theory of Cheeger, Colding and Naber adapted to the singular setting is described in \cite{Bamler-CGN}.
Note that in the following result, we only obtain Ricci flatness of the limiting singular space on their regular part and we won't characterize the curvature on the singular part.
This will not create any issues for us, as the results in \cite{Bamler-CGN} surprisingly don't depend on such a curvature characterization.

Lastly, we mention that there is a further regularity property of singular spaces, namely $Y$-regularity, whose proof we will postpone to section \ref{sec:epsregularity}, as it relies on the compactness result from this section.
This $Y$-regularity property, will then allow us to carry out all the necessary steps in the theory of Cheeger, Colding and Naber, as described in \cite{Bamler-CGN},  in section \ref{sec:mainproof} and deduce $L^p$-curvature bounds that are independent of the a priori a priori $L^p$-curvature bounds.

The main compactness result of this section is the following:

\begin{Proposition}[compactness assuming a priori curvature bounds] \label{Prop:convtosingspaceLpasspt}
Let $(M_i, \linebreak[1] (g^i_t)_{t \in [-T_i, 0]})$, $T_i \geq 2$, be a sequence of Ricci flows on compact, $n$-dimensional manifolds $M_i$ such that
\begin{enumerate}[label=(\roman*)]
\item \label{(i)-Prop:convtosingspaceLpasspt} $\nu[g_{-T_i}^i, 2 T_i] \geq - A$ for some uniform $A < \infty$.
\item \label{(ii)-Prop:convtosingspaceLpasspt} $|R| \leq \rho_i$ on $M_i \times [-T_i, 0]$ for some sequence $0 < \rho_i < 1$.
\item \label{(iii)-Prop:convtosingspaceLpasspt} There is a constant $\mathbf{p}_0 > 2$ such that for any $0 < \mathbf{p} < \mathbf{p}_0$ there is a constant $\mathbf{E}_{\mathbf{p}} < \infty$ such that for all $(x,t) \in M_i \times [-T_i,0]$ and $0 < r, s < 1$ we have
\[ \big| \{ \rrm (\cdot, t) < s \} \cap B^{M_i} (x,t, r) \big|_t \leq \mathbf{E}_{\mathbf{p}} s^{\mathbf{p}} r^n . \]
\end{enumerate}
Let $q_i \in M_i$ be a sequence of basepoints.
Then, after passing to a subsequence, the pointed Riemannian manifolds $(M_i, g^i_0, q_i)$ converge (in the sense of Definition \ref{Def:convergencescheme}) to a pointed singular space $(\mathcal{X}, q_\infty) = (X, d, \RR, g, q_\infty)$ (in the sense of Definition \ref{Def:singlimitspace}) with singularities of codimension $\mathbf{p}_0$ (in the sense of Definition \ref{Def:codimensionsingularities}).
Moreover:
\begin{enumerate}[label=(\alph*)]
\item \label{(a)-Prop:convtosingspaceLpasspt} For all $0 < \mathbf{p} < \mathbf{p}_0$ and all $x \in X$ and $0 < r, s < 1$ we have
\[ \big| \{ \rrm^\infty < s \} \cap B^X (x,r) \cap \RR \big| \leq \mathbf{E}_{\mathbf{p}} s^{\mathbf{p}} r^n. \]
Here $\rrm^\infty$ denotes the curvature radius on $\XX$.
\item \label{(b)-Prop:convtosingspaceLpasspt} If $\rho_i \to 0$, then $\Ric \equiv 0$ on $\RR$ and $\XX$ has mild singularities (in the sense of Definition \ref{Def:mild}).
If furthermore $\mathbf{p}_0 > 1$, then $\XX$ is $Y(A)$-tame at scales $c(A) \sqrt{T_\infty}$, where $T_\infty := \limsup_{i \to \infty} T_i$ and $Y(A, \mathbf{p}_0) < \infty$ can be chosen only depending on $A$ and $\mathbf{p}_0$ and $c(A) > 0$ can be chosen only depending on $A$ (in particular, both of these constants are independent of $\mathbf{E}_{\mathbf{p}}$).
\end{enumerate}
The convergence $(M_i, g^i_0, q_i)$ to $(\XX, q_\infty)$ can be understood as follows:
there is a convergence scheme $\{ (U_i, V_i, \Phi_i) \}_{i = 1}^\infty$ such that:
\begin{enumerate}[label=(\alph*), start=3]
\item \label{(c)-Prop:convtosingspaceLpasspt} For any $x \in \RR$ and $r > 0$ we have
\begin{multline*}
\qquad\qquad \big| B^{X} (x, r) \cap \RR \big| \leq \liminf_{i \to \infty} \big| B^{M_i} (\Phi_i(x), 0, r) \big|_0 \\
 \leq \limsup_{i \to \infty} \big| B^{M_i} (\Phi_i(x), 0, r) \big|_0 \leq  \big| \ov{B^{X} (x,r)} \cap \RR \big|. 
\end{multline*}
\item \label{(d)-Prop:convtosingspaceLpasspt} For any $x \in \RR$ we have
\[ \rrm^\infty (x) = \lim_{i \to \infty} \rrm (\Phi_i ( x),0). \]
\item \label{(e)-Prop:convtosingspaceLpasspt} For any $D < \infty$ and $\sigma > 0$ and sufficiently large $i$ (depending on $D$ and $\sigma$) we have
\[ \rrm (\cdot, 0) < \sigma \textQQqq{on} B^{M_i} (q_i, 0, D) \setminus V_i \]
and
\[ \rrm^\infty < \sigma \textQQqq{on} B^{X} (q_\infty, D) \setminus U_i. \]
\end{enumerate}
\end{Proposition}

A more effective version of Proposition \ref{Prop:convtosingspaceLpasspt} is the following:

\begin{Proposition} \label{Prop:effectivelimitXX}
For any $A, E < \infty$ and $\eta > 0$ and $\mathbf{p}_0 > 2$ there is a $\rho = \rho (A, E, \eta, \mathbf{p}_0) > 0$ such that the following holds:
Let $(M, (g_t)_{t \in [-2,0]} )$ be a Ricci flow on a compact, $n$-dimensional manifold $M$, $x_0 \in M$ a point and $0 < r_0 < 1$ a scale and assume that
\begin{enumerate}[label=(\roman*)]
\item \label{(i)-Prop:effectivelimitXX} $\nu [ g_{-2}, 4] \geq - A$.
\item \label{(ii)-Prop:effectivelimitXX} $|R| \leq \rho$ on $M \times [-2,0]$.
\item \label{(iii)-Prop:effectivelimitXX} For all $(x,t) \in M \times [-1,0]$ and $0 < r, s < 1$ we have
\[ | \{ \rrm (\cdot, t) < sr \} \cap B(x,t, r) |_t \leq E s^{\mathbf{p}_0} r^n. \]
\end{enumerate}
Then for any $(q, t) \in M \times [-1/2,0]$ there is a pointed singular space $(\XX, q_\infty) = (X, d, \RR, g, q_\infty)$ with $\Ric \equiv 0$ on $\RR$ and mild singularities (in the sense of Definition \ref{Def:mild}), subsets $U \subset \RR$ and $V \subset M$ and a diffeomorphism $\Phi : U \to V$ such that the following holds:
\begin{enumerate}[label=(\alph*)]
\item \label{(a)-Prop:effectivelimitXX} $q_\infty \in U$ and $d_t (\Phi(q_\infty), q) < \eta$.
\item \label{(b)-Prop:effectivelimitXX} $\Vert \Phi^* g_t - g \Vert_{C^{[\mu^{-1}]} (U)} < \eta$.
\item \label{(c)-Prop:effectivelimitXX} $\rrm^{\infty} < \eta$ on $B^X (q_\infty, \eta^{-1} ) \setminus U$, where $\rrm^\infty$ denotes the curvature radius on $\XX$.
\item \label{(d)-Prop:effectivelimitXX} $\rrm (\cdot, t) < \eta$ on $B^M (q, t, \eta^{-1} ) \setminus V$ and
\[ |B^X (q, t, \eta^{-1} ) \setminus V |_t < \eta. \]
\item \label{(e)-Prop:effectivelimitXX} For any $y_1, y_2 \in U$ we have
\[ | d(y_1, y_2) - d_t ( \Phi (y_1), \Phi (y_2)) | < \eta. \]
\item \label{(f)-Prop:effectivelimitXX} For any $0 < r < \eta^{-1}$ we have
\[ \qquad (1- \eta) \big| B^{X} (q,r) \cap \RR \big| - \eta < \big| B^M (q, t, r) \big|_t < (1+ \eta) \big|B^X (q_\infty, r) \cap \RR \big| + \eta. \]
\item \label{(g)-Prop:effectivelimitXX} For all $x \in X$ and $0 < r, s < 1$ we have
\[ \big| \{ \rrm^\infty < sr \} \cap B^X (x,r) \cap \RR \big| \leq E s^{\mathbf{p}_0} r^n. \]
\end{enumerate}
\end{Proposition}

The major difficulty in the proof of Proposition \ref{Prop:convtosingspaceLpasspt} lies in verifying properties (C) and (E) of \cite[subsec 1.2]{Bamler-CGN} for the sequence of pointed Riemannian manifolds $(M_i, g^i_0, q_i)$.
These properties will imply that the length metric on the smooth part $\RR$ of the limiting space $\XX$ is equal to the restriction of its metric to $\RR$, as well as the mildness of singularities of $\XX$.
This will be achieved by constructing minimal time-$0$ geodesics as limits of certain $\LL$-geodesics, as explained in subsection \ref{subsec:MiLLgeodesic}.
The other statements will follow more or less using standard techniques and the results of \cite{Bamler-CGN}.

\subsection{Existence of short $\LL$-geodesics} \label{subsec:MiLLgeodesic}
In this subsection we will show that, in a Ricci flow with bounded scalar curvature, between almost any two points we can find an $\LL$-geodesic on a short time-interval that approximates a minimizing geodesic at time $0$.

We first show that the $\LL$-length is bounded from below in terms of the distance.
This proof is similar to one direction of the distance distortion bound of \cite{Bamler-Zhang-Part2}.

\begin{Lemma}[$\LL$-length is almost bounded from below by distance] \label{Lem:LLlowerbound}
For any $A, D < \infty$ there is a $C = C (A,D) < \infty$ such that the following holds:

Let $(M, (g_t)_{t \in [-2,0]} )$ be a Ricci flow on a compact, $n$-dimensional manifold $M$ and $x_0 \in M$ with the property that
\begin{enumerate}[label=(\roman*)]
\item $\nu [ g_{-2}, 4] \geq - A$.
\item $|R| \leq 1$ on $M \times [-2,0]$.
\end{enumerate}

Let $0 < \theta \leq \frac18$ and let $\gamma : [0, \theta] \to M$ be a smooth curve and denote by $\LL (\gamma)$ its $\LL$-length based at time $0$.
Assume that $d_0 (\gamma (0), \gamma(\theta)) \leq D$.
Then
\[ 2\sqrt{\theta} \LL(\gamma) > d_0^2 (\gamma(0), \gamma(\theta)) - C \theta^{1/3}. \]
\end{Lemma}

\begin{proof}
Let $x_0 := \gamma(0)$, $y_0 := \gamma(\theta)$ and $\eta := \theta^{2/3} \geq 2 \theta$.
Observe first that by Proposition \ref{Prop:Distdistortion}, we have
\begin{equation} \label{eq:distdistortioneta}
 | d_{t_1} (x_0, y_0) - d_{t_2} (x_0, y_0) | < C_1 \sqrt{\eta} \qquad \text{for all} \qquad t_1, t_2 \in [-\eta, 0],
\end{equation}
where $C_1 = C_1 (A,D) < \infty$.

Next consider the solution $u \in C^0 (M \times [-\eta, 0]) \cap C^\infty (M \times (- \eta, 0])$ to the heat equation coupled with the Ricci flow
\[ \partial_t u = \triangle_t u, \qquad u(\cdot, - \eta) = \min \{ d_{-\eta} (x_0, \cdot ), D \}. \]
Then by Corollary \ref{Cor:HKintegrals} we have for some constant $C_2 = C_2 (A) < \infty$
\begin{multline*}
 u(x_0,0) = \int_M K(x_0,0; z,-\eta ) \min \{ d_{-\eta} (x_0, z) , D \}dg_{-\eta} (z) \\
  \leq \int_M K(x_0,0; z,-\eta ) d_{-\eta} (x_0, z) dg_{-\eta} (z) \leq C_2 \sqrt{\eta}.
\end{multline*}

Next, by (\ref{eq:distdistortioneta}) we have
\[ D \geq d_0 (x_0, y_0) \geq d_{- \eta} (x_0, y_0) - C_1 \sqrt{\eta}. \]
So by the triangle inequality we have for all $z \in M$,
\begin{multline*}
d_{-\eta} (x_0, y_0 ) - \min \big\{ d_{-\eta} (x_0, z), D \big\} \\
\leq d_{-\eta} (x_0, y_0) - \min \big\{ d_{-\eta} (x_0, z), d_{-\eta} (x_0, y_0) - C_1 \sqrt{\eta} \big\} 
\leq d_{-\eta} (y_0, z) + C_1 \sqrt{\eta}.
\end{multline*}
We now estimate $u(y_0, -\theta)$.
Recall that $\theta \leq \frac12 \theta^{2/3} = \frac12 \eta$.
So we obtain, similarly as before,
\begin{alignat*}{1}
 u (y_0, -\theta) &= d_{-\eta} (x_0, y_0)  - \int_M K(y_0, -\theta; z, - \eta) \\
 &\qquad\qquad\qquad\qquad\qquad \cdot \big( d_{-\eta} (x_0, y_0) - \min \{ d_{-\eta} (x_0,z), D \} \big) dg_{-\eta} (z) \\
&\geq d_0 (x_0, y_0) - C_1 \sqrt{\eta} - \int_M K(y_0, -\theta; z, - \eta) d_{-\eta} (y_0, z) dg_{-\eta} (z) \\
&\geq d_0 (x_0, y_0) - C_3 \sqrt{\eta},
\end{alignat*}
where $C_3 = C_3 (A) < \infty$.
So, in conclusion,
\begin{equation} \label{eq:ux0y0}
u(y_0, -\theta) - u(x_0, 0) \geq d_0 (x_0, y_0) - C_4 \sqrt{\eta},
\end{equation}
where $C_4 = C_4 (A) < \infty$.

Consider now the quantity $|\nabla u |$ on $M \times [-\eta, 0]$.
We claim that in the barrier sense,
\[ \partial_t |\nabla u | \leq \triangle |\nabla u |. \]
Recall that this means the following: For any $(x,t) \in M \times (-\eta,0]$ and $\eps > 0$ there is an open neighborhood $U \subset M$ of $x$, $\delta > 0$ and a smooth function $\phi : U \times (t - \delta, t] \to \IR$ satisfying $\partial_t \phi \leq \triangle \phi + \eps$ such that $\phi \leq |\nabla u |$ with equality at $(x,t)$.
In fact, whenever $|\nabla u | > 0$, we have by Kato's inequality
\begin{multline*}
 \partial_t |\nabla u| = \frac{\partial_t |\nabla u |^2}{2|\nabla u|}  = \frac{2\langle \nabla \triangle u, \nabla u \rangle + 2\Ric (\nabla u, \nabla u)}{2|\nabla u|} =
\frac{ \triangle  \langle \nabla u, \nabla u \rangle - 2 |\nabla^2 u |^2 }{2|\nabla u|} \\
= \frac{ 2 \triangle |\nabla u| \cdot |\nabla u| + 2 |\nabla |\nabla u||^2 - 2 |\nabla^2 u |}{2|\nabla u|} 
\leq \triangle |\nabla u| 
\end{multline*}
So if $|\nabla u | (x,t) > 0$, then $|\nabla u|$ is smooth in a neighborhood of $(x,t)$ and we can set $\phi = |\nabla u|$.
On the other hand, if $|\nabla u |(x,t) = 0$, then we can set $\phi \equiv 0$.

Since $|\nabla u | (\cdot, -\eta) \leq 1$, we have by the maximum principle that
\begin{equation} \label{eq:nablau1}
 |\nabla u | \leq 1 \qquad \text{on} \qquad M \times [- \eta , 0].
\end{equation}
We can control the time-derivative of $u$ using \cite[Lemma 3.1(a)]{Bamler-Zhang-Part1}.
We obtain that there is a constant $C_5  = C_5(D)< \infty$ such that for all $t \in (-\eta, 0]$
\[ |\partial_t u| = |\triangle u| \leq \frac{C_5}{\eta + t} \qquad \text{on} \qquad M. \]
So, since $\theta \leq  \frac12 \eta$, we get
\begin{equation} \label{eq:dtubound}
 |\partial_t u| \leq \frac{2C_5}{\eta} \qquad \text{on} \qquad M \times [-\theta, 0]. 
\end{equation}
It follows using (\ref{eq:nablau1}) and (\ref{eq:dtubound}), that for any $\tau \in [0,\theta]$
\[ \frac{d}{d\tau} u (\gamma (\tau), -\tau) = \big\langle \nabla u (\gamma(\tau), -\tau), \gamma' (\tau) \big\rangle_{-\tau} - \partial_t u (\gamma(\tau), - \tau) \leq |\gamma' (\tau)|_{-\tau} + \frac{2C_5}{\eta} \]
So, using (\ref{eq:ux0y0}),
\begin{multline*}
 d_0 (x_0, y_0) - C_4 \sqrt{\eta} - \frac{2C_5}{\eta}\theta \leq u(y_0, - \theta) - u(x_0, 0) - \frac{2C_5}{\eta}\theta \\
 \leq \int_0^\theta |\gamma' (\tau)|_{-\tau} d\tau  
 \leq \bigg( \int_0^\theta 2 \sqrt{\tau} |\gamma' (\tau) |^2_{-\tau} d\tau \bigg)^{1/2} \bigg( \int_0^\theta \frac{1}{2\sqrt{\tau}} d\tau \bigg)^{1/2} \\ = \bigg( 2 \sqrt{\theta} \int_0^\theta  \sqrt{\tau} |\gamma' (\tau)|^2_{-\tau} d\tau \bigg)^{1/2}  . 
\end{multline*}
It follows that for some $C_6 = C_6 (A, D) < \infty$
\begin{multline*}
 2 \sqrt{\theta} \LL(\gamma) = 2 \sqrt{\theta} \int_0^\theta \sqrt{\tau} \big(  |\gamma' (\tau) |^2 + R(\gamma(\tau), - \tau) \big) d\tau \\
  \geq 2 \sqrt{\theta} \int_0^\theta \sqrt{\tau}  |\gamma'(\tau)|^2 d\tau - 2 \sqrt{\theta} \int_0^\theta \sqrt{\tau} \geq 2 \sqrt{\theta} \int_0^\theta \sqrt{\tau}  |\gamma' (\tau) |^2 d\tau - \frac{4}3 \theta^2 \\\geq \bigg( d_0 (x_0, y_0) - C_6 \sqrt{\eta} - 2C_7 \frac{\theta}{\eta} \bigg)^2 - \frac43 \theta^2 \geq d^2_0 (x_0, y_0) - C_6  \theta^{1/3} . 
\end{multline*}
This finishes the proof.
\end{proof}

Next we show that between almost every pair of points we can find a minimizing $\LL$-geodesic whose $\LL$-length is bounded from above by the distance at time $0$ and that satisfies the further condition that the inverse of the curvature radius along this $\LL$-geodesic is bounded from above in the $L^{1.5}$-norm.

\begin{Lemma}[short $\LL$-geodesics along which curvature is bounded in $L^{1.5}$] \label{Lem:shortLLLprrm}
For any $A, E, D < \infty$ and $\sigma_0, \theta_0, \delta > 0$ there are constants $\theta = \theta (A,  D, \theta_0,  \delta) \in (0, \theta_0)$ and $C = C(A, E, D, \theta_0, \sigma_0, \delta) < \infty$ such that the following holds:

Let $(M, (g_t)_{t \in [-2,0]} )$ be a Ricci flow on a compact, $n$-dimensional manifold $M$ and $x_0 \in M$ with the property that
\begin{enumerate}[label=(\roman*)]
\item \label{(i)-Lem:shortLLLprrm} $\nu [ g_{-2}, 4] \geq - A$.
\item \label{(ii)-Lem:shortLLLprrm} $|R| \leq 1$ on $M \times [-2,0]$.
\item \label{(iii)-Lem:shortLLLprrm} For all $(x,t) \in M \times [-1,0]$ and $0 < r, s < 1$ we have
\[ | \{ \rrm (\cdot, t) < sr \} \cap B(x,t, r) |_t \leq E s^2 r^n. \]
\item $\rrm (x_0, 0) > \sigma_0$.
\end{enumerate}

Then there is an open subset $S \subset B(x_0, 0, D)$ such that
\begin{enumerate}[label=(\alph*)]
\item  \label{(a)-Lem:shortLLLprrm} $|B(x_0, 0, D) \setminus S |_0 < \delta$.
\item \label{(b)-Lem:shortLLLprrm} For any $y \in S$ there is a minimizing $\LL$-geodesic $\gamma : [0, \theta] \to M$ between $x_0$ and $y$, $\gamma(0) = x_0$, $\gamma(\theta) = y$, such that
\begin{enumerate}[label=(b\arabic*)]
\item \label{(b1)-Lem:shortLLLprrm} Its $\LL$-length can be estimated as
 \[ | 2 \sqrt{\theta} \LL (\gamma) - d_0^2 (x_0, y) | < \delta . \]
\item \label{(b2)-Lem:shortLLLprrm} We have
\[ \gamma(\tau) \in B(x_0, 0, D+\delta) \textQQqq{for all} \tau \in [0, \theta]. \]
\item \label{(b3)-Lem:shortLLLprrm} We have the integral curvature bound
\[ \int_0^\theta \rrm^{-1.5} (\gamma (\tau), - \tau) d\tau < C. \]
\end{enumerate}
\end{enumerate}
\end{Lemma}

\begin{proof}
We will use the notation from subsection \ref{subsec:LLgeometry}.
Note that for any $(x,t) \in M \times [-2, 0]$
\begin{equation} \label{eq:lowerboundovL}
 \ov{L} (x,t) \geq - 2 \sqrt{-t} \int_{0}^{-t} \sqrt{-t'}  dt' = - \frac43 t^2  \geq - 2 t^2 .
\end{equation}

\begin{Claim}
There is a constant $\theta_1 = \theta_1 (A, D, \theta_0, \delta) \in (0, \theta_0)$ such that whenever $0 <\theta \leq \theta_1$ there is an open subset $S'_\theta \subset B(x_0, 0, D)$ such that
\[ | B(x_0, 0, D) \setminus S'_\theta |_0 < \delta / 2 \]
and such that for all $y \in S'_\theta$
\[   | \ov{L} (y, - \theta) -   d^2_0 (x_0, y)  | < \delta . \]
Here $\ov{L} = \ov{L}_{(x_0, 0)}$ denotes the $\ov{L}$-distance with respect to $(x_0, 0)$.
\end{Claim}

\begin{proof}
Assume that $0<\theta \leq \theta_1$, where $\theta_1$ will be determined in the course of the proof of this claim.
Due to Lemma \ref{Lem:LLlowerbound}, it suffices to verify the bound $\ov{L} (y, - \theta) < d_0^2 (x_0, y) + \delta$.
The idea behind the following proof is to utilize the following inequality from \cite[equation (7.15)]{PerelmanI}, which holds in the barrier sense
\[ \partial_t \ov{L} \geq \triangle \ov{L} - 2n . \]
Since $\lim_{t \to 0} \ov{L}(z, t) = d_0^2 (x_0, z)$ for all $z \in M$, we find that for all $z \in B(x_0, 0, D+1)$
\begin{equation} \label{eq:integralKovL}
 \int_M K(z, 0; y, - \theta) \ov{L} (y, - \theta) dg_{- \theta} (y) \leq d_0^2 (x_0,z ) + 2n\theta. 
\end{equation}
It follows that
\[ \int_M K(z, 0; y, - \theta) \big( \ov{L} (y, - \theta) - d_0^2 (x_0, z) \big) dg_{- \theta} (y) \leq  2n\theta. \]
So, using (\ref{eq:lowerboundovL}), Lemma \ref{Lem:LLlowerbound} and Corollary \ref{Cor:HKintegrals}, we get that for some $C_1 = C_1 (A, D), C_2 = C_2 (A, D) < \infty$ and for all $z \in B(x_0, 0, D+1)$
\begin{alignat*}{1}
 \int_{B(z,0,\sqrt{\theta})} & K(z,  0; y,  - \theta) \big( \ov{L} (y, - \theta)  - d_0^2 (x_0, z) \big)_+ dg_{- \theta} (y) \displaybreak[1] \\
  &\leq  2n\theta + \int_{M} K(z, 0; y, - \theta) \big( \ov{L} (y, - \theta)  - d_0^2 (x_0, z) \big)_- dg_{- \theta} (y) \displaybreak[1] \\
    &\leq 2n \theta +  \int_{B(x_0, 0, D+2)} K(z, 0; y, - \theta) \big( \big( \ov{L} (y, - \theta)  - d_0^2 (x_0, y) \big)_- \\
    &\qquad\qquad\qquad\qquad\qquad\qquad\qquad\qquad + \big| d_0^2 (x_0, z) - d^2_0 (x_0,y) \big| \big) dg_{- \theta} (y)  \\
    &\qquad\quad + \int_{M \setminus B(z, 0, 1)} K(z, 0; y, -\theta) \big( 2 \theta^2 + (D+1)^2 \big) dg_{-\theta} (y) \displaybreak[1] \\
    &\leq 2n \theta +  \int_{B(x_0, 0, D+2)} K(z, 0; y, - \theta) \big( C_1 \theta^{1/3} + 2 (D+2) d_0 (z,y) \big) dg_{- \theta} (y) \\
& \qquad\qquad\qquad\qquad\qquad\qquad\qquad +  \big( 2 \theta^2 + (D+1)^2 ) \cdot C_1 \exp \Big( { - \frac{1}{C_1 \theta} } \Big) \displaybreak[1] \\
    &\leq 2n \theta +  C_1 \theta^{1/3} + C_1 (D+2) \sqrt{\theta} + C_1 ((D+1)^2 + 2) \exp \Big( { - \frac{1}{C_1 \theta} } \Big) \\
&  \leq C_2 \theta^{1/3}.
\end{alignat*}
Using the lower heat kernel bound from Proposition \ref{Prop:GaussianHKbounds}, it follows that for some $C_3 = C_3 (A, D) < \infty$
\[ \theta^{-n/2} \int_{B(z,0, \sqrt{\theta})} \big( \ov{L} (y, - \theta)  - d_0^2 (x_0, z) \big)_+ dg_{0} (y) \leq C_3 \theta^{1/3}. \]
By the triangle inequality we have for all $y \in B(z, 0, \sqrt{\theta})$
\begin{multline*}
 |d^2_0 (x_0, y) - d^2_0 (x_0, z) | \leq (d_0 (x_0, y) + d_0 (x_0, z) ) \cdot | d_0 (x_0, y) - d_0  (x_0, z) | \\\leq 2 (D+1) \sqrt{\theta}.
\end{multline*}
So
\begin{multline*}
 \theta^{-n/2} \int_{B(z,0, \sqrt{\theta})} \big( \ov{L} (y, - \theta)  - d_0^2 (x_0, y) \big)_+ dg_{0} (y)  \\
 \leq  \theta^{-n/2} \int_{B(z,0, \sqrt{\theta})} \big( \big( \ov{L} (y, - \theta)  - d_0^2 (x_0, z) \big)_+ + 2 (D+1) \sqrt{\theta} \big) dg_{0} (y) \\
 \leq C_3 \theta^{1/3} + 2 (D+1) \sqrt{\theta} \leq C_4 \theta^{1/3}.
\end{multline*}
Letting $z$ vary over $B(x_0, 0, D+1)$ and using Fubini's Theorem and Proposition~\ref{Prop:VolumeBound} yields for some constants $C_5 = C_5 (A, D), C_6 = C_6 (A,D) < \infty$
\begin{alignat*}{1}
\int_{B(x_0, 0, D)} & \big( \ov{L} (y, - \theta)  - d_0^2 (x_0, y) \big)_+ dg_0 (y)  \\
& \leq C_5 \theta^{-n/2} \int_{B(x_0, 0, D)} \int_{B(y, 0 \sqrt{\theta})} \big( \ov{L} (y, - \theta)  - d_0^2 (x_0, y) \big)_+ dg_0 (z) dg_0 (y) \displaybreak[1] \\
&\leq C_5 \theta^{-n/2} \int_{B(x_0, 0, D+1)} \int_{B(z, 0, \sqrt{\theta})} \big( \ov{L} (y, - \theta)  - d_0^2 (x_0, y) \big)_+ dg_0 (y) dg_0 (z)\displaybreak[1] \\
&\leq C_5 C_4 \theta^{1/3} | B(x_0, 0, D+1)|_0  \leq C_6 \theta^{1/3}.
\end{alignat*}

We can now choose $\theta_1 = \theta_1 (A, D, \theta_0, \delta ) \in (0, \theta_0)$ uniformly in such a way that
\[ \int_{B(x_0, 0, D)}  \big( \ov{L} (y, - \theta)  - d_0^2 (x_0, y) \big)_+ dg_0 (y) < \delta^2/2. \]
Let $S'_\theta \subset B(x_0, 0, D)$ be the set of points $z \in B(x_0, 0, D)$ such that
\[ \ov{L} (z, -\theta) - d^2_0 (x_0, z) < \delta . \]
Then
\[ |B(x_0, 0, D) \setminus S'_\theta |_0 \cdot \delta \leq  \int_{B(x_0, 0, D)}  \big( \ov{L} (y, - \theta)  - d_0^2 (x_0, y) \big)_+ dg_0 (y) < \delta^2/2.  \]
It follows that
\[ |B(x_0, 0, D) \setminus S'_\theta |_0 < \delta/2. \]
This finishes the proof of the claim.
\end{proof}

The set $S$ will arise as a subset of $S'_\theta$ for sufficiently small $\theta \leq \theta_1$.
By the Claim, any subset of $S'_\theta$ already satisfies assertion \ref{(b1)-Lem:shortLLLprrm} of the lemma.
We will now show that for sufficiently small $\theta$, depending only on $A, D, \theta_0, \delta$, the set $S'_\theta$ also satisfies assertion \ref{(b2)-Lem:shortLLLprrm} in the sense that for any $y \in S'_\theta$ and \emph{any} minimizing $\LL$-geodesic $\gamma : [0, \theta] \to M$ between $x_0$ and $y$, $\gamma(0) = x_0$, $\gamma(\theta) = y$ we have the bound from assertion \ref{(b2)-Lem:shortLLLprrm}.
Consider such a minimizing $\LL$-geodesic $\gamma$.
Then for any $\tau_1 \in [0, \theta]$, we have
\begin{multline} \label{eq:tau1LL1}
2 \sqrt{\tau_1} \LL (\gamma|_{[0, \tau_1]} )
 \leq 2 \sqrt{\tau_1} \int_0^{\tau_1} \sqrt{\tau} \big( |\gamma'(\tau)|^2_{-\tau} + (R(\gamma(\tau), - \tau) + 1) \big) d\tau  \\
 \leq 2 \sqrt{\tau_1}  \int_0^{\theta} \sqrt{\tau} \big( |\gamma'(\tau)|^2_{-\tau} + (R(\gamma(\tau), - \tau) + 1) \big) d\tau \displaybreak[1] \\
 = 2 \sqrt{\tau_1} \bigg( \LL(\gamma) + \frac23 \theta^{3/2} \bigg) \leq 2 \sqrt{\theta} \LL (\gamma) + 4 \theta^2 .
\end{multline}
So, again by Lemma \ref{Lem:LLlowerbound}, we obtain for small enough $\theta$, depending on $\delta$, that
\begin{multline} \label{eq:tau1LL2}
 d_0^2(x_0, \gamma(\tau_1)) < 2 \sqrt{\tau_1} \LL (\gamma |_{[0, \tau_1]} ) + C \tau_1^{1/3} \\
< 2 \sqrt{\theta} \LL(\gamma) + 4 \theta^2 + C \theta^{1/3} 
< d^2(x_0, y) + 2 \delta < D^2 + 2\delta. 
\end{multline}
Note that in the third inequality we have applied assertion \ref{(b1)-Lem:shortLLLprrm}, which holds due to the Claim, assuming $y \in S \subset S'_\theta$.
The bound (\ref{eq:tau1LL2}) shows assertion \ref{(b2)-Lem:shortLLLprrm}, since $D > 1$.
We will fix $\theta$ from now on, set $S' := S'_\theta$ and focus on assertion \ref{(b3)-Lem:shortLLLprrm}.

For the rest of the proof let $C = C(A, D, E, \theta_0, \sigma_0, \delta) < \infty$ be a generic constant.
Using the estimates (\ref{eq:tau1LL1}) and (\ref{eq:tau1LL2}), we also obtain that for any such $\LL$-geodesic $\gamma : [0, \theta] \to M$ we have
\begin{multline*}
 d^2_0 (x_0, \gamma(\tau_1)) < 2 \sqrt{\tau_1} \LL (\gamma |_{[0, \tau_1]} ) + C \tau_1^{1/3} 
< 2 \sqrt{\tau_1} \bigg( \LL(\gamma) + \frac23 \theta^{3/2} \bigg) + C \tau_1^{1/3} \\
< 2 \sqrt{\tau_1} \bigg( \frac{D^2+1}{2 \sqrt{\theta}} + 1 \bigg) + C \tau_1^{1/3}. 
\end{multline*}
Since $\theta$ was determined in terms of $A$, $D$, $\theta_0$, $\delta$, this shows that we can choose a constant $0 < \tau_0 = \tau_0 (A, D, E, \theta_0, \sigma_0, \delta) < \min \{ (\eps \sigma_0)^2, \theta_0 \}$ such that
\begin{equation} \label{eq:closetox0easy}
 d_0 (x_0, \gamma(\tau)) < \eps \sigma_0 \textQQqq{for all} \tau \in [0, \tau_0]. 
\end{equation}
Here $\eps = \eps(A)$ is the constant from Proposition \ref{Prop:Pseudoloc}.
Note that this bound can also be derived using \cite[equation (26.8)]{MR2460872} and the upper bound on $\LL (\gamma)$.
See also Lemma \ref{Lem:LLnotleavparnbhd} below for a more precise result.

Consider now the $\LL$-exponential map at $x_0$
\[ \LL\exp_{\tau} = \LL\exp_{x_0, \tau} : T_{x_0} M \to M. \]
and the subsets $\mathcal{D}^\LL_\tau \subset T_{x_0} M$ and $\mathcal{G}^\LL_\tau \subset M$ as defined in subsection \ref{subsec:LLgeometry}.
Set moreover
\[ U := \LL\exp^{-1}_\theta (S') \cap \mathcal{D}^\LL_\theta. \]
Then
\[ \LL\exp_\theta (U) = S' \cap \mathcal{G}^\LL_\theta. \]

Define the function $f : M \times [0, \theta] \to [0, \infty)$ by
\[ f (x,\tau) := \big( \rrm (x, -\tau) \big)^{-1.5}. \]
Using assumption \ref{(iii)-Lem:shortLLLprrm}, Propositions~\ref{Prop:VolumeBound}, \ref{Prop:Distdistortion} and a covering argument we deduce that there is a constant $E^* = E^* (A, D, E) < \infty$ such that for all $t \in [-1,0]$ and $0 < s < 1$
\[ |\{ \rrm (\cdot, t) < s \} \cap B(x_0, 0, D+1) |_t \leq E^* s^2. \]
Using Fubini's Theorem and Proposition \ref{Prop:VolumeBound}, this implies that
\begin{multline} \label{eq:Fubini-computation}
\int_{B(x_0, 0, D+1)} f(x,\tau) dg_{-\tau}(x)
= \int_{B(x_0, 0, D+1)} \int_0^\infty \chi_{s < f(x, \tau)} ds dg_{-\tau}(x) \\
\leq  \int_{B(x_0, 0, D+1)} \int_1^\infty \chi_{s < f(x, \tau)} ds dg_{-\tau}(x) + |B(x_0, 0, D+1)|_{-\tau} \\
\leq \int_1^\infty |\{ s < \rrm^{-1.5}(\cdot, -\tau) \} \cap B(x_0, 0, D+1)|_{-\tau} ds + |B(x_0, 0, D+1)|_{-\tau} \\
\leq E^* \int_1^\infty (s^{-\frac{1}{1.5}})^2 ds + |B(x_0, 0, D+1)|_{-\tau}
< E^{**} = E^{**} (A, D, E) < \infty.
\end{multline}

Note that for any $v \in U$, the map $\tau \mapsto \LL\exp_{-\tau} ( v)$ describes a minimizing $\LL$-geodesic between $(x_0, 0)$ and $(\LL\exp_\theta (v), -\theta)$.
Since $\LL \exp_\theta (v) \in S'$, we know by Claim 1 that $\LL\exp_\tau (v) \in B(x_0, 0, D+1)$ for all $\tau \in (0, \theta]$.
In other words, $\LL\exp_\tau (U) \subset B(x_0, 0, D+1)$ for all $\tau \in (0, \theta]$.
So for any $\tau \in (0, \theta]$ we have
\begin{equation} \label{eq:intUfEss}
 \int_{U} f (\LL\exp_\tau (v), \tau) J^\LL (v,\tau) dv \leq \int_{B(x_0, 0, D+1)} f(x, \tau) dg_{-\tau} (x) <  E^{**}. 
\end{equation}
By the monotonicity of the quantity $\tau^{-n/2} e^{-l(\LL\exp_{\tau}(v), -\tau)} J^\LL (v, \tau)$, we find that for all $\tau \in [\tau_0, \theta]$
\begin{multline} \label{eq:JLLJLL}
 J^\LL (v, \theta) \leq \bigg( \frac{\theta}{\tau} \bigg)^{n/2} e^{l(\LL\exp_\tau(v), -\theta) -l(\LL\exp_\tau (v), -\tau)} J^\LL (v, \tau) \\
  \leq \bigg( \frac{\theta}{\tau_0} \bigg)^{n/2} \exp \bigg( \frac{D+1}{2 \sqrt{\theta}} + 1 \bigg) J^\LL (v,\tau). 
\end{multline}
Here we have used the lower bound
\[ l(\LL\exp_\tau (v), \tau) > - \frac1{2\sqrt{\tau}} \int_0^\tau \sqrt{\tau'} d\tau' > - 1. \]
Plugging (\ref{eq:JLLJLL}) back into (\ref{eq:intUfEss}) and integrating from $\tau_0$ to $\theta$ gives us
\begin{multline} \label{eq:tau0theta}
 \int_{\tau_0}^\theta  \int_{U} f (\LL\exp_\tau (v), \tau) J^\LL (v,\theta) dv d\tau \\
 \leq \bigg( \frac{\theta}{\tau_0} \bigg)^{n/2} \exp \bigg( \frac{D+1}{2 \sqrt{\theta}} + 1 \bigg) \int_{\tau_0}^\theta  \int_{U} f (\LL\exp_\tau (v), \tau) J^\LL (v,\tau) dv d\tau \leq C E^{**} . 
\end{multline}
On the other hand, recall from (\ref{eq:closetox0easy}) that for any $\tau \in (0, \tau_0]$ and $v \in U$ we have
\[ \LL\exp_{\tau} (v) \in B(x_0, 0, \eps \sigma_0) \]
Therefore, using Proposition \ref{Prop:Pseudoloc}), for any such $\tau$ and $v$
\[ f(\LL \exp_\tau (v), \tau) \leq ( \eps \sigma_0 )^{-1.5}. \]
It follows, using Proposition \ref{Prop:VolumeBound}, that
\begin{multline} \label{eq:0tau0}
 \int_0^{\tau_0} \int_{U} f (\LL\exp_\tau (v), \tau) J(v,\theta) dv d\tau  \leq ( \eps \sigma_0 )^{-1.5} \int_0^{\tau_0} \int_{U} J^\LL (v,\theta) dv d\tau \\
 = ( \eps \sigma_0 )^{-1.5} \tau_0 |B(x_0, 0, D)|_{-\theta} \leq  C. 
\end{multline}
Combining (\ref{eq:tau0theta}) and (\ref{eq:0tau0}) yields
\begin{equation} \label{eq:0thetaf}
 \int_{0}^\theta  \int_{U} f (\LL\exp_\tau (v), \tau) J^\LL (v,\theta) dv d\tau \leq C^* 
\end{equation}
for some $C^* = C^* (A,D,E,\theta_0, \sigma_0, \delta) < \infty$.

Define $h : U \to [0, \infty)$ by
\[ h(v) := \int_0^{\theta} f(\LL\exp_\tau (v), \tau ) d\tau. \]
Then, by Fubini's Theorem and (\ref{eq:0thetaf}),
\[ \int_U h(v) J^\LL (v, \theta) dv d\tau \leq C^*. \]
Let
\[ W := \big\{ v \in U \;\; : \;\; h(v) < 4\delta^{-1}C^* \big\} \]
and
\[ S := \LL\exp_\theta (W). \]
Then $S$ is open and $S \subset S'$.

Let us first check that $S$ satisfies assertion \ref{(a)-Lem:shortLLLprrm}.
To do this, observe that
\begin{multline*}
|S' \setminus S|_0 = |(S' \cap \mathcal{G}^\LL_\theta ) \setminus S |_0 \leq 2 | (S' \cap \mathcal{G}^\LL_\theta ) \setminus S |_{-\theta} 
= 2 |\LL\exp_\theta (U) \setminus \LL\exp_\theta (W) |_{-\theta} \\
\leq 2 \cdot \frac{\delta/4}{C^*} \int_{U \setminus W} h(v) J^\LL (v, \theta) dv  \leq \delta/2.
\end{multline*}
So
\[ |B(x_0, 0, D) \setminus S|_0 = |B(x_0, 0, D) \setminus S' |_0 + |S \setminus S'|_0 < \delta / 2 + \delta / 2 = \delta. \]

Next, we will check that $S$ satisfies assertion \ref{(b)-Lem:shortLLLprrm}.
Let $y \in S$ and $\gamma : [0, \theta] \to M$ be a minimizing $\LL$-geodesic between $x_0$ and $y$.
Since $y \in S \subset \mathcal{G}_\theta^\LL$, this $\LL$-geodesic is unique and there is a vector $v \in U \setminus W$ such that $\gamma (\tau) = \LL\exp_{\tau} (v)$.
As discussed earlier, assertions \ref{(b1)-Lem:shortLLLprrm}, \ref{(b2)-Lem:shortLLLprrm} hold.
For assertion \ref{(b3)-Lem:shortLLLprrm} observe that
\[ \int_0^\theta \rrm^{-1.5} (\gamma(\tau), -\tau) d\tau = \int_0^\theta f(\LL\exp_\tau (v), \tau) d\tau = h(v) < 4\delta^{-1} C^* =: C . \]
Note that the right-hand side only depends on $A, E, D, \theta_0, \sigma_0, \delta$.
\end{proof}

We will now use the integral curvature bound in Lemma \ref{Lem:shortLLLprrm}\ref{(b3)-Lem:shortLLLprrm} to integrate the $\LL$-geodesic equation, under the assumption that the scalar curvature is small.
This will then enable us to bound the speed $\gamma'(\tau)$ of any $\LL$-geodesic $\gamma : [0, \theta] \to M$ whose endpoint lies in $S$.

\begin{Lemma}[existence of $\LL$-geodesics with controlled speed] \label{Lem:LLcontrolledspeed}
For any $A, E, D < \infty$ and $ \sigma_0, \theta_0, \delta > 0$ there are constants $\theta = \theta (A, D, \theta_0, \delta) \in (0, \theta_0)$ and $\rho = \rho (A, E, D, \theta_0, \sigma_0, \delta) > 0$ such that the following holds:

Let $(M, (g_t)_{t \in [-2,0]} )$ be a Ricci flow on a compact, $n$-dimensional manifold $M$ and $x_0 \in M$ with the property that
\begin{enumerate}[label=(\roman*)]
\item $\nu [ g_{-2}, 4] \geq - A$.
\item $|R| \leq \rho$ on $M \times [-2,0]$.
\item For all $(x,t) \in M \times [-1,0]$ and $0 < r, s < 1$ we have
\[ | \{ \rrm (\cdot, t) < sr \} \cap B(x,t, r) |_t \leq E s^2 r^n. \]
\item $\rrm (x_0, 0) > \sigma_0$.
\end{enumerate}
Then there is an open subset $S \subset B(x_0, 0, D)$ such that
\begin{enumerate}[label=(\alph*)]
\item \label{(a)-Lem:LLcontrolledspeed} $|B(x_0, 0, D) \setminus S |_0 < \delta$.
\item \label{(b)-Lem:LLcontrolledspeed} For any $y \in S$ there is a minimizing $\LL$-geodesic $\gamma : [0, \theta] \to M$ between $x_0$ and $y$, $\gamma(0) = x_0$, $\gamma(\theta) = y$, such that the following is true:
\begin{enumerate}[label=(b\arabic*)]
\item \label{(b1)-Lem:LLcontrolledspeed} The $\LL$-length satisfies the bound
 \[ 2 \sqrt{\theta} \LL (\gamma) < d_0^2 (x_0, y) + \delta.  \]
 \item \label{(b2)-Lem:LLcontrolledspeed} We have
 \[ \gamma(\tau) \in B(x_0, 0, D+\delta) \textQQqq{for all} \tau \in [0, \theta]. \]
 \item \label{(b3)-Lem:LLcontrolledspeed} For any $\theta' \in (0, \theta]$ we have
\[ \frac{1}{2 \sqrt{\theta'}} \LL (\gamma |_{[0, \theta']} ) >  \frac{1}{2 \sqrt{\theta}} \LL (\gamma  ) - \delta. \]
\item \label{(b4)-Lem:LLcontrolledspeed} For all $\tau \in (0, \theta]$ we have
\[  \tau |\gamma'(\tau)|^2_{-\tau} < \frac{d^2_0(x_0,y) + \delta}{4\theta} . \]
\end{enumerate}
\end{enumerate}
\end{Lemma}

\begin{proof}
Without loss of generality, we may assume that $0 < \delta, \theta_0 < 1$ and $D \geq 1$.
Use Lemma \ref{Lem:shortLLLprrm} to choose and fix $\theta = \theta (A,  D, \theta_0,  \delta / 4) \in (0, \theta_0)$ and $C = C (A, \linebreak[1] E, \linebreak[1] D, \linebreak[1] \theta_0, \linebreak[1] \sigma_0, \linebreak[1] \delta / 4) < \infty$.
We now claim that the subset $S \subset B(x_0, 0, D)$ from Lemma \ref{Lem:shortLLLprrm} satisfies the assertions of this lemma if we assume that $\rho$ is chosen sufficiently small, depending only on $A, E, D, \theta_0, \sigma_0, \delta$ and on $\theta, C$.
Obviously, $S$ satisfies assertion \ref{(a)-Lem:LLcontrolledspeed}.

In order to verify assertion \ref{(b)-Lem:LLcontrolledspeed} choose $y \in S$ and pick a minimizing $\LL$-geodesic $\gamma : [0, \theta] \to M$ between $x_0$ and $y$ such that
 \begin{equation} \label{eq:2tauLLanddsquared}
  | 2 \sqrt{\theta} \LL (\gamma) - d_0^2 (x_0, y) | < \delta / 4  
  \end{equation}
and
\[ \int_0^\theta \rrm^{-1.5} (\gamma (\tau), - \tau) d\tau < C. \]
Assuming $\rho < \delta  / 4$, it follows that
\begin{multline} \label{eq:energyboundLL}
 \bigg| \int_{0}^{\theta} \sqrt{\tau} | \gamma' (\tau ) |_{- \tau}^2 d\tau - \LL(\gamma) \bigg| 
 \leq \int_{0}^{\theta} \sqrt{\tau} |  R (\gamma (\tau), - \tau) | d\tau 
 \leq \frac23 \rho \theta^{3/2} < \delta \theta / 4. 
\end{multline}
We will now bound the oscillation of $\sqrt{\tau}$ times the integrand on the left-hand side.
Set $\rrmm (x,t) := \min \{ \rrm(x,t), 1 \}$.
Using the $\LL$-geodesic equation for $\gamma$ (see \cite[(7.2)]{PerelmanI}) and Proposition \ref{Prop:Ricsmall}, we get that there is a constant $C_* = C_* (A) < \infty$, which only depends on $A$, such that for all $\tau \in (0, \theta]$
\begin{align*} 
 \bigg| \nabla^{g_{-\tau}}_{\gamma'(\tau)} \gamma' (\tau) + \frac1{2\tau} &  \gamma' (\tau) + \Ric_{-\tau} (\gamma' (\tau)) \bigg|_{-\tau}  \\
 & = \bigg| \frac12 \nabla R (\gamma(\tau), -\tau) -  \Ric_{-\tau} (\gamma' (\tau)) \bigg|_{-\tau}  \\
 &\leq C_* \big( \rho^{1/2} \rrmm^{-1.5} (\gamma(\tau), - \tau) + \rho^{1/2} \rrmm^{-1} (\gamma (\tau), - \tau) |\gamma' (\tau)|_{-\tau} \big)  \\
 &\leq C_* \rho^{1/2}  \rrmm^{-1.5} (\gamma(\tau), -\tau) \big( 1 + |\gamma' (\tau)|_{-\tau} \big) . 
\end{align*}
So
\begin{align*} 
 \frac{d}{d\tau} \big( \tau  |\gamma' (\tau)|_{-\tau}^2 \big)
&= 2 \tau \big\langle \nabla^{g_{-\tau}}_{\gamma' (\tau)} \gamma' (\tau), \gamma' (\tau) \big\rangle_{-\tau} +  |\gamma' (\tau)|_{-\tau}^2 + 2 \tau \Ric_{-\tau} (\gamma'(\tau), \gamma' (\tau))  \\
&\leq 2 \tau \cdot C_* \rho^{1/2} \rrmm^{-1.5} (\gamma(\tau), -\tau) \big( 1 + |\gamma' (\tau)|_{-\tau} \big) \cdot  |\gamma' (\tau)|_{-\tau}  \\
&\leq 4 C_* \rho^{1/2} \rrmm^{-1.5} (\gamma(\tau), -\tau) \big( 1 + \tau |\gamma' (\tau)|^2_{-\tau} \big) . 
\end{align*}
This implies that
\[ \frac{d}{d\tau} \log \big( 1+ \tau |\gamma' (\tau)|_{-\tau}^2 \big) \leq 4 C_* \rho^{1/2} \rrmm^{-1.5} (\gamma(\tau), -\tau) . \]
Integrating this inequality yields that for any $\tau_1, \tau_2 \in (0, \theta]$
\begin{multline}
 \bigg| \log \bigg( \frac{1+ \tau_1 |\gamma' (\tau_1)|_{-\tau_1}^2}{1+ \tau_2 |\gamma' (\tau_2)|_{-\tau_2}^2} \bigg) \bigg| \leq 4 C_* \rho^{1/2} \int_0^\theta  \big(  \rrm^{-1.5} (\gamma(\tau), -\tau) +1  \big) d\tau \\
\leq 4C_* \rho^{1/2} (C + \theta). \label{eq:logtaugammaprime}
\end{multline}
By combining (\ref{eq:2tauLLanddsquared}), (\ref{eq:energyboundLL}) and (\ref{eq:logtaugammaprime}), we find that there is a $\tau_* \in (0, \theta]$ such that
\[ \tau_* |\gamma' (\tau_*)|^2_{-\tau_*} \leq \sqrt{\theta} \cdot \sqrt{\tau_*}|\gamma' (\tau_*)|^2_{-\tau_*} < C_{**} (D, A, C). \]
So by (\ref{eq:logtaugammaprime}) a similar upper bound holds for all $\tau \in (0, \theta]$.
Thus, again by (\ref{eq:logtaugammaprime}), we can find a constant $\rho > 0$ whose value only depends on $D, A, C$, and therefore only on $A, E, D, \theta_0, \rho_0, \delta$, such that for any $\tau_1, \tau_2 \in (0, \theta]$
\[ \big| \tau_1 |\gamma'(\tau_1)|^2_{-\tau_1} - \tau_2 |\gamma'(\tau_2)|^2_{-\tau_2} \big| < \delta / 4. \]
Then, by (\ref{eq:energyboundLL}) and (\ref{eq:2tauLLanddsquared}), we have for any $\tau_0 \in (0, \theta]$
\begin{multline}
 \tau_0 |\gamma' (\tau_0)|^2_{-\tau_0} = \frac1{2\sqrt{\theta}} \int_0^\theta \frac{\tau_0 |\gamma' (\tau_0)|^2_{-\tau_0} }{\sqrt{\tau}} d\tau < \frac1{2\sqrt{\theta}} \int_0^\theta \frac{\tau |\gamma' (\tau)|^2_{-\tau} + \delta / 4}{\sqrt{\tau}} d\tau \\
  < \frac1{2\sqrt{\theta}} \big( \LL (\gamma) + \delta \theta / 4 \big) + \delta / 4 
  < \frac1{2\sqrt{\theta}}  \LL (\gamma) + \delta / 2 \\
  < \frac{d_0^2(x_0, y) + \delta / 4}{4\theta} + \delta / 4
  < \frac{d_0^2(x_0, y) + \delta}{4\theta} . \label{eq:taugammaprimeintupper}
\end{multline}
This proves assertion \ref{(b4)-Lem:LLcontrolledspeed}.
For assertion \ref{(b1)-Lem:LLcontrolledspeed} observe that, similarly as in (\ref{eq:taugammaprimeintupper}), for any $\tau_0 \in (0, \theta]$,
\begin{multline*}
 \tau_0 |\gamma' (\tau_0)|^2_{-\tau_0} = \frac1{2\sqrt{\theta}} \int_0^\theta \frac{\tau_0 |\gamma' (\tau_0)|^2_{-\tau_0} }{\sqrt{\tau}} d\tau > \frac1{2\sqrt{\theta}} \int_0^\theta \frac{\tau |\gamma' (\tau)|^2_{-\tau} - \delta / 4 }{\sqrt{\tau}} d\tau \\
 \geq \frac1{2 \sqrt{\theta}} \LL (\gamma) - \frac1{2\sqrt{\theta}} \int_0^\theta \sqrt{\tau} \rho d\tau - \frac{\delta \theta / 4}{2 \sqrt{\theta}}  > \frac1{\sqrt{\theta}} \LL(\gamma) - \theta \rho- \delta/2.
\end{multline*}
So
\begin{multline*}
 \frac1{2 \sqrt{\theta'}} \LL (\gamma |_{[0, \theta']} ) 
 = \frac1{2 \sqrt{\theta'}} \int_0^{\theta'} \bigg( \frac{\tau |\gamma'(\tau)|^2_{-\tau}}{\sqrt{\tau}} + \sqrt{\tau} R(\gamma(\tau), - \tau) \bigg) d\tau \\
  >  \frac1{2 \sqrt{\theta'}} \int_0^{\theta'} \frac1{\sqrt{\tau}} \bigg( \frac1{2\sqrt{\theta}} \LL(\gamma) - \theta \rho - \delta/2  \bigg) d\tau - \theta' \rho  = \frac1{2\sqrt{\theta}} \LL(\gamma) - \delta/2 - 2 \theta \rho .  
\end{multline*}
This establishes assertion \ref{(b3)-Lem:LLcontrolledspeed} for $\rho < \delta/4$.
\end{proof}

Before we continue with our analysis of $\LL$-geodesics, we establish the following technical lemma, which we will later apply to $\LL$-geodesics of controlled speed.
This lemma can be seen as a generalization of (\ref{eq:closetox0easy}) in the proof of Lemma \ref{Lem:shortLLLprrm} or of \cite[equation (26.8)]{MR2460872}.

\begin{Lemma}[curve of controlled speed cannot leave parabolic neighborhood too soon] \label{Lem:LLnotleavparnbhd}
For any $A, D < \infty$ and $\sigma_0 > 0$ there is a constant $0 < \alpha_0(A, D, \sigma_0) < 1/2$ such that the following holds:

Let $(M, (g_t)_{t \in [-2,0]} )$ be a Ricci flow on a compact, $n$-dimensional manifold $M$ and $x_0 \in M$ with the property that
\begin{enumerate}[label=(\roman*)]
\item $\nu [ g_{-2}, 4] \geq - A$.
\item $|R| \leq 1$ on $M \times [-2,0]$.
\item $\rrm (x_0, 0) > \sigma_0$.
\end{enumerate}

Assume that $0 < \theta < 1$ and that $\gamma : (0, \theta] \to M$, $\lim_{\tau \to 0} \gamma(\tau) = x_0$ is a smooth curve such that
\[ \tau |\gamma'(\tau)|_{-\tau}^2 < \frac{D^2+1}{4 \theta}  \textQQqq{for all} \tau \in (0 , \theta]. \]
Then 
\[ \gamma(\tau) \in B(x_0, 0, \sigma_0 / 10) \textQQqq{for all} \tau \in [0, \alpha_0 \theta]. \]
\end{Lemma}

\begin{proof}
By Proposition \ref{Prop:Pseudoloc} there is a constant $0 < \eps = \eps(A) < 1/10$ such that
\[ |{\Rm}| < (\eps \sigma_0)^{-2} \textQQqq{on} P := P(x_0, 0, \eps\sigma_0, - (\eps \sigma_0)^2). \]
By a simple distance distortion estimate this implies that for all $\tau \in (0, (\eps \sigma_0)^2]$
\[ |\gamma'(\tau)|_0 \leq 10 |\gamma'(\tau)|_{-\tau} \textQQqq{if} \gamma(\tau) \in B(x_0, 0, \eps \sigma_0). \]

Choose $\tau_0^* = \alpha_0^* \theta \in (0, \min \{ (\eps \sigma_0)^2, \theta \}]$ maximal with the property that $\gamma(\tau) \in B(x_0, \linebreak[1] 0, \linebreak[1] \eps \sigma_0)$ for all $\tau \in [0, \tau^*_0)$.
It suffices to derive a lower bound on $\alpha_0^*$ in terms of $A, D, \sigma_0$.
For any $\tau' \in [0, \tau_0^*)$
\begin{align*}
d_0 (x_0, \gamma(\tau')) &\leq \length_0 (\gamma |_{[0, \tau']} ) \leq \int_0^{\tau'} |\gamma' (\tau)|_0 d\tau 
\leq 10 \int_0^{\tau'} |\gamma'(\tau)|_{-\tau} d\tau \displaybreak[1] \\
&\leq 10 \int_0^{\tau'} \bigg( \frac{\tau |\gamma'(\tau)|^2_{-\tau}}{\tau} \bigg)^{1/2} d\tau
< 10 \int_0^{\tau'} \bigg( \frac{D^2  + 1}{4 \theta \tau} \bigg)^{1/2} d\tau \displaybreak[1] \\
&< 10 \int_0^{\tau_0^*} \frac{(D^2 + 1)^{1/2}}{\sqrt{\theta \tau}} d\tau 
< 20 \sqrt{\alpha_0^* \theta} \bigg(\frac{D^2+1}{\theta}  \bigg)^{1/2}  \\
&<  20 \sqrt{\alpha_0^*}\big(D^2 + 1 \big)^{1/2}
\end{align*}
It follows that $\tau^*_0 = \min \{ (\eps \sigma_0)^2, \theta \}$ or
\[ \eps \sigma_0 < 40 \sqrt{\alpha_0^*}\big( D^2 + 1 \big)^{1/2}. \]
In the second case we obtain a lower bound on $\alpha^*_0$ in terms of $A, D, \sigma_0$ immediately.
In the first case we have $\tau^*_0 = (\eps \sigma_0)^2$ or $\tau^*_0 = \theta$.
This implies that $\alpha^*_0 =  (\eps \sigma_0)^2 \theta^{-1} > (\eps \sigma_0)^2$ or $\alpha^*_0 = 1$.
So we still have a lower bound on $\alpha^*_0$ in terms of $A, D, \sigma_0$.
\end{proof}

Next, we improve the results of Lemmas \ref{Lem:shortLLLprrm} and \ref{Lem:LLcontrolledspeed} to find even more regular $\LL$-geodesics.
The main innovation of the following lemma, is that the integral curvature bound along short $\LL$-geodesics does not depend on the size of the time-interval $[0,\theta]$ on which these $\LL$-geodesics are defined or the size of the constant $\eta$ that governs the preciseness by which we can control the derivative of $\eta$.
So we will later be able to choose $\theta$ arbitrarily small, without deteriorating the integral curvature bound.
In order to achieve this independence, however, we have to assume that the scalar curvature is sufficiently small.

\begin{Lemma}[Existence of $\LL$-geodesics with controlled speed along which the curvature is bounded] \label{Lem:firstsetS}
For any $A, E, D < \infty$ and $\theta_0, \sigma_0, \delta, \eta > 0$ there are constants $C = C(A, E, D, \sigma_0, \delta) < \infty$, $\theta = \theta (A, E, D, \theta_0, \sigma_0, \delta, \eta) \in (0,\theta_0)$ and $\rho = \rho(A, \linebreak[1] E, \linebreak[1] D, \linebreak[1] \theta_0, \linebreak[1] \sigma_0, \linebreak[1] \delta, \linebreak[1] \eta) > 0$ such that the following holds:

Let $(M, (g_t)_{t \in [-2,0]} )$ be a Ricci flow on a compact, $n$-dimensional manifold $M$ and $x_0 \in M$ such that
\begin{enumerate}[label=(\roman*)]
\item \label{(i)-Lem:firstsetS} $\nu [ g_{-2}, 4] \geq - A$.
\item \label{(ii)-Lem:firstsetS} $|R| \leq \rho$ on $M \times [-2,0]$.
\item \label{(iii)-Lem:firstsetS} For all $(x,t) \in M \times [-1,0]$ and $0 < r, s < 1$ we have
\begin{equation} \label{eq:rrmlessELemfirstsetS}
 | \{ \rrm (\cdot, t) < sr \} \cap B(x,t, r) |_t \leq E s^{2} r^n. 
\end{equation}
\item $\rrm (x_0, 0) > \sigma_0$.
\end{enumerate}

Then there is a subset $S \subset B(x_0, 0, D)$ such that
\begin{enumerate}[label=(\alph*)]
\item \label{(a)-Lem:firstsetS} $|B(x_0, 0, D) \setminus S |_0 < \delta$.
\item \label{(b)-Lem:firstsetS} For any $y \in S$ there is a minimizing $\LL$-geodesic $\gamma : [0, \theta] \to M$ between $x_0$ and $y$, $\gamma(0) = x_0$, $\gamma(\theta) = y$ such that
\begin{enumerate}[label=(b\arabic*)]
\item \label{(b1)-Lem:firstsetS} Its $\LL$-length satisfies
\[ 2 \sqrt{\theta} \LL (\gamma) < d^2_0 (x_0, y) + \eta. \]
\item \label{(b2)-Lem:firstsetS} We have
 \[ \gamma(\tau) \in B(x_0, 0, D+\eta) \textQQqq{for all} \tau \in [0, \theta]. \]
\item \label{(b3)-Lem:firstsetS} For all $\tau \in (0, \theta]$ we have
\[  \tau |\gamma'(\tau)|^2_{-\tau} < \frac{d^2_0(x_0,y) + \eta}{4\theta}  . \]
\item \label{(b4)-Lem:firstsetS} We have
\[ \int_0^\theta \big( \rrm (\gamma(\tau), 0) \big)^{-1.5} d\tau < C \theta. \]
Note that here we take $\rrm$ at time $0$ and not at time $-\tau$ as in Lemma~\ref{Lem:shortLLLprrm}\ref{(b3)-Lem:shortLLLprrm}.
\end{enumerate}
\end{enumerate}
\end{Lemma}

Note that in this lemma the constant $C$ does not depend on $\eta$ or $\theta_0$.

\begin{proof}
We use a similar argument as in the proof of Lemma \ref{Lem:shortLLLprrm}.
The difference is that this time we have a better estimate on the Jacobian of the $\LL$-exponential map, due to assertion \ref{(b2)-Lem:LLcontrolledspeed} in Lemma \ref{Lem:LLcontrolledspeed}.
This fact will allow us to choose $C$ independently of $\theta$.

Assume without loss of generality that $0 < \theta_0, \sigma_0, \delta, \eta < 1$.
Let us first establish the following two bounds:

\begin{Claim}
There is a constant $E^* = E^* (A, E, D) < \infty$ such that
\begin{equation} \label{eq:intrrm-15}
 \int_{B(x_0, 0, D+1)} \rrm^{-1.5} (x,0) dg(x) < E^*. 
\end{equation}
and
\[ |B(x_0, 0, D)|_0 < E^*. \]
\end{Claim}

\begin{proof}
The second bound is a direct consequence of Proposition~\ref{Prop:VolumeBound}.
For the first bound follows note that a ball packing argument combined with Proposition~\ref{Prop:VolumeBound} implies that $B(x_0, 0, D+1)$ can be covered by a bounded number of $1$-balls  (depending only on $A$, $D$).
So (\ref{eq:rrmlessELemfirstsetS}) in assumption \ref{(iii)-Lem:firstsetS} also holds for $x = p$ and $r = D+1$ if we replace $E$ by a constant depending only on $A, E, D$.
The bound (\ref{eq:intrrm-15}) now follows similarly as in (\ref{eq:Fubini-computation}).
\end{proof}

We can now choose the constants $C$, $\theta$ and $\rho$.
Let $\alpha_0 = \alpha_0 (A,D, \sigma_0)$ be the constant from Lemma \ref{Lem:LLnotleavparnbhd}.
Next, choose
\[ E^{**} = E^{**} (A, E, D, \sigma_0) := 2e \alpha_0^{-n/2} E^* + 2(\sigma_0/2)^{-1.5}E^*  \]
and
\[ C = C(A, E, D, \sigma_0, \delta) := 4 E^{**} \delta^{-1}. \]
Note that $C$ does not depend on the choice of $\theta_0$ or $\eta$.
Next, we choose
\[ \delta^* := \min \{ \delta/2, \eta \} \]
and we assume that $\rho < \min \{ \rho_{\ref{Lem:LLcontrolledspeed}} (A, E, D, \theta_0, \sigma_0, \delta^*), 1/10 \}$, where $\rho_{\ref{Lem:LLcontrolledspeed}}$ is the constant from Lemma \ref{Lem:LLcontrolledspeed}.

We now apply Lemma \ref{Lem:LLcontrolledspeed} with $A \leftarrow A$, $E \leftarrow E$, $D \leftarrow D$, $\theta_0 \leftarrow \theta_0$, $\sigma_0 \leftarrow \sigma_0$ and $\delta \leftarrow \delta^*$, assuming $\rho$ to be sufficiently small.
We obtain the constant $\theta = \theta(A, E, D, \theta_0, \sigma_0, \delta^*) \in (0, \theta_0)$ and a subset $S \subset B(x_0, 0, D)$, which we will denote henceforth by $S'$, that satisfies assertions \ref{(a)-Lem:LLcontrolledspeed}, \ref{(b)-Lem:LLcontrolledspeed} of that lemma.
The subset $S \subset B(x_0, 0, D)$, whose existence is claimed in this lemma, will arise as a subset of $S'$.

Our proof uses again $\LL$-geometry and the terminology recalled in subsection~\ref{subsec:LLgeometry}.
Define
\[ U := \LL\exp_{\theta}^{-1} (S') \cap \mathcal{D}^\LL_{\theta}. \]
Then
\[ \LL\exp_\theta (U) = S' \cap \mathcal{G}^\LL_{\theta}. \]
For simplicity, we set
\[ f(x) := (\rrm (x,0))^{-1.5} \chi_{B(x_0, 0, D+1)}(x), \]
where $\chi_{B(x_0, 0, D+1)}$ is the characteristic function of $B(x_0, 0, D+1)$.
Note that by the Claim, we have for all $\tau \in [0, \theta]$ (assuming $\rho < 1/10$)
\[ \int_M f(x) dg_{-\tau} (x) \leq e^{0.1 \tau} \int_M f(x) dg_0 (x) < 2 E^*. \]
So by the transformation formula, we have for all $\tau \in (0, \theta]$
\[ \int_U f \big( \LL\exp_{\tau} (v) \big) J^\LL (v,\tau) dv = \int_{\LL \exp_{x_0, \tau} (U)} f(x) dg_{-\tau} (x)  < 2E^* \]

Using assertion \ref{(b3)-Lem:LLcontrolledspeed} of Lemma \ref{Lem:LLcontrolledspeed} and the estimate $\delta^* < 1$, we have for all $v \in U$ and $\tau \in (0, \theta]$
\begin{equation} \label{eq:ltautheta}
 l (\LL\exp_\tau (v), -\tau) > l (\LL\exp_\tau(v), -\theta) - 1. 
\end{equation}
Moreover, by the monotonicity of $\tau^{-n/2} e^{-l(v, -\tau)} J(v, \tau)$ we have
\begin{equation} \label{eq:Perelmanmonotthetatau}
 \theta^{-n/2} e^{- l( \LL\exp_{\tau}(v), -\theta)} J^\LL (v, \theta) \leq \tau^{-n/2} e^{-l(\LL\exp_\tau(v),-\tau)} J^\LL (v,\tau). 
\end{equation}
Combining (\ref{eq:ltautheta}) with (\ref{eq:Perelmanmonotthetatau}) yields that for all $\tau \in [\alpha_0 \theta, \theta]$
\[  J^\LL (v, \theta) \leq e \Big( \frac{\theta}{\alpha_0 \theta} \Big)^{n/2} J^\LL (v, \tau) \leq e \alpha_0^{-n/2} J^\LL (v,\tau). \]
So for all $\tau \in [\alpha_0 \theta, \theta]$ we have
\begin{equation} \label{eq:integrallargetau}
  \int_U f \big( \LL\exp_{\tau} (v) \big) J^\LL (v,\theta) dv \leq e \alpha_0^{-n/2} \int_U f \big( \LL\exp_{\tau} (v) \big) J^\LL (v,\tau) dv \leq 2e \alpha_0^{-n/2} E^*.
\end{equation}

We now estimate the left-hand side of (\ref{eq:integrallargetau}) for all $\tau \in (0, \alpha_0 \theta]$.
Observe that by Lemma \ref{Lem:LLnotleavparnbhd}, we have $\LL \exp_{\tau} (v) \in B(x_0, 0, \sigma_0/10)$ for all $\tau \in (0, \alpha_0 \theta]$.
Since $\rrm(\cdot, 0)$ is $1$-Lipschitz with respect to $g_0$, we have the bound $f < (\sigma_0/2)^{-1.5}$ on $B(x_0, 0, \sigma_0/10)$.
So for all $\tau \in (0, \tau_0]$
\begin{multline} \label{eq:integralsmalltau}
 \int_U f \big( \LL\exp_{\tau} (v) \big) J^\LL (v,\theta) dv \leq (\sigma_0/2)^{-1.5} \int_U J^\LL (v,\tau) dv \\
  \leq (\sigma_0/2)^{-1.5} |B(x_0, 0, D)|_{-\theta} \leq 2(\sigma_0/2)^{-1.5} |B(x_0, 0, D)| \leq 2(\sigma_0/2)^{-1.5}E^*.
\end{multline}
Integrating (\ref{eq:integralsmalltau}) from $0$ to $\alpha_0\theta$ and (\ref{eq:integrallargetau}) from $\alpha_0 \theta$ to $\theta$ yields
\begin{equation} \label{eq:Estarstartintegral}
\int_0^\theta \int_U f \big( \LL\exp_{\tau} (v) \big) J^\LL (v,\theta) dv d\tau \leq \big(2e \alpha_0^{-n/2} E^* + 2(\sigma_0/2)^{-1.5}E^*\big) \theta = E^{**} \theta. 
\end{equation}
Define $h : U \to [0, \infty)$ by
\[ h(v) := \int_0^\theta f \big( \LL\exp_{\tau} (v) \big) d\tau. \]
Then, by Fubini's Theorem and (\ref{eq:Estarstartintegral}),
\[ \int_U h(v) J(v,\theta) dv d\tau \leq E^{**} \theta. \]
Let
\[ W := \{ v \in U \;\; : \;\; h(v) < C \theta \} \textQq{and} S := \LL\exp_{\theta} (W). \]
Then $S \subset \LL\exp_{\theta} (U) \subset S'$.
We claim that $S$ satisfies assertion \ref{(b)-Lem:firstsetS}.
For any $y = \LL\exp_\theta (v)$, $v \in W$ we choose the minimizing $\LL$-geodesic $\gamma (\tau) := \LL \exp_\tau (v)$.
As $y = \gamma (\theta) \in \mathcal{G}^\LL_\theta$, this $\LL$-geodesic is the only minimizing $\LL$-geodesic between $x_0$ and $y$.
So assertions \ref{(b1)-Lem:firstsetS}--\ref{(b3)-Lem:firstsetS} follow immediately from assertions \ref{(b1)-Lem:LLcontrolledspeed}, \ref{(b2)-Lem:LLcontrolledspeed} and \ref{(b4)-Lem:LLcontrolledspeed} of Lemma \ref{Lem:LLcontrolledspeed}.
For assertion \ref{(b4)-Lem:firstsetS} observe that
\[ \int_0^\theta \big( \rrm (\gamma(\tau), 0) \big)^{-1.5} d\tau = \int_0^\theta f (\LL\exp_\tau (v)) d\tau = h(v) < C \theta. \]
To see assertion \ref{(a)-Lem:firstsetS}, observe that
\begin{multline*}
|S' \setminus S|_0 = |(S' \cap \mathcal{G}^\LL_\theta) \setminus S|_0
\leq 2|\LL\exp_\theta (U) \setminus \LL\exp_\theta (W) |_{-\theta} \\
= 2\int_{U \setminus W} J(v, \theta) dv
\leq 2\int_{U \setminus W} (C \theta)^{-1} h(v) J(v, \theta) dv \\
\leq 2(C \theta)^{-1} E^{**} \theta
= 2C^{-1} E^{**}
= \delta/2.
\end{multline*}
So
\[ |B(x_0, 0,D) \setminus S|_0 \leq |B(x_0, 0, D) \setminus S'|_0 + |S' \setminus S|_0 < \delta^* + \delta/2 < \delta. \]
This finishes the proof.
\end{proof}

Finally, we can improve the curvature bound in Lemma \ref{Lem:firstsetS}\ref{(b4)-Lem:firstsetS} from an integral bound to a pointwise bound.
This bound will also enable us to show that the $\LL$-geodesics are almost minimizing with respect to the time-$0$ metric.

\begin{Proposition} \label{Prop:shortcurveuniformsigma}
For any $A, E, D < \infty$ and $\sigma_0 , \delta, \eta > 0$ there are constants $\sigma = \sigma (A, E, D, \sigma_0, \delta)$ and $\rho = \rho(A, E, D,\sigma_0, \delta, \eta) > 0$ such that the following holds:

Let $(M, (g_t)_{t \in [-2,0]} )$ be a Ricci flow on a compact, $n$-dimensional manifold $M$ and $x_0 \in M$ such that
\begin{enumerate}[label=(\roman*)]
\item \label{(i)-Prop:shortcurveuniformsigma} $\nu [ g_{-2}, 4] \geq - A$.
\item \label{(ii)-Prop:shortcurveuniformsigma} $|R| \leq \rho$ on $M \times [-2,0]$.
\item For all $(x,t) \in M \times [-1,0]$ and $0 < r, s < 1$ we have
\[ | \{ \rrm (\cdot, t) < sr \} \cap B(x,t, r) |_t \leq E s^{2} r^n. \]
\item \label{(iii)-Prop:shortcurveuniformsigma} $\rrm (x_0, 0) \geq \sigma_0$.
\end{enumerate}
Then there is an open subset $S \subset B(x_0, 0, D)$ such that the following holds:
\begin{enumerate}[label=(\alph*)]
\item \label{(a)-Prop:shortcurveuniformsigma} $|B(x_0, 0, D) \setminus S |_0 < \delta$.
\item \label{(b)-Prop:shortcurveuniformsigma} For any $y \in S$ there is a curve $\gamma : [0, 1] \to M$ between $x_0$ and $y$, $\gamma(0) = x_0$, $\gamma(1) = y$ such that the following holds:
\begin{enumerate}[label=(b\arabic*)]
\item \label{(b1)-Prop:shortcurveuniformsigma} Its time-$0$ length satisfies
 \[ \length_0 (\gamma) < d_0 (x_0,y) + \eta. \]
 \item \label{(b2)-Prop:shortcurveuniformsigma} We have
 \[  \rrm (\gamma(s), 0)  > \sigma \qquad \text{for all} \qquad s \in [0, 1]. \]
\end{enumerate}
\end{enumerate}
\end{Proposition}

Note that the constant $\sigma$ does not depend on the choice of $\eta$.

\begin{proof}
The idea of the proof will be to use the subset $S$ and the $\LL$-geodesics $\gamma$ from Lemma \ref{Lem:firstsetS}.
We will show that the function $\rrm (\gamma(\tau),0)$ does not oscillate too much and use assertion \ref{(b4)-Lem:firstsetS} in Lemma \ref{Lem:firstsetS} to establish a lower bound for $\rrm (\gamma(\tau),0)$.
Using this bound and assertion \ref{(b3)-Lem:firstsetS} of Lemma \ref{Lem:firstsetS}, we can then derive an upper bound on the time-$0$ length of $\gamma$.

We will first choose the relevant constants.
Let $\eps = \eps(A) > 0$ be the constant from Proposition \ref{Prop:Pseudoloc}.
Let $\alpha_0 = \alpha_0 (A, D, \sigma_0) < 1/2$ be the constant from Lemma~\ref{Lem:LLnotleavparnbhd} and set
\begin{equation} \label{eq:Defofa}
 a := \min \bigg\{  \alpha_0^{1/2} \big( 2(  D^2   + 1 )\big)^{-1/2}, \frac1{10} \bigg\}. 
\end{equation}
Note that $a$ depends only on $A, E, D, \sigma_0, \delta$.
Next, let $C = C(A, E, D, \sigma_0, \delta)$ be the constant from Lemma \ref{Lem:firstsetS} and choose $ \sigma> 0$ small enough such that
\begin{equation} \label{eq:choiceofsigma}
 \tfrac12 a (2\sigma)^{-0.5}  > C   , \quad a\sigma < \alpha_0 /2, \quad \sigma < \sigma_0 / 2 \textQq{and} \sigma <1/10. 
\end{equation}
Note that all constants occurring in (\ref{eq:choiceofsigma}) only depend on $A, E, D, \sigma_0, \delta$.
So $\sigma$ can be chosen depending only on these constants as well.

Next, choose and fix a constant $0 < \eta^* < 1$, depending only on $D, \eta$, such that
\[
(1+ \eta_*) \big( d^2 +  \eta^* \big)^{1/2} < d + \eta \textQQqq{for all} d \in [0,D).
\]

\begin{Claim}
There is a constant $\rho_1 = \rho_1 (A, E, D, \sigma_0, \delta, \eta ) > 0$ such that if $\rho < \rho_1$, then for all $x \in M$ with $\rrm (x, 0) \geq \sigma$, all $v \in T_x M$ and all $t \in [-(\eps \sigma)^2, 0]$ we have the distortion estimate
\[ |v|_{0} \leq  (1+ \eta^*) |v|_{t}  . \]
\end{Claim}

\begin{proof}
This fact follows from Propositions \ref{Prop:Pseudoloc} and \ref{Prop:Ricsmall}.
The constant $\rho_1$ can be chosen depending only on $\eps$, $\sigma$, which in turn only depend on $A$, $E$, $D$, $\sigma_0$, $\delta$, $\eta$.
\end{proof}

Assume in the following that $\rho < \rho_1$ and that $\rho < \rho_{\ref{Lem:firstsetS}} (A, E, D, \theta_0, (\eps \sigma)^2, \delta, \eta^*)$, where $\rho_{\ref{Lem:firstsetS}}$ is the constant from Lemma \ref{Lem:firstsetS}.
Let furthermore
\[ 0 < \theta = \theta_{\ref{Lem:firstsetS}} (A, E, D, (\eps \sigma)^2, \sigma_0, \delta, \eta^*) < (\eps \sigma)^2 \]
be the constant from Lemma \ref{Lem:firstsetS}.

We now apply Lemma \ref{Lem:firstsetS} for $A \leftarrow A$, $E \leftarrow E$, $D \leftarrow D$, $\theta_0 \leftarrow (\eps \sigma)^2$, $\sigma_0 \leftarrow \sigma_0$, $\delta \leftarrow \delta$ and $\eta \leftarrow \eta^*$.
We obtain a subset $S \subset B(x_0, 0, D)$ that satisfies assertions \ref{(a)-Lem:firstsetS}, \ref{(b)-Lem:firstsetS} of Lemma \ref{Lem:firstsetS}.
We claim that $S$ satisfies assertions \ref{(a)-Prop:shortcurveuniformsigma}, \ref{(b)-Prop:shortcurveuniformsigma} of this proposition as well.
Obviously, assertion \ref{(a)-Prop:shortcurveuniformsigma} holds.

It remains to check assertion \ref{(b)-Prop:shortcurveuniformsigma}.
Let $y \in S$ and let $\gamma : [0, \theta] \to M$ be the $\LL$-geodesic from Lemma~\ref{Lem:firstsetS}\ref{(b)-Lem:firstsetS}.
In the following, we will show that $\gamma$ satisfies properties \ref{(b1)-Prop:shortcurveuniformsigma}, \ref{(b2)-Prop:shortcurveuniformsigma}, after reparameterization.
By this we mean that $\length_0 (\gamma) < d_0(x_0 , y) + \eta$ and
\begin{equation} \label{eq:rrmmboundedfrombelow}
 \rrm \big( \gamma(\tau),0 \big) \geq \sigma  \textQQqq{for all} \tau \in [0, \theta].
\end{equation}
An important tool in our analysis will be Lemma \ref{Lem:firstsetS}\ref{(b4)-Lem:firstsetS}, which states that 
\begin{equation} \label{eq:sumNsqboundedrecall}
  \int_0^\theta \big( \rrm (\gamma(\tau), 0) \big)^{-1.5} d\tau < C \theta. 
\end{equation}

We first show that the bound (\ref{eq:rrmmboundedfrombelow}) holds whenever $\tau \in [0, \alpha_0 \theta]$.
Indeed, for all such $\tau$, we have $d_0 (x_0, \gamma(\tau)) < \sigma_0 / 10$ by Lemma \ref{Lem:LLnotleavparnbhd} and thus, by the fact that $\rrm$ is $1$-Lipschitz with respect to $g_0$, we have $\rrm (\gamma(\tau), 0) > \sigma_0 / 2 > \sigma$.

Next, we choose $\tau_1 = \alpha_1 \theta$ maximal such that 
\[ \rrm (\gamma(\tau), 0 ) \geq \sigma \textQQqq{for all} \tau \in [0, \tau_1]. \]
Then $\alpha_1 = 1$ or $\alpha_1 < 1$ and we have
\begin{equation} \label{eq:sigmaatalpha1}
 \rrm ( \gamma(\alpha_1 \theta), 0) = \sigma. 
\end{equation}
Moreover, by our previous discussion we have
\begin{equation} \label{eq:alph1biggeralph0}
 \alpha_1 > \alpha_0 
\end{equation}
We will find later that we indeed have $\alpha_1 = 1$, hence establishing assertion \ref{(b2)-Prop:shortcurveuniformsigma}.

Using the Claim, we find that
\[
 |\gamma'(\tau)|_0 \leq (1+\eta^*) |\gamma'(\tau)|_{-\tau} \textQQqq{for all} \tau \in [0, \tau_1].
\]
It follows that for any $0 \leq \tau'  \leq \alpha_1 \theta$ we have, using Lemma~\ref{Lem:firstsetS}\ref{(b3)-Lem:firstsetS},
\begin{multline} \label{eq:distancelinearu1u2}
 d_0 (\gamma(\tau'), \gamma(\alpha_1 \theta)) 
 \leq  \length_0 (\gamma |_{[\tau', \alpha_1 \theta]} ) 
 \leq  (1+\eta^*) \int_{\tau'}^{\alpha_1 \theta} |\gamma'(\tau)|_{-\tau} d\tau  \\
  = (1+\eta^*) \int_{\tau'}^{\alpha_1 \theta} \bigg( \frac{\tau |\gamma'(\tau)|^2_{-\tau}}{\tau} \bigg)^{1/2}  d\tau \\
  \leq (1+\eta^*) \bigg( \frac{d^2_0(x_0, y) + \eta^*}{4 \theta}  \bigg)^{1/2} \int_{\tau'}^{\alpha_1 \theta}   \tau^{-1/2} d\tau .
\end{multline} 
Assume that $\alpha_1 < 1$.
We now apply (\ref{eq:distancelinearu1u2}) for $\tau'  \in [(\alpha_1 - a \sigma) \theta, \tau_1]$ and note that by (\ref{eq:alph1biggeralph0}) and (\ref{eq:Defofa})
 \[ (\alpha_1 - a \sigma) \theta > (\alpha_0 - a \sigma) \theta > \tfrac12 \alpha_0 \theta. \]
We then obtain (using (\ref{eq:Defofa}), (\ref{eq:distancelinearu1u2}) and the crude estimates $\eta^* < 1$, $\theta \leq 1$)
\begin{multline*}
 d_0 (\gamma(\tau'), \gamma(\tau_1)) \leq 2 \bigg( \frac{D^2 + 1}{4 \theta}  \bigg)^{1/2} \int_{\tau'}^{\alpha_1 \theta}   \tau^{-1/2} d\tau \\
  \leq 2 \bigg( \frac{D^2 + 1}{4 \theta}  \bigg)^{1/2} \cdot a \sigma \theta \cdot \big( \tfrac12 \alpha_0 \theta \big)^{-1/2} \leq \alpha_0^{-1/2} \big( 2( D^2 +1 ) \big)^{1/2} a \cdot \sigma \leq \sigma.
\end{multline*}
So by the $1$-Lipschitz continuity of $\rrm (\cdot, 0 )$ and (\ref{eq:sigmaatalpha1}), we get that
\begin{equation} \label{eq:2sigmarrm}
 \rrm ( \gamma(\tau), 0 ) < \sigma + \sigma  = 2 \sigma \textQQqq{for all} \tau \in [(\alpha_1 - a\sigma)\theta, \alpha_1 \theta ]. 
\end{equation}
Using (\ref{eq:sumNsqboundedrecall}) and (\ref{eq:2sigmarrm}), it follows that
\[
 \tfrac12 a (2\sigma)^{-0.5} \theta = a \sigma \theta \cdot (2\sigma)^{-1.5} <   \int_0^\theta \big( \rrm (\gamma(\tau), 0) \big)^{-1.5} d\tau < C \theta.
\]
This inequality contradicts (\ref{eq:choiceofsigma}).
So we indeed have $\alpha_1 = 1$ and $\tau_1 = \theta$, which proves assertion \ref{(b2)-Prop:shortcurveuniformsigma}.

To see assertion \ref{(b1)-Prop:shortcurveuniformsigma}, we apply (\ref{eq:distancelinearu1u2}) again for $\tau' = 0$ and obtain (using $\theta < 1$ and $D > 1$)
\begin{multline*}
\length_0 (\gamma |_{[0, \theta]} )  \leq (1+\eta^*) \bigg( \frac{d^2_0(x_0, y) + \eta^*}{4 \theta}  \bigg)^{1/2} \cdot 2 \theta^{1/2} \\
\leq (1+\eta^*) \big( d^2_0(x_0, y) +  \eta^* \big)^{1/2} < d_0 (x_0, y) + \eta.
\end{multline*}
This shows assertion \ref{(b1)-Prop:shortcurveuniformsigma}.
\end{proof}

If we don't assume that the scalar curvature is small, we obtain a similar result as Proposition \ref{Prop:shortcurveuniformsigma}, but this time $\sigma$ depends on $\delta$.

\begin{Corollary} \label{Cor:shortcurveboundedrrm}
For any $A, E, D < \infty$ and $ \sigma_0, \delta > 0$ there is a constant $\sigma = \sigma (A, E, D, \sigma_0, \delta) > 0$ such that the following holds:

Let $(M, (g_t)_{t \in [-2,0]} )$ be a Ricci flow on a compact, $n$-dimensional manifold $M$ such that
\begin{enumerate}[label=(\roman*)]
\item \label{(i)-Cor:shortcurveboundedrrm} $\nu [ g_{-2}, 4] \geq - A$.
\item \label{(ii)-Cor:shortcurveboundedrrm} $|R| \leq 1$ on $M \times [-2,0]$.
\item \label{(iii)-Cor:shortcurveboundedrrm} For all $(x,t) \in M \times [-1,0]$ and $0 < r, s < 1$ we have
\[ | \{ \rrm (\cdot, t) < sr \} \cap B(x,t, r) |_t \leq E s^{2} r^n. \]
\end{enumerate}

Then for any $x, y \in M$ with $d_0 (x,y) < D$ and
\[ \rrm (x, 0), \; \rrm (y, 0) \geq \sigma_0 \]
there is a smooth curve $\gamma : [0, 1] \to M$ between $x$ and $y$, $\gamma(0) = x$, $\gamma(1) = y$ whose time-$0$ length satisfies
\[ \length_0 (\gamma) < d_0 (x_0,y) + \delta \]
and such that
\[  \rrm (\gamma(s), 0)  > \sigma \qquad \text{for all} \qquad s \in [0, 1]. \]
\end{Corollary}

\begin{proof}
The idea of the proof is to apply Proposition \ref{Prop:shortcurveuniformsigma} several times at small scales.
We therefore first derive the following claim.

\begin{Claim}
For any $\delta^*, \sigma_0^* > 0$ we can find constants $\sigma^* = \sigma^* (A, E, \sigma_0^*, \delta^*), r_0^* = r_0^* (A, \linebreak[1] E, \linebreak[1] \sigma^*_0, \linebreak[1] \delta^*) > 0$ such that the following holds:

For any $0 < r \leq r_0^*$ and any $z_0 \in M$ with $\rrm (z_0, 0) > \sigma_0^* r$ there is an open subset $S_{z_0, r} \subset B(z_0, 0, r)$ such that
\[ |B(z_0, 0, r) \setminus S_{z_0, r} |_0 < \delta^* r^n \]
and such that for any $z_1 \in S_{z_0, r}$ there is a smooth curve $\gamma^* : [0,1] \to M$ between $z_0$ and $z_1$, $\gamma(0) = z_0$, $\gamma (1) = z_1$ such that
\[ \length_0 (\gamma) < d_0 (z_0, z_1) + \delta^* r \]
and such that
\[ \rrm (\gamma(s), 0) > \sigma^* r \textQQqq{for all} s \in [0,1]. \]
\end{Claim}

\begin{proof}
Using the constants $\sigma$ and $\rho$ from Proposition \ref{Prop:shortcurveuniformsigma}, we define $\sigma^* (A, \linebreak[1] E, \linebreak[1] \sigma_0^*, \linebreak[1] \delta^*) := \sigma (A, E, 1, \sigma_0^*, \delta^*)$ and $r_0^* (A, E, \sigma_0^*, \delta^*) = \rho^{1/2} (A, E, 1, \sigma_0^*, \delta^*, \delta^*)$.
The claim then follows from Proposition \ref{Prop:shortcurveuniformsigma} for $D \leftarrow 1$ after rescaling by $r^{-1}$. 
\end{proof}

Before we continue with the proof, let us first choose all the constants.
Let 
\[ \nu := \min \bigg\{ \frac{\delta}{8D}, \frac1{10} \bigg\}. \]
Next, use Proposition \ref{Prop:VolumeBound} to find a $\kappa = \kappa (A) > 0$ such that $| B(z,0,r) |_0 \geq \kappa r^n$ for all $z \in M$ and $0 \leq r < 1$.
Choose $\sigma^*_0 > 0$ such that
\begin{equation} \label{eq:Esigmastar}
  E \Big( \frac{\sigma^*_0}{\nu} \Big)^2 < \frac12 \kappa  \textQQqq{and} \sigma^*_0 < \sigma_0
\end{equation}
and choose
\begin{equation} \label{eq:deltastart12kappa}
 \delta^* := \min \Big\{ \frac12 \kappa \nu^n, \nu \Big\}. 
\end{equation}
Choose $\sigma^* := \sigma^* (A, E, \sigma^*_0)$ and $r^*_0 := r^*_0 (A, E, \sigma^*_0, \delta^*)$ according to the Claim. 
Lastly choose $N \in \IN$ large enough such that
\begin{equation} \label{eq:2DN4DN}
\frac{2D}N < r^*_0 \textQQqq{and} \frac{4D}{N} \nu < \sigma_0. 
\end{equation}

Let $\ov\gamma : [0, 1] \to M$ be a time-$0$ minimizing geodesic between $x$ and $y$.
and choose
\[ z_j := \ov\gamma \Big( \frac{j}{N} \Big) \qquad \text{for} \qquad j = 0, \ldots, N. \]
Note that then
\[ d_0 (z_{j-1}, z_j) = \frac{d_0 (x,y)}{N} =: \frac12 r \qquad \text{for all} \qquad j = 1, \ldots, N, \]
where $r < \frac{2D}N < r^*_0$.
We will now inductively choose points $z'_0, \ldots, z'_{N} \in M$ and smooth curves $\gamma_1, \ldots, \gamma_{N} : [0,1] \to M$ such that $\gamma_j (0) = z'_{j-1}$ and $\gamma_j(1) = z'_j$ with the following properties:
\begin{subequations}
\begin{alignat}{1}
 z'_0 = z_0 = x \qquad& \label{eq:zp0z0x} \displaybreak[1] \\
 d_0 (z'_j, z_j ) < \nu r \qquad &\text{for all} \qquad j = 0, \ldots, N \label{eq:zpjzj} \displaybreak[1] \\
 \gamma_j (0) = z'_{j-1} \textQq{and} \gamma_j(1) = z'_j \qquad &\text{for all} \qquad j = 1, \ldots, N \displaybreak[1] \\
  \length_0 (\gamma_j) \leq  d_0 (z'_{j-1}, z'_j) + \delta^* r \qquad &\text{for all} \qquad j = 1, \ldots N \label{eq:dzpzpzz} \displaybreak[1] \\
 \rrm (z'_j) > \sigma^*_0 r  \qquad &\text{for all} \qquad j = 0, \ldots, N \displaybreak[1] \label{eq:zpje} \\
  \rrm (\gamma_j (s), 0) > \sigma^* r \qquad &\text{for all} \qquad s \in [0,1] \notag \\
  & \textQQq{and} j = 1, \ldots, N \label{eq:zpjf}
\end{alignat}
\end{subequations}
Let us first choose $z'_0 = x$.
Then (\ref{eq:zp0z0x}), (\ref{eq:zpjzj}) obviously hold for $j = 0$.
Next, assume that $j \in \{ 1, \ldots, N-1 \}$ and that $z'_0, \ldots, z'_{j-1}$ and $\gamma_1, \ldots, \gamma_{j-1}$ have already been constructed such that (\ref{eq:zp0z0x})--(\ref{eq:zpjf}) hold.
We will now construct $z'_j$.

For this apply the claim with our choices of $\delta^*, \sigma^*_0$ for $z_0 \leftarrow z'_{j-1}$ and $r \leftarrow r$.
This is possible due to (\ref{eq:zpje}).
We obtain a subset $S_{z'_{j-1}, r} \subset B(z'_{j-1}, r)$ such that
\[ | B(z'_{j-1}, 0, r) \setminus S_{z'_{j-1}, r}  |_0 < \delta r^n  \] 
We now claim that
\begin{equation} \label{eq:intersectionof3}
 S_{z'_{j-1},r} \cap B(z_j, 0, \nu r) \cap \{ \rrm (\cdot, 0) \geq \sigma^*_0 r \} \neq \emptyset. 
\end{equation}
To see this, note that $B(z_j, 0, \nu r) \subset B(z'_{j-1}, 0, r)$, because by (\ref{eq:zpjzj})
\[ d_0 (z'_{j-1}, z_j) \leq d_0 (z'_{j-1}, z_{j-1}) + d_0 ( z_{j-1}, z_j) < \nu r + \frac12 r \leq (1-\nu) r. \] 
So, if (\ref{eq:intersectionof3}) was false, then
\[ \{ \rrm (\cdot, 0) \geq \sigma_0^* r \} \cap B(z_j, 0, \nu r) \subset B(z'_{j-1}, 0, r) \setminus S_{z'_{j-1}, r}. \]
This would imply
\[ \big| \{ \rrm (\cdot, 0) \geq \sigma_0^* r \} \cap B(z_j, 0, \nu r) \big|_0 < \delta^* r^n \]
So
\[  \big| B(z_j, 0, \nu r) \big|_0 - \big| \{ \rrm (\cdot, 0) < \sigma_0^* r \} \cap B(z_j, 0, \nu r) \big|_0 < \delta^* r^n. \]
Using assumption \ref{(iii)-Cor:shortcurveboundedrrm}, this implies
\[ \kappa (\nu r)^n - E \Big( \frac{\sigma_0^*}{\nu} \Big)^2 (\nu r)^n < \delta^* r^n. \]
So
\[ \kappa \nu^n - E \Big( \frac{\sigma_0^*}{\nu} \Big)^2 \nu^n < \delta^*, \]
which contradicts (\ref{eq:Esigmastar}) and (\ref{eq:deltastart12kappa}).
So (\ref{eq:intersectionof3}) is indeed true and we can pick a $z'_j \in S_{z'_{j-1},r} \cap B(z_j, 0, \nu r)$ such that $\rrm (z'_j, 0) \geq \sigma^*_0 r$ as well as a curve $\gamma_j : [0,1] \to M$ such that (\ref{eq:zpjzj})--(\ref{eq:zpjf}) are satisfied.
This finishes the induction and shows that we can choose $z'_0, \ldots, z'_{N}$ and $\gamma_1, \ldots, \gamma_{N}$ such that (\ref{eq:zp0z0x})--(\ref{eq:zpjf}) hold.

Lastly, we choose a minimizing geodesic $\gamma_{N+1} : [0,1] \to M$ between $z'_N$ and $y$.
Note since $\rrm (y, 0) > \sigma_0 \geq 2 \nu r$ and since $\rrm (\cdot, 0)$ is $1$-Lipschitz with respect to $g_0$, we have, using (\ref{eq:2DN4DN}),
\[ \gamma_{N+1} (s) > \sigma_0 - \nu r  = \sigma_0 - \frac{2D}{N} \nu \geq \frac12 \sigma_0  \textQQqq{for all} s \in [0,1] \]
Joining the curves $\gamma_1, \ldots, \gamma_{N+1}$, and smoothing yields a smooth curve $\gamma : [0,1] \to M$ between $x$ and $y$ such that
\[ \rrm (\gamma(s), 0) > \min \Big\{ \sigma^*_0 r, \frac12 \sigma_0 \Big\} \textQQqq{for all} s \in [0,1]. \]
and such that
\begin{align*}
 \length_0 (\gamma) &< \big( d_0(z'_0, z'_1) + \delta^* r \big) + \ldots + \big( d_0 (z'_{N-1}, z'_N) + \delta^* r \big) + \nu r \\
 &\leq \big( d_0 (z_0, z_1) + 2\nu r + \delta^* r \big)  + \big( d_0 (z_1 + z_2 ) + 2 \nu r + \delta^* r ) + \ldots \\
 &\qquad\qquad\qquad\qquad\qquad + \big( d_0 (z_{N-1}, z_N) + 2 \nu r + \delta^* r \big) + \nu r \\
 &\leq d_0 (x,y) + (3N+1) \nu r \\
 & = d_0 (x,y) + (3N+1) \nu \frac{2D}N \\
 & \leq d_0 (x,y) + 8 D \nu \leq d_0 (x,y) + \delta.
\end{align*}
This finishes the proof.
\end{proof}

\subsection{Proof of the compactness result}
Using Proposition \ref{Prop:shortcurveuniformsigma} and Corollary~\ref{Cor:shortcurveboundedrrm}, we are now able to prove the main results of this section, Propositions~\ref{Prop:convtosingspaceLpasspt} and \ref{Prop:effectivelimitXX}.

\begin{proof}[Proof of Proposition \ref{Prop:convtosingspaceLpasspt}]
The proposition will essentially follow from \cite[Theorem 1.2]{Bamler-CGN}.
In order to apply this theorem, we have to verify that the sequence $(M_i, g^i_0, q_i)$ of pointed Riemannian manifolds satisfies properties (A)--(F), which are mentioned in \cite[subsec 1.2]{Bamler-CGN}.
Before doing so, we mention that in this paper we use a slightly different definition of the curvature radius $\rrm$ than in \cite{Bamler-CGN}.
The definition of the curvature radius in this paper (see Definition~\ref{Def:curvradius}) does not involve curvature derivatives, while the definition of $\rrm$ or $\td{r}_{\Rm}$ in \cite{Bamler-CGN} does.
However, this difference does not create any issues, since in our setting these three radii are comparable to one another:
By definition, the curvature radius $\rrm$ in our paper is bounded from below by $\rrm \geq \td{r}_{\Rm}$ from \cite{Bamler-CGN} and by Proposition \ref{Prop:Pseudoloc} and Shi's estimates, it is bounded from above by $C \td{r}_{\Rm} \leq C \rrm$ from \cite{Bamler-CGN}, where $C = C(A) < \infty$ only depends on $A$.

Let us now verify the properties from \cite[subsec 1.2]{Bamler-CGN}:
Property (A) is a direct consequence of Proposition \ref{Prop:VolumeBound}.
Note that the constant $A$ appearing in property (A) can be chosen depending only on the constant $A$ from assumption \ref{(i)-Prop:convtosingspaceLpasspt} of this proposition.
Also the constant $T$ appearing in this property can be chosen to be any $T < \limsup_{i \to \infty} T_i = T_\infty$.
Next, property (B) is a direct consequence of assumption \ref{(iii)-Prop:convtosingspaceLpasspt}.
Here, the constant $\mathbf{p}_0$ appearing in this property has to be chosen slightly smaller than the constant $\mathbf{p}_0$ in assumption \ref{(iii)-Prop:convtosingspaceLpasspt}.
Property (C) is a direct consequence of Corollary~\ref{Cor:shortcurveboundedrrm}.
Property (D) follows  from assumptions \ref{(i)-Prop:convtosingspaceLpasspt} and \ref{(ii)-Prop:convtosingspaceLpasspt} of this proposition, where the constant $A$ of this property depends on the constant $A$ of our proposition.
Property (E) follows from Proposition \ref{Prop:shortcurveuniformsigma} in the case in which $\rho_i \to 0$.
Finally, property (F) follows from \cite[Theorem 1.3]{Bamler-Zhang-Part1}.
The constants $A$ and $T$ appearing in this property can be chosen depending on the constants $A$ and $T_\infty$ of this proposition, respectively.

Now \cite[Theorem 1.2]{Bamler-CGN} implies that the pointed Riemannian manifolds $(M_i, g^i_0, q_i)$ converge to a pointed singular space $(\XX, q_\infty)$ with singularities of codimension $\mathbf{p}_0$.
Note here, that in the context of \cite{Bamler-CGN}, the Riemannian metric $g$ on the regular part $\RR$ of the singular space $\XX = (X, d, \RR, g)$ as well as the convergence of the metric were allowed to have regularity $C^3$ and in the present paper, we require the regularity to be $C^\infty$.
This extra regularity follows easily from Proposition \ref{Prop:Pseudoloc} and Shi's estimates.

The first part of assertion \ref{(b)-Prop:convtosingspaceLpasspt} of this proposition, namely the fact that $\Ric = 0$, follows from Propositions \ref{Prop:Pseudoloc}, \ref{Prop:Ricsmall} and the fact that $\rho_i \to 0$.
The statement about the mildness in assertion \ref{(b)-Prop:convtosingspaceLpasspt} follows from assertion (a) of \cite[Theorem 1.2]{Bamler-CGN} and the fact that property (E) of \cite{Bamler-CGN} holds if $\rho_i \to 0$.
The statement about the tameness in assertion \ref{(b)-Prop:convtosingspaceLpasspt} follows from assertion (b) of \cite[Theorem~1.2]{Bamler-CGN}.

Assertions \ref{(c)-Prop:convtosingspaceLpasspt} and \ref{(e)-Prop:convtosingspaceLpasspt} of this proposition follow from assertions (a) and (c) of \cite[Theorem 3.1]{Bamler-CGN}, which is the more detailed version of \cite[Theorem~1.2]{Bamler-CGN}.
Assertion (b) of \cite[Theorem 3.1]{Bamler-CGN} and our discussion of the different definitions of the curvature radius implies that there is a $C = C(A) < \infty$ such that for all $x \in \RR$
\begin{equation} \label{eq:rrmquasiconvergence}
 C^{-1} \limsup_{i \to \infty} \rrm^{M_i} (\Phi_i (x), 0 ) \leq \rrm^\infty (x) \leq C \liminf_{i \to \infty} \rrm^{M_i} (\Phi_i (x), 0). 
\end{equation}
We will now deduce assertion \ref{(d)-Prop:convtosingspaceLpasspt} of this proposition from this inequality.
For this let $x \in \RR$ and fix some $0 < r < \rrm^\infty (x)$.
Then $B^X (x,r) \subset \RR$ is relatively compact in $\RR$ and $|{\Rm}| < r^{-2}$ on $B^X (x,r)$.
By compactness, we have $B^X (x,r) \subset U_i$ for large $i$.
By lifting curves of length $0 < r' < r$ that start in $\Phi_i (x)$, we can conclude that for any $0 < r' < r$ and large $i$ we have $B^{M_i} (\Phi_i (x), r') \subset \Phi_i (B^X( x,r))$.
So, by the smooth convergence on $\RR$ we have
\begin{equation} \label{eq:rrmgeqrrminfty}
 \liminf_{i \to \infty} \rrm^{M_i} (\Phi_i (x), 0) \geq \rrm^\infty (x). 
\end{equation}
We now show the reverse inequality.
So let $x \in \RR$ and assume that $0 < r \leq \limsup_{i \to \infty} \rrm^{M_i} (\Phi_i (x), 0)$.
We want to show that then $\rrm^\infty (x) \geq r$.
To see this, observe that for any $0 < r' < r$ and $y \in B^X (x, r') \cap \RR$ we have $\lim_{i \to \infty} d^{M_i}_0 (\Phi_i (x), \Phi_i (y) ) = d^X (x,y) < r'$.
So, since the $\rrm^{M_i} (x, \cdot)$ are $1$-Lipschitz with respect to $g^i_0$, we get
\[ \limsup_{i \to \infty} \rrm^{M_i} (\Phi_i(y), 0) \geq \limsup_{i \to \infty} \rrm^{M_i} ( \Phi_i (x), 0) - r' \geq r - r'. \]
Using (\ref{eq:rrmquasiconvergence}), we find that for all $y \in B^X (x,r')$
\[ \rrm^\infty (y) \geq C^{-1} (r - r'). \]
It follows that $B^X (x, r') \subset \RR$ for all $r' < r$ and hence $B^X(x,r')$ is relatively compact in $\RR$ for all $r' < r$.
By smooth convergence we have $|{\Rm}| < r^{-2}$ on $B^X(x,r)$, which shows that $r \leq \rrm^\infty (x)$.
So
\begin{equation} \label{eq:rrmleqrrminfty}
\limsup_{i \to \infty} \rrm^{M_i} (\Phi_i (x), 0) \leq \rrm^\infty (x). 
\end{equation}
Combining (\ref{eq:rrmgeqrrminfty}) and (\ref{eq:rrmleqrrminfty}) yields assertion \ref{(d)-Prop:convtosingspaceLpasspt}.

Finally, assertion \ref{(a)-Prop:convtosingspaceLpasspt} is a consequence of assertion \ref{(d)-Prop:convtosingspaceLpasspt} and assumption \ref{(iii)-Prop:convtosingspaceLpasspt}.
\end{proof}

As a consequence, we obtain Proposition \ref{Prop:effectivelimitXX}.

\begin{proof}[Proof of Proposition \ref{Prop:effectivelimitXX}.]
First observe that by parabolic rescaling and rechoosing the constants $A, E$ it suffices to prove the proposition for $t = 0$.
Fix now $A, E, \eta$ and $\mathbf{p}_0$ and assume that the statement was wrong.
Choose a sequence $\rho_i \to 0$ and consider counterexamples $(M_i, (g^i_t)_{t \in [-2, 0]} )$ of Ricci flows that satisfy assumptions \ref{(i)-Prop:effectivelimitXX}--\ref{(iii)-Prop:effectivelimitXX} for $\rho$ replaced by $\rho_i$ and points $q_i \in M$ such that the conclusion does not hold.
Then we can use Proposition \ref{Prop:convtosingspaceLpasspt} to conclude that the $(M_i, g^i_0, q_i)$ converge to a pointed singular space $(X,d, \RR,g, q_\infty)$ with $\Ric \equiv 0$ on $\RR$ and mild singularities.
Let $(U_i, V_i, \Phi_i)$ be a convergence scheme for this convergence.
We claim that for sufficiently large $i$, the subsets $U = U_i' := U_i \cap B(x_i, \eta^{-1} - \eta / 2)$, $V = V'_i := \Phi_i (U'_i)$ and $\Phi = \Phi'_i := \Phi_i |_{U'_i}$ satisfy assertions \ref{(a)-Prop:effectivelimitXX}--\ref{(g)-Prop:effectivelimitXX}.

To see this, observe that by definition of a convergence scheme assertions \ref{(a)-Prop:effectivelimitXX}, \ref{(b)-Prop:effectivelimitXX} and \ref{(e)-Prop:effectivelimitXX} hold automatically for large $i$.
Assertions \ref{(c)-Prop:effectivelimitXX}, \ref{(d)-Prop:effectivelimitXX} are a direct consequence of Proposition \ref{Prop:convtosingspaceLpasspt}\ref{(e)-Prop:convtosingspaceLpasspt}.
Assertion \ref{(f)-Prop:effectivelimitXX} follows from Proposition \ref{Prop:convtosingspaceLpasspt}\ref{(c)-Prop:convtosingspaceLpasspt} and the fact that due to volume comparison on $\XX$ (see \cite[Proposition 4.1]{Bamler-CGN}) the map $r \mapsto |B^X (x,r) \cap \RR|$ is continuous.
Assertion \ref{(g)-Prop:effectivelimitXX} is a direct consequence of Proposition \ref{Prop:convtosingspaceLpasspt}\ref{(a)-Prop:convtosingspaceLpasspt}.
\end{proof}

\section{$\eps$-regularity Theorem} \label{sec:epsregularity}
\subsection{Statement of the results}
The main result of this section will be an $\eps$-regularity theorem for Ricci flows with small scalar curvature that satisfy an additional a priori $L^p$-curvature bound.
More specifically, we will prove that any ball with almost Euclidean volume has bounded curvature at its center, if the scalar curvature bound is small enough.
An important property of this $\eps$-regularity theorem is that the constants quantifying the curvature bound at the center of this ball and the degree to which the volume of the ball is almost Euclidean (namely $\eps$) are independent of the imposed a priori $L^p$-curvature bound.
Only the scalar curvature bound will depend on this a priori $L^p$-curvature bound.

Based on the $\eps$-regularity theorem and by passing to the limit we will furthermore show that the limiting singular space in Proposition \ref{Prop:convtosingspaceLpasspt} from section \ref{sec:firstconvsing} is $Y$-tame.
Here $Y$ only depends on a lower bound $A$ on Perelman's $\nu$-functional and not on the integral curvature bound $\mathbf{E}_{\mathbf{p}}$.

Let us now state the main results more precisely:

\begin{Proposition}[$\eps$-regularity Theorem] \label{Prop:epsreg}
For any $A, E < \infty$ and $\mathbf{p} > 3$ there are constants $\eps_0 = \eps_0 (A), \sigma_0 = \sigma_0 (A), \rho = \rho (A, E, \mathbf{p}) > 0$ such that the following holds:

Let $(M, (g_t)_{t \in [-2,0]} )$ be a Ricci flow on a compact, $n$-dimensional and orientable manifold $M$ and assume that
\begin{enumerate}[label=(\roman*)]
\item \label{(i)-Prop:epsreg} $\nu [ g_{-2}, 4] \geq - A$.
\item \label{(ii)-Prop:epsreg} $|R| \leq \rho$ on $M \times [-2,0]$.
\item \label{(iii)-Prop:epsreg} For all $(x,t) \in M \times [-1,0]$ and $0 < r, s < 1$ we have
\[ | \{ \rrm (\cdot, t) < sr \} \cap B(x,t, r) |_t \leq E s^{\mathbf{p}} r^n. \]
\end{enumerate}
Then for any $x_0 \in M$ and $0 < r_0 < 1$ for which
\[ |B(x_0, 0, r_0 )|_0 > ( \omega_n - \eps_0) r_0^n \]
we have $\rrm (x_0, 0) > \sigma_0 r_0$.
\end{Proposition}

Note that $\eps_0$ and $\sigma_0$ don't depend on $E$ or $\mathbf{p}$; only $\rho$ does.
This fact will be important for us in section \ref{sec:mainproof}, where a value for $E = E(A)$ will be determined based on $\sigma_0$ such that assumption \ref{(iii)-Prop:epsreg} always holds.

We also remark that a similar result was considered in \cite[Proposition~4.16]{Chen-Wang-II}.

Using Proposition \ref{Prop:epsreg}, we can refine the compactness results of section \ref{sec:firstconvsing}.

\begin{Corollary}[regularity of the limit] \label{Cor:checkregularityasspt}
For any $A < \infty$ there is a $Y = Y(A) < \infty$ such that the following holds:

Assume that we are in the setting of Proposition \ref{Prop:convtosingspaceLpasspt} and assume that $\rho_i \to 0$, $\mathbf{p}_0 > 3$ and that all $M_i$ are orientable.
Let $\XX$ be the limiting singular space and set $T_\infty := \limsup_{i \to \infty} T_i$.

Then $\XX$ is $Y$-regular at scale $\sqrt{T_\infty}$ (in the sense of Definition \ref{Def:Yregularity}).
\end{Corollary}

\subsection{Existence of almost geodesics that stay away from high curvature regions}
The main result of this subsection, Proposition \ref{Prop:almostgeodesicalmostalways}, states that, as long as the scalar curvature bound $\rho$ is chosen small enough, we can find an almost geodesic between almost every pair of points with the following property:
All points on this almost geodesic have bounded curvature and small Ricci curvature at all times of the time-interval $[-1/2,0]$.

In order to show Proposition \ref{Prop:almostgeodesicalmostalways}, we will first show that we can find such almost geodesics on which the curvature is bounded at a single time.

\begin{Lemma} \label{Lem:gammaavoidatt}
For any $A, E, D < \infty$, $\mathbf{p} > 2$ and $\delta > 0$ there are constants $0 < \sigma_* = \sigma_* (A, E, D, \mathbf{p}, \delta) < 1$, $C_* = C_* (A, E, D) < \infty$ and $\rho = \rho (A, E, D, \mathbf{p}, \delta) > 0$ such that the following holds:

Let $(M, (g_t)_{t \in [-2,0]} )$ be a Ricci flow on a compact, $n$-dimensional manifold $M$ with the property that
\begin{enumerate}[label=(\roman*)]
\item \label{(i)-Lem:gammaavoidatt} $\nu [ g_{-2}, 4] \geq - A$.
\item \label{(ii)-Lem:gammaavoidatt} $|R| \leq \rho$ on $M \times [-2,0]$.
\item \label{(iii)-Lem:gammaavoidatt} For all $(x,t) \in M \times [-1,0]$ and $0 < r, s < 1$ we have
\[ | \{ \rrm (\cdot, t) < sr \} \cap B(x,t, r) |_t \leq E s^{\mathbf{p}} r^n. \]
\end{enumerate}
Let $x_0 \in M$ and $t_0 \in [-\frac12,0]$.
Then there is a closed subset
\[ S' \subset B(x_0, t_0, D) \times B(x_0, t_0, D), \] 
a smooth function
\[ l :  \big( (B(x_0, t_0, D) \times B(x_0, t_0, D)) \setminus S' , \qquad (y_1, y_2) \longmapsto l_{y_1, y_2} \]
and a smooth family of curves
\[ \gamma_{y_1, y_2}: [0, l_{y_1, y_2}] \longrightarrow M, \qquad (y_1, y_2) \in  \big( (B(x_0, t_0, D) \times B(x_0, t_0, D) \big) \setminus S' \]
with the following properties:
\begin{enumerate}[label=(\alph*)]
\item \label{(a)-Lem:gammaavoidatt} $|S'|_{t_0} < \delta / 2$ with respect to the product measure $dg_{t_0} \otimes dg_{t_0}$.
\item \label{(b)-Lem:gammaavoidatt} For any  $(y_1, y_2) \in (B(x_0, t_0, D) \times B(x_0, t_0, D)) \setminus S'$ we have $\gamma_{y_1, y_2} (0) = y_1$, $\gamma_{y_1, y_2} (l_{y_1,y_2}) = y_2$ and
\[ 1- \delta <  |\gamma'_{y_1, y_2} (s) |_{t_0} < 1 + \delta \textQQqq{for all} s \in [0,l_{y_1, y_2}]. \]
\item \label{(c)-Lem:gammaavoidatt} For any such pair $(y_1, y_2)$ we have $\rrm (\gamma_{y_1, y_2} (s), t_0) > \sigma_*$ for all $s \in [0,l_{y_1, y_2}]$.
\item \label{(d)-Lem:gammaavoidatt} For any such pair $(y_1, y_2)$ we have
\[ | \length_{t_0} (\gamma_{y_1, y_2}) - d_{t_0} (y_1, y_2) | < \delta. \]
\item \label{(e)-Lem:gammaavoidatt} For any such pair $(y_1, y_2)$ we have
\[ l_{y_1, y_2} > \sigma_*. \]
\item \label{(f)-Lem:gammaavoidatt} We have the segment inequality
\begin{multline*}
\qquad\quad \int_{(B(x_0, t_0, D) \times B(x_0, t_0, D)) \setminus S'} \int_0^{l_{y_1, y_2}} f(\gamma_{y_1, y_2} (s)) ds dg_{t_0} (y_1) dg_{t_0} (y_2) \\ \leq C_* \int_{B(x_0, t_0, 10D)} f dg_{t_0} 
\end{multline*}
for any non-negative, bounded and Borel measurable function $f : M \to [0, \infty)$.
\end{enumerate}
\end{Lemma}

\begin{proof}
Without loss of generality, we may assume in the following that $D > 1$.
Let $0 < \eta < \min \{  (10 D)^{-1}, \sqrt{\delta/4} \}$ be a constant whose value we will determine in the course of this proof, depending only on $A, E, D, \delta$ and choose $\rho = \rho (A, E, \eta) > 0$ according to Proposition \ref{Prop:effectivelimitXX}.
Then, by Proposition \ref{Prop:effectivelimitXX} for $(q,t) \leftarrow (x_0, t_0)$, we can find a pointed singular space $(\XX, q_\infty) = (X, d, \RR, g, q_\infty)$ with $\Ric = 0$ on $\RR$ and mild singularities, subsets $U \subset \RR$, $V \subset M$ and a diffeomorphism $\Phi : U \to V$ such that assertions \ref{(a)-Prop:effectivelimitXX}--\ref{(g)-Prop:effectivelimitXX} of Proposition~\ref{Prop:effectivelimitXX} hold.
Moreover, by assertion \ref{(g)-Prop:effectivelimitXX} of Proposition~\ref{Prop:effectivelimitXX} and a covering argument involving volume comparison on $\XX$ (see \cite[Proposition 4.1]{Bamler-CGN}), we find a constant $E^* = E^* (D, E) < \infty$ such that for all $0 < s < 1$
\begin{equation} \label{eq:rrmsEstarspp}
 \big| \{ \rrm^\infty < s \} \cap B^X (q_\infty, 10 D) \cap \RR \big| \leq E^* s^{\mathbf{p}}. 
\end{equation}
Here $\rrm^\infty$ denotes the curvature radius on $\XX$.

By \cite[Proposition 7.4]{Bamler-CGN} we have a segment inequality on $\XX$.
More specifically, we can find an open subset $\mathcal{G}^* \subset \RR \times \RR$ such that $(\RR \times \RR) \setminus \mathcal{G}^*$ has measure zero and such that for any $(z_1, z_2) \in \mathcal{G}^*$ there is a unique minimizing arclength geodesic
\[ \gamma^*_{z_1, z_2} : [0, d^X (z_1, z_2)] \to \RR. \]
and $d^X (z_1, z_2)$ as well as $\gamma^*$ depends smoothly on $z_1, z_2$.
Then the segment inequality on $\XX$ can be expressed as
\begin{equation} \label{eq:segmentonXXinproof}
 \int_{(B^X(q_\infty, 2D)  \times B^X (q_\infty, 2D)) \cap \mathcal{G}^*} \int_0^{d^X (z_1, z_2)} f(\gamma^*_{z_1, z_2} (s)) ds dg(z_1) dg(z_2) \leq C \int_{\RR} f dg 
\end{equation}
for any non-negative, bounded and Borel measurable function $f : \RR \to [0,\infty)$.
Here $C = C(D) < \infty$ denotes a uniform constant (the constant $C$ depends only on $D$ due to volume comparison on $\XX$, see \cite[Proposition 4.1]{Bamler-CGN}).
Pick $0 < \sigma_* = \sigma_* (D, \delta) < 1$ such that
\begin{equation} \label{eq:choiceofsigmastar}
 C E^* 4^{\mathbf{p}} \sigma_*^{\mathbf{p}-1} < \delta/32 \textQQqq{and}  \omega_n \sigma_*^n, \, \omega_n^2 (2D)^n \sigma_*^n < \delta / 32. 
\end{equation}
Here $\omega_n$ denotes the volume of the unit ball in Euclidean space.

We will now construct a subset $S^* \subset \RR \times \RR$ which will be used to construct $S'$.
Define first
\[ S^*_1 := \big( \big( B^X (q_\infty, 2D) \cap \RR \big) \times \big( B^X (q_\infty, 2D) \cap \RR \big)  \big) \setminus \mathcal{G}^*. \]
Note that $S^*_1$ is closed in $( B^X (q_\infty, 2D) \cap \RR ) \times ( B^X (q_\infty, 2D) \cap \RR )$ and has measure zero.
Next, define
\begin{multline*}
 S_2^* := \big\{ (z_1, z_2) \in \big( B^X (q_\infty, 2D) \cap \RR \big) \times \big( B^X(q_\infty, 2D) \cap \RR \big) \;\; : \;\; d^X (z_1, z_2) \leq \sigma_* \big\} \\
 = \bigcup_{z_1 \in B^X (q_\infty, 2D) \cap \RR} \{ z_1 \} \times \big( \big( B^X (q_\infty, 2D) \cap \RR \big) \cap \overline{B^X (z_1, \sigma_*)} \big).
\end{multline*}
Then $S^*_2$ is closed in $( B^X (q_\infty, 2D) \cap \RR ) \times ( B^X (q_\infty, 2D) \cap \RR )$.
By volume comparison (see \cite[Proposition 4.1]{Bamler-CGN}) and Fubini's Theorem, we have
\[ |S^*_2| \leq \big| B^X (q_\infty, 2D) \cap \RR \big| \cdot \omega_n \sigma_*^n \leq \omega_n^2 (2D)^n \sigma_*^n < \delta / 32. \]
Lastly, set
\begin{multline*}
 S^*_3 := \big\{ (z_1, z_2) \in \big( B^X (q_\infty, 2D) \cap \RR \big) \times \big( B^X(q_\infty, 2D) \cap \RR \big) \cap \mathcal{G}^* \;\; : \\
  \;\; \rrm^\infty (\gamma_{z_1, z_2}^* (s)) \leq 2\sigma_* \textQq{for some} s \in [0, d^X (z_1, z_2)] \big\}. 
\end{multline*}
Then $S^*_3$ is closed in $(( B^X (q_\infty, 2D) \cap \RR ) \times ( B^X (q_\infty, 2D) \cap \RR )) \cap \mathcal{G}^*$.
We will now bound its measure. 
For this, let
\[ W := \{ \rrm^\infty < 4\sigma_* \} \cap B(q_\infty, 4D) \cap \RR. \]
Note that, since $\rrm^\infty$ is $1$-Lipschitz, we have for any $(z_1, z_2) \in S^*_3 \setminus S^*_2$
\[ |\{ s \in [0, d(z_1, z_2)] \;\; : \;\; \gamma^*_{z_1, z_2} (s) \in W \}| \geq \sigma_*. \]
So setting $f := \chi_W$ in (\ref{eq:segmentonXXinproof}) yields, in view of (\ref{eq:rrmsEstarspp}) and (\ref{eq:choiceofsigmastar}),
\[ |S^*_3 \setminus S^*_2| \cdot \sigma_* \leq C E^* (4\sigma_*)^{\mathbf{p}} < (\delta / 32) \sigma_*, \]
which implies
\[ \big|  S^*_3 \setminus S^*_2 \big| < \delta / 32. \]
Set now
\[ S^* := S^*_1 \cup S^*_2 \cup S^*_3. \]
Then
\[ |S^* | = | S^*_1 \cup S^*_2 \cup (S^*_3 \setminus S^*_2) | < \delta / 16. \]

We will now construct $S' \subset M$.
For this note that
\[
 \big(  B^X (q_\infty, 2D)  \times B^X (q_\infty, 2D) \big)  
 \setminus S^* 
 \subset \{ \rrm^\infty > \sigma_* \} \times \{ \rrm^\infty > \sigma_* \},
\]
because any pair of points that belongs to the set on the left-hand side can be connected by a unique minimizing geodesic inside $\{ \rrm^\infty > 2\sigma_* \}$ (since $S^*_3 \subset S^*$).
Therefore, we also have $\rrm^\infty > 2\sigma_* > \sigma_*$ at its endpoints of such a geodesic.
So, assuming $\eta < \sigma_*$ and using Proposition \ref{Prop:effectivelimitXX}\ref{(c)-Prop:effectivelimitXX}, we have
\[ \big( B^X (q_\infty, 2D)  \times B^X (q_\infty, 2D)  \big)  
 \setminus S^* 
 \subset U \times U. \]
Let now $S'$ be defined as follows
\[ S' := B(x_0, t_0, D) \times B(x_0, t, D) \setminus ( \Phi \times \Phi ) \Big(  \big( B^X (q_\infty, 2D)  \times B^X (q_\infty, 2D)\big)  
 \setminus S^* \Big). \]
Then $S'$ is closed in $B(x_0, t_0, D) \times B(x_0, t, D)$ and can be expressed as the union of
\[ S'_1 := (\Phi \times \Phi) (  S^* ) \]
with
\[  S'_2 := (B(x_0, t_0, D) \setminus V) \times (B(x_0, t, D) \setminus V).  \]
For sufficiently small $\eta$ (in a uniform way) we may assume that Jacobian of $\Phi$ is bounded from above by $2$ (see Proposition \ref{Prop:effectivelimitXX}\ref{(b)-Prop:effectivelimitXX}) and hence
\[ |S'_1|_{t_0} \leq 4 | S^* | < \delta / 4. \]
Furthermore, by Proposition \ref{Prop:effectivelimitXX}\ref{(d)-Prop:effectivelimitXX}
\[ |S'_2 |_{t_0} \leq |B(x_0, t_0, D) \setminus V|^2_{t_0} < \eta^2 < \delta / 4. \]
So
\[ |S'|_{t_0} \leq |S'_1 |_{t_0} + |S'_2 |_{t_0} < \delta / 2, \]
which implies assertion \ref{(a)-Lem:gammaavoidatt}.

Before we verify the remaining assertions, let us summarize our construction:
For any $(y_1, y_2) \in (B(x_0, t_0, D) \times B(x_0, t_0, D)) \setminus S' \subset V \times V$, we have
\[ (\Phi^{-1} (y_1), \Phi^{-1} (y_2)) \in  \big(  B^X (q_\infty, 2D)  \times B^X (q_\infty, 2D)  \big)   \setminus S^*.  \]
As $S^*_3 \subset S^*$, this implies that for all $s \in [0, d^X (\Phi^{-1} (y_1), \Phi^{-1} (y_2))]$ we have
\[ \rrm^\infty (\gamma^*_{\Phi^{-1} (y_1), \Phi^{-1} (y_2)} (s) ) > 2 \sigma_*. \]
So, since we assumed $\eta < \sigma_*$, we conclude
\[ \gamma^*_{\Phi^{-1} (y_1), \Phi^{-1} (y_2)} (s)  \in U \]
for all such $s$.
Hence it is possible to make the following definition: 
for any $(y_1, y_2) \in (B(x_0, t_0, D) \times B(x_0, t_0, D)) \setminus S' \subset V \times V$ let
\[ l_{y_1, y_2} := d^X (\Phi^{-1} (y_1), \Phi^{-1} (y_2)) \]
and
\[ \gamma_{y_1, y_2} (s) := \Phi \big( \gamma^*_{\Phi^{-1} (y_1), \Phi^{-1} (y_2)} (s) \big), \qquad s \in [0, l_{y_1, y_2}]. \]
Assertions \ref{(b)-Lem:gammaavoidatt}, \ref{(d)-Lem:gammaavoidatt}, \ref{(e)-Lem:gammaavoidatt} now follow immediately for small enough $\eta$, depending on $\delta$.
For assertion \ref{(c)-Lem:gammaavoidatt} observe that for small enough $\eta$, we have
\[ \rrm (\gamma_{y_1, y_2} (s) , t_0) > \tfrac12 \rrm^\infty (\Phi^{-1} (\gamma_{y_1, y_2} (s))) = \tfrac12 \rrm^\infty (\gamma^*_{\Phi^{-1} (y_1), \Phi^{-1} (y_2)} (s))) > \sigma_*.  \]
Assertion \ref{(f)-Lem:gammaavoidatt} follows from (\ref{eq:segmentonXXinproof}) by replacing $f : M \to [0, \infty)$ by $(f \circ \Phi) \chi_U : \RR \to [0, \infty)$.
\end{proof}

We can now prove the main result of this subsection.
Similarly to the previous lemma, Lemma~\ref{Lem:gammaavoidatt}, the following proposition asserts the existence of almost geodesic curves between almost every pair of points at a given time $t_0$.
However, in contrast to Lemma~\ref{Lem:gammaavoidatt}, Proposition~\ref{Prop:almostgeodesicalmostalways} asserts additionally that along any such curve we have a lower bound on $\rrm$ and a small upper bound on $|{\Ric}|$ \emph{at all times $t \in [-1,0]$} --- not only at time $t_0$.

In a vague sense, Proposition~\ref{Prop:almostgeodesicalmostalways} follows by analyzing the set of points $x \in M$ where $\rrm (x, t) \leq 2\sigma$ for some $t \in [-1,0]$, where $\sigma$ is a small constant.
We will find that the volume of this set is roughly bounded by $C \sigma^{\mathbf{p} - 2}$.
On the other hand, we will show that any almost geodesic on which $\rrm (\cdot, t) \leq \sigma$ for some $t \in [-1,0]$ must intersect this set of points in a curve of length $\gtrsim \sigma$.
Using the segment inequality from assertion \ref{(f)-Lem:gammaavoidatt} of Lemma~\ref{Lem:gammaavoidatt} and a length distortion bound, it will follow that this set of almost geodesics has measure $< C \sigma^{\mathbf{p} -3}$.
So, since we assumed that $\mathbf{p} > 3$, this measure can be made arbitrarily small.
Therefore, the set of almost geodesics on which $\rrm (x, t) \leq \sigma$ for some $t \in [-1,0]$ can be discarded.

\begin{Proposition} \label{Prop:almostgeodesicalmostalways}
For any $A, E, D < \infty$, $\mathbf{p} > 3$ and $\delta > 0$ there are constants $\rho = \rho (A, \linebreak[2] E, \linebreak[2] D, \linebreak[2] \mathbf{p}, \linebreak[2] \delta), \linebreak[1]\sigma (A, E, D, \mathbf{p}, \delta) > 0$ such that the following holds:

Let $(M, (g_t)_{t \in [-2,0]} )$ be a Ricci flow on a compact, $n$-dimensional manifold $M$ with the property that
\begin{enumerate}[label=(\roman*)]
\item \label{(i)-Prop:almostgeodesicalmostalways} $\nu [ g_{-2}, 4] \geq - A$.
\item \label{(ii)-Prop:almostgeodesicalmostalways} $|R| \leq \rho$ on $M \times [-2,0]$.
\item \label{(iii)-Prop:almostgeodesicalmostalways} For all $(x,t) \in M \times [-1,0]$ and $0 < r, s < 1$ we have
\[ | \{ \rrm (\cdot, t) < sr \} \cap B(x,t, r) |_t \leq E s^{\mathbf{p}} r^n. \]
\end{enumerate}
Let $x_0 \in M$ and $t_0 \in [-\frac12,0]$.
Then there is a subset
\[ S \subset B(x_0, t_0, D) \times  B(x_0, t_0, D) \]
such that the following holds:
\begin{enumerate}[label=(\alph*)]
\item \label{(a)-Prop:almostgeodesicalmostalways} We have
\[ |S|_{t_0} < \delta \]
with respect to the product measure $dg_{t_0} \otimes dg_{t_0}$
\item \label{(b)-Prop:almostgeodesicalmostalways} For any
\[ (y_1, y_2) \in \big(B(x_0, t_0 , D)  \times B(x_0, t_0, D) \big) \setminus S \]
there is a smooth curve $\gamma_{y_1, y_2} : [0,1] \to M$ with $\gamma_{y_1, y_2} (0) = y_1$, $\gamma_{y_1, y_2} (1) = y_2$ such that
\[ | \length_{t_0} (\gamma_{y_1, y_2}) - d_{t_0} (y_1, y_2) | < \delta \]
and such that
\[ \qquad \rrm (\gamma_{y_1, y_2} (s), t) > \sigma \textQQqq{for all} s \in [0,1] \textQq{and} t \in [-1,0] \]
and
\[ \qquad |{\Ric}| (\gamma_{y_1, y_2} (s), t) < \delta \textQQqq{for all} s \in [0,1] \textQq{and} t \in [-1,0]. \]
\end{enumerate}
\end{Proposition}

\begin{proof}
Fix $A, E, D < \infty$, $\mathbf{p} > 3$ and $\delta > 0$.
Assume without loss of generality that $\delta < 0.1$.
Assume that $\rho$ is chosen small enough such that we can apply Lemma~\ref{Lem:gammaavoidatt} at $(x_0, t_0)$ and let $0 < \sigma_* = \sigma_* (A, E, D, \mathbf{p}, \delta) < 1$  and $C_* (A, E, D) < \infty$ be the constants from this lemma.
Next, using Proposition \ref{Prop:Distdistortion}, we may choose a constant $D^* = D^* (A, D) < \infty$ such that
\[ B(x_0, t_0, D) \subset B(x_0, t, D^*) \textQQqq{for all} t \in [-1,0]. \]
And using assumption \ref{(iii)-Prop:almostgeodesicalmostalways}, a covering argument and Proposition \ref{Prop:VolumeBound}, we can find a constant $E^* = E^* (A, D, E) < \infty$ such that for all $0 < s < 1$ and $t \in [-1,0]$
\begin{equation} \label{eq:Estarboundgeodesicintime}
 | \{ \rrm (\cdot, t) < s \} \cap B(x_0, t_0, D) |_t < E^* s^{\mathbf{p}}. 
\end{equation}
Let $\eps = \eps (A) > 0$ be the constant from the Backwards and Forward Pseudolocality Theorem, Proposition \ref{Prop:Pseudoloc} and assume without loss of generality that $\eps < 0.1$.
Moreover, choose an integer $N = N (A, E, D, \mathbf{p}, \delta) < \infty$ such that
\[ 8 \delta^{-1}  N C_* E^* \bigg( \frac{2}{N^{1/2} \eps} \bigg)^{\mathbf{p}-1}  < 1 \textQQqq{and} \frac{10}{N^{1/2} \eps} < \eps\sigma_* < 1 \]
(this is possible since $\mathbf{p} > 3$) and determine $\sigma_0 = \sigma_0 (A, E, D, \mathbf{p}, \delta) > 0$ by
\[ \sigma_0 := \frac2{N^{1/2} \eps}. \]
Then
\begin{equation} \label{eq:recallconstrelation}
 \sigma_0 < \eps \sigma_*, \quad 8 \delta^{-1} N C_* E^* \sigma_0^{\mathbf{p}-1} < 1, \quad \frac{100}N < (\eps \sigma_*)^2 \textQq{and}  (\eps \sigma_0/2)^2 = \frac1N.  
\end{equation}
Define $\sigma = \sigma(A, E, D, \mathbf{p}, \delta) > 0$ by
\[ \sigma := \eps \sigma_0 / 2 = \frac1{\sqrt{N}}. \]
We will assume in the following that $\rho = \rho (A, E, D, \mathbf{p}, \delta) > 0$ is chosen small enough such that, using Proposition \ref{Prop:Ricsmall}, we can conclude that $|{\Ric}|(y,t) < \delta$ at any $(y,t) \in M \times [-1,0]$ at which $\rrm (y,t) > \sigma$.

Set for $i = 0, \ldots, N-1$
\[ W_i := \bigg\{ \rrm \Big( \cdot, - \frac{i}{N} \Big) < \sigma_0 \bigg\} \subset M \]
and define
\[ f(x) := \sum_{i=0}^{N-1} \chi_{W_i}, \]
where $\chi_{W_i}$ denotes the characteristic function of $W_i$.
Then, by (\ref{eq:Estarboundgeodesicintime}),
\begin{equation} \label{eq:fintegralNE}
 \int_{B(x_0, t_0, D^*)} f dg_{t_0} < N E^* \sigma^{\mathbf{p}}_0 . 
\end{equation}
Consider now the subset $S' \subset B(x_0, t_0, D) \times B(x_0, t_0, D)$ from Lemma \ref{Lem:gammaavoidatt} and the family of curves
\[ \gamma_{y_1, y_2}: [0, l_{y_1, y_2}] \longrightarrow M, \qquad (y_1, y_2) \in  \big( (B(x_0, t_0, D) \times B(x_0, t_0, D) \big) \setminus S' \]
Set
\begin{multline*}
 S'' := \bigg\{ (y_1, y_2) \in \big( B(x_0, t_0, D) \times B(x_0, t_0, D) \big) \setminus S' \;\; : \\
  \;\;  \int_0^{l_{y_1, y_2}} f(\gamma_{y_1, y_2} (s)) ds \geq 2 \delta^{-1} N C_* E^* \sigma_0^{\mathbf{p}}  \bigg\}.
\end{multline*}
Then, using the segment inequality in assertion \ref{(e)-Lem:gammaavoidatt} of Lemma \ref{Lem:gammaavoidatt} combined with (\ref{eq:fintegralNE}), we find
\[ |S''|_{t_0} \cdot 2 \delta^{-1} N C_* E^* \sigma_0^{\mathbf{p}} < C_* NE^* \sigma_0^{\mathbf{p}}. \]
This implies $|S''|_{t_0} < \delta / 2$, so if we set $S := S' \cup S''$, then assertion \ref{(a)-Prop:almostgeodesicalmostalways} holds.

We will now verify assertion \ref{(b)-Prop:almostgeodesicalmostalways}.
To do this, observe that, by (\ref{eq:recallconstrelation}), for any $(y_1, y_2) \in (B(x_0, t_0, D) \times B(x_0, t_0, D)) \setminus S$ we have for any $i = 0, \ldots, N-1$
\begin{equation} \label{eq:1dintsmallW}
 \int_0^{l_{y_1, y_2}} \chi_{W_i} (\gamma_{y_1, y_2} (s)) ds < 2 \delta^{-1} N C_* E^* \sigma_0^{\mathbf{p}} <  \sigma_0/4.
\end{equation} 
We will use this bound to show that for all $i = 0, \ldots, N-1$ we have
\begin{equation} \label{eq:wantforalli}
 \rrm \Big( \gamma_{y_1, y_2} (s), - \frac{i}{N} \Big) > \sigma_0 / 2 \textQQqq{for all} s \in [0,l_{y_1, y_2}] . 
\end{equation}

To see that (\ref{eq:wantforalli}) holds for all $i =0, \ldots, N - 1$, recall first that by assertion \ref{(c)-Lem:gammaavoidatt} of Lemma \ref{Lem:gammaavoidatt} we have $\rrm (\gamma_{y_1, y_2} (s), t_0) > \sigma_*$ for all $s \in [0,l_{y_1, y_2}]$.
It follows by Proposition \ref{Prop:Pseudoloc} and (\ref{eq:recallconstrelation}) that for all $t \in [-1,0]$ with $|t - t_0| \leq (\eps\sigma_*)^2$ we have
\[
 \rrm (\gamma_{y_1, y_2} (s), t) > \eps\sigma_* > \sigma_0 / 2 \textQQqq{for all} s \in [0,l_{y_1, y_2}]. 
\]
So since $(\eps \sigma_*)^2 > \frac{100}N$ (see (\ref{eq:recallconstrelation})), there is at least one $i$ for which (\ref{eq:wantforalli}) holds.
More specifically, there are $i_1, i_2 \in \{ 0, \ldots, N-1 \}$ such that $- \frac{i_2}N < t_0 < - \frac{i_1}N$ and such that (\ref{eq:wantforalli}) holds for all $i = i_1, \ldots, i_2$.
Assume now that (\ref{eq:wantforalli}) does not hold for all $i = 0, \ldots, N-1$ and pick $i_0$ such that (\ref{eq:wantforalli}) does not hold for $i = i_0$ and such that $|t_0 + \frac{i_0}{N} |$ is minimal with this property.
In other words, (\ref{eq:wantforalli}) holds for all $i = 0, \ldots, N-1$ for which $|t_0 + \frac{i}{N} | < |t_0 + \frac{i_0}{N} |$ but not for $i = i_0$.
We will now derive a contradiction to this assumption.
Recall from (\ref{eq:recallconstrelation}) that $(\eps \sigma_0 / 2)^2 = \frac1N$.
So, by Proposition \ref{Prop:Pseudoloc}, we have
\begin{equation} \label{eq:biggerthansigmat}
  \rrm (\gamma_{y_1, y_2} (s), t) > \eps\sigma_0 / 2 = \sigma \textQQqq{for all} s \in [0,l_{y_1, y_2}] 
\end{equation}
for all $t \in [- \frac{i_0}{N},  t_0]$ or $t \in [t_0 , - \frac{i_0}{N}]$, depending on whether $i_0 > i_2$ or $i_0 < i_1$.
As discussed before, by our small choice of $\rho$, this implies that for all such $t$ we have
\begin{equation} \label{eq:Ricboundbetween}
 |{\Ric}| (\gamma_{y_1, y_2} (s), t) < \delta \textQQqq{for all} s \in [0, l_{y_1, y_2}]. 
\end{equation}
So, since $| \gamma'_{y_1, y_2} (s) |_{t_0} < 1 + \delta$ and $\delta < 0.1$, a distance distortion estimate yields that $| \gamma'_{y_1, y_2} (s) |_{t} < 2$ for all $t \in [- \frac{i_0}{N},  t_0]$ or $t \in [t_0 , - \frac{i_0}{N}]$.
So the function $s \mapsto \rrm (\gamma_{y_1, y_2} (s), - \frac{i_0}N)$ is $2$-Lipschitz.
By the choice of $i_0$, there is an $s_0 \in [0, l_{y_1, y_2}]$ such that
\[ \rrm \Big( \gamma_{y_1, y_2} (s_0), - \frac{i_0}N \Big) \leq \sigma_0 / 2 \]
Using the $2$-Lipschitz property, we conclude that for all $s \in [0, l_{y_1, y_2} ]$ with $|s - s_0| < \sigma_0 / 4$ we have
\[ \rrm \Big( \gamma_{y_1, y_2} (s), - \frac{i_0}N \Big) < \sigma_0.  \]
In other words, $\gamma_{y_1, y_2} (s) \in W_{i_0}$ for all $s \in [0, l_{y_1, y_2} ]$ with $|s - s_0| < \sigma_0 / 4$.
Since, by Lemma \ref{Lem:gammaavoidatt}\ref{(e)-Lem:gammaavoidatt}, $l_{y_1, y_2} > \sigma_* > \sigma_0 / 4$, this contradicts (\ref{eq:1dintsmallW}).

So (\ref{eq:wantforalli}) holds for all $i = 0, \ldots, N-1$.
As before, it follows that (\ref{eq:biggerthansigmat}) and (\ref{eq:Ricboundbetween}) hold for all $t \in [-1,0]$.
This finishes the proof of assertion \ref{(b)-Prop:almostgeodesicalmostalways}.
\end{proof}

\subsection{Existence of almost optimal $\LL$-geodesics and a reduced volume bound}
In this subsection, we will use the existence of almost geodesics that avoid regions of high curvature to construct short $\LL$-geodesics between a given basepoint and a set of points of large measure.
Based on this construction we will eventually derive a good lower bound on the reduced volume at that basepoint.

\begin{Lemma} \label{Lem:LLalmostperfect}
For any $A, E, D < \infty$, $\mathbf{p} > 3$, $0 < \tau_0 \leq 1/2$ and $\delta > 0$ there is a constant $\rho = \rho (A, \linebreak[2] E, \linebreak[2] D, \linebreak[2] \mathbf{p}, \linebreak[2] \tau_0, \linebreak[2] \delta) > 0$ such that the following holds:

Let $(M, (g_t)_{t \in [-2,0]} )$ be a Ricci flow on a compact, $n$-dimensional and orientable manifold $M$ with the property that
\begin{enumerate}[label=(\roman*)]
\item \label{(i)-Lem:LLalmostperfect} $\nu [ g_{-2}, 4] \geq - A$.
\item \label{(ii)-Lem:LLalmostperfect} $|R| \leq \rho$ on $M \times [-2,0]$.
\item \label{(iii)-Lem:LLalmostperfect} For all $(x,t) \in M \times [-1,0]$ and $0 < r, s < 1$ we have
\[ | \{ \rrm (\cdot, t) < sr \} \cap B(x,t, r) |_t \leq E s^{\mathbf{p}} r^n. \]
\end{enumerate}
Consider a point $x_0 \in M$.
Then there is a subset
\[ S \subset B(x_0, t_0, D) \]
such that the following holds:
\begin{enumerate}[label=(\alph*)]
\item \label{(a)-Lem:LLalmostperfect} We have
\[ |S|_{0} < \delta. \]
\item \label{(b)-Lem:LLalmostperfect} For any
\[ z \in B(x_0, 0 , D)  \setminus S \]
we have
\[ \ov{L}_{(x_0, 0)} (z, - \tau_0) = 2 \sqrt{\tau_0} L_{(x_0,0)} (z,- \tau_0) < d^2_0 (x_0, z) + \delta. \]
Here $L_{(x_0, 0)}$ denotes the $\LL$-distance based at $(x_0, 0)$.
\end{enumerate}
Moreover, the reduced volume at $(x_0, 0)$ satisfies
\begin{equation} \label{eq:reducedVdelta}
 \td{V}_{(x_0, 0)} (\tau_0) >  \int_{B(x_0, 0, D)} (4\pi \tau_0)^{-n/2} \exp \Big({- \frac{d^2_0 (x_0, z)}{4 \tau_0} } \Big) dg_0 (z) - \delta. 
\end{equation}
\end{Lemma}

\begin{proof}
Let us first construct $S$ such that assertions \ref{(a)-Lem:LLalmostperfect} and \ref{(b)-Lem:LLalmostperfect} hold.
The bound (\ref{eq:reducedVdelta}) will then follow easily.

In order to construct $S$, we first introduce a constant $0 < \theta  \leq \tau_0$ that will be chosen small enough in the course of the proof depending only on $A, E, D, \tau_0$ and $\delta$.
The constant $\rho$ will be chosen small enough depending on $A, E, D, \tau_0, \delta$ and $\theta$ in the course of this proof.
In the following we will construct curves with bounded $\LL$-length between $(x_0, 0)$ and a large set of points $(z, -\tau_0)$.
These curves will arise as a concatenation of a short $\LL$-geodesic between $(x_0, 0)$ and some point $(y, -\theta)$ and a reparameterization of the curves from Proposition \ref{Prop:almostgeodesicalmostalways} to the interval $[\theta, \tau_0]$.

Let us first construct sufficiently many short $\LL$-geodesics based at $(x_0,0)$.
For this, we will argue similarly as in the Claim of Lemma \ref{Lem:shortLLLprrm}.
As explained in the proof of this claim we obtain the following special case of (\ref{eq:integralKovL}) for $z \leftarrow x_0$
\[ \int_M K(x_0, 0; y, - \theta) \ov{L}_{(x_0,0)} (y, - \theta) dg_{- \theta} (y) \leq 2n \theta. \]
Next, assuming $\rho < 1$,
\[ \ov{L}_{(x_0,0)} (y, - \theta) > - 2 \sqrt{\theta} \int_0^{\theta} \sqrt{\tau} d\tau = - \frac43 \theta^2 > - 2\theta^2. \]
Consider now the integral
\[ \int_M K(x_0, 0; y, - \theta) \big( \ov{L}_{(x_0,0)} (y, - \theta) + 2 \theta^2 \big) dg_{- \theta} (y) \leq 2n \theta + 2\theta^2. \]
The integrand of this integral is positive everywhere.
So, by volume distortion estimates, the lower bound on the heat kernel (see Proposition \ref{Prop:GaussianHKbounds}) and the fact that $|B(x_0, 0, \sqrt{\theta})|_0 > c \theta^{n/2}$ for some uniform $c = c(A) > 0$ (see Proposition \ref{Prop:VolumeBound}), we can find a constant $C = C(A) < \infty$ such that
\[ \fint_{B(x_0, 0, \sqrt\theta)} \ov{L}_{(x_0,0)} (y, - \theta) dg_0(y) \leq \fint_{B(x_0, 0, \sqrt\theta)} \big( \ov{L}_{(x_0,0)}  (y, - \theta) + 2 \theta^2 \big) dg_0(y) \leq C \theta. \]
It follows that if we set
\[ U := \big\{ y \in B(x_0, 0, \sqrt{\theta} ) \;\; : \;\;  \ov{L}_{(x_0,0)} (y, -\theta) < 2C \theta \big\}, \]
then 
\[ |U|_0 >  \frac12 \big| B(x_0, 0, \sqrt{\theta}) \big|_0  > \frac{c}2 \theta^{n/2}. \]

We now fix the constant $\theta = \theta (A, E, D, \mathbf{p}, \tau_0, \delta) > 0$ small enough such that $0 < \theta < \tau_0$ and such that the following holds:
\begin{equation} \label{eq:thetachoice}
 2C \sqrt{\theta} + e^\theta \frac{\sqrt{\tau_0}}{\sqrt{\tau_0} - \sqrt{\theta}} \big( d + 2 \sqrt{\theta}  \big)^2 + \theta < d^2 + \delta \textQQqq{for all} 0 \leq d < D.
\end{equation}
Based on this choice, we pick $\delta' = \delta' (A, E, D, \mathbf{p}, \theta, \delta) > 0$ such that
\[ \delta' < \theta \textQQqq{and} \delta' \Big( \frac{c}2 \theta^{n/2} \Big)^{-1} < \delta. \]
Apply now Proposition \ref{Prop:almostgeodesicalmostalways} with $A \leftarrow A$, $E \leftarrow E$, $D \leftarrow D$, $\mathbf{p} \leftarrow \mathbf{p}$, $\delta \leftarrow \delta'$, $(x_0, t_0) \leftarrow (x_0, 0)$ for sufficiently small $\rho$ (depending on $A, E, D, \tau_0, \delta$) and let $S' \subset B(x_0, \linebreak[1] 0, \linebreak[1] D) \times B(x_0, 0, D)$ be the subset that is denoted by $S$ in that proposition.
For any $z \in B(x_0, 0, D)$ let $m (z) := | S' \cap (\{ z \} \times B(x_0, 0, D)) |_0$ be the time-$0$ measure of the section through $z$.
By Fubini's Theorem
\[ \int_{B(x_0, 0, D)} m(y) dg_0 (y) = |S'|_0 < \delta'. \]
It follows that there is a point $y \in U$ such that
\[ m(y) < \delta' \Big( \frac{c}2 \theta^{n/2} \Big)^{-1} < \delta. \]
Let now $S \subset B(x_0, 0, D)$ be the subset for which
\[ \{ y \} \times S = S' \cap \big( \{ y \} \times B(x_0, 0, D) \big). \]
Then $S$ satisfies assertion \ref{(a)-Lem:LLalmostperfect}.

Next, we show assertion \ref{(b)-Lem:LLalmostperfect}.
Choose an $\LL$-geodesic $\gamma^* : [0, \theta] \to M$ with $\gamma(0) = x_0$, $\gamma(\theta) = y$ and
\begin{equation} \label{eq:LLgammastar}
 \LL_0 (\gamma^*) = L_{(x_0, 0)} (y, - \theta) = \frac1{2 \sqrt\theta} \ov{L}_{(x_0,0)} (y, -\theta) < C \sqrt\theta. 
\end{equation}
Let $z \in S$ and recall that $(y,z) \in S'$.
Let $\gamma : [\sqrt{\theta},\sqrt{\tau_0}] \to M$ be a constant speed parameterization of the curve $\gamma_{z, y} : [0,1] \to M$ obtained in Proposition \ref{Prop:almostgeodesicalmostalways}\ref{(b)-Prop:almostgeodesicalmostalways} and define $\ov\gamma : [\theta, \tau_0] \to M$ by $\ov\gamma (\tau) := \gamma(\sqrt{\tau})$.
So $\ov\gamma' (\tau) = \frac1{2 \sqrt{\tau}} \gamma' (\sqrt{\tau})$.
Thus, since by Proposition \ref{Prop:almostgeodesicalmostalways}, we have $|{\Ric}|(\ov{\gamma} (s), t) < \delta' < \theta$ for all $s \in [\theta, \tau_0]$ and $t \in [-1,0]$, we have
\begin{align}
2 \sqrt{\tau_0} \int_{\theta}^{\tau_0} \sqrt{\tau} & |\ov{\gamma}' (\tau) |^2_{-\tau} d\tau \leq 2 \sqrt{\tau_0} \cdot e^{\delta'} \int_{\theta}^{\tau_0} \sqrt{\tau} |\ov{\gamma}' (\tau) |^2_{0} d\tau \notag \\
&\leq 2 \sqrt{\tau_0} \cdot e^{\theta} \int_{\theta}^{\tau_0} \frac1{4\sqrt{\tau}} |\gamma' (\sqrt\tau) |^2_{0} d\tau =  e^\theta \sqrt{\tau_0} \int_{\sqrt{\theta}}^{\sqrt{\tau_0}} |\gamma' (s)|_0^2 ds \notag \displaybreak[1] \\
&=  e^\theta \sqrt{\tau_0}\big( \sqrt{\tau_0} - \sqrt{\theta} \big) \bigg( \frac{ \length_0 (\gamma) }{\sqrt{\tau_0} - \sqrt{\theta}} \bigg)^2 \notag \displaybreak[1] \\
& \leq   e^\theta \frac{\sqrt{\tau_0}}{\sqrt{\tau_0} - \sqrt{\theta}} \big( d_0(y, z) + \delta' \big)^2 \notag \displaybreak[1] \\
& \leq   e^\theta \frac{\sqrt{\tau_0}}{\sqrt{\tau_0} - \sqrt{\theta}} \big( d_0(x_0,z) +  \sqrt{\theta} + \delta' \big)^2 \notag \\
&<   e^\theta \frac{\sqrt{\tau_0}}{\sqrt{\tau_0} - \sqrt{\theta}} \big( d_0(x_0,z) + 2 \sqrt{\theta}  \big)^2 . \label{eq:LLgammabar1}
\end{align}
Moreover, for sufficiently small $\rho$ (depending on $\theta$)
\begin{equation} \label{eq:LLgammabar2}
 \int_\theta^{\tau_0} \sqrt{\tau} R(\ov\gamma(\tau), -\tau) d\tau \leq \frac23  \rho \tau_0^{3/2} < \rho < \theta / 2. 
\end{equation}
Let now $\gamma^{**}$ be the concatenation of $\gamma^*$ and $\ov\gamma$.
Then by (\ref{eq:LLgammastar}), (\ref{eq:LLgammabar1}), (\ref{eq:LLgammabar2}) and (\ref{eq:thetachoice})
\begin{multline*}
 \ov{L}_{(x_0, 0)} (z,-\tau_0) \leq 2 \sqrt{\tau_0} \LL_0 (\gamma^{**}) \\ < 2 C \sqrt{\theta} +e^\theta \frac{\sqrt{\tau_0}}{\sqrt{\tau_0} - \sqrt{\theta}} \big( d_0(x_0,z) + 2 \sqrt{\theta}  \big)^2  + \theta 
 < d^2_0 (x_0, z) + \delta.
\end{multline*}
This proves assertion \ref{(b)-Lem:LLalmostperfect}.

We will now show how the first part of the lemma implies (\ref{eq:reducedVdelta}).
For this, observe that
\begin{align*}
 \td{V}_{(x_0, 0)} (\tau_0) &= \int_M (4\pi \tau_0)^{-n/2} e^{-l (z, - \tau)} dg_{-\tau_0} (z) \\
&> e^{-\rho} \int_{B(x_0, t_0, D) \setminus S } (4\pi \tau_0)^{-n/2} \exp \bigg({- \frac{\ov{L}_{(x_0,0)} (z, - \tau_0)}{4 \tau_0} }\bigg) dg_{0} (z) \displaybreak[1] \\
&> e^{-\rho} \int_{B(x_0, t_0, D) \setminus S } (4\pi \tau_0)^{-n/2} \exp \bigg({- \frac{d^2_0(x_0,z) + \delta}{4 \tau_0} }\bigg) dg_{0} (z) \displaybreak[1] \\
&> e^{-\rho} \int_{B(x_0, t_0, D) } (4\pi \tau_0)^{-n/2} \exp \bigg({- \frac{d^2_0(x_0,z) + \delta}{4 \tau_0} }\bigg) dg_{0} (z) \\
&\qquad\qquad\qquad\qquad\qquad\qquad\qquad\qquad\qquad - e^{-\rho + \delta / 4 \tau_0} (4\pi \tau_0)^{-n/2} |S|_0
\end{align*}
Note that $|B(x_0, t_0, D)|_0$ is bounded from above by a constant that only depends on $A$ and $D$ (see Proposition \ref{Prop:VolumeBound}).
So choosing $\rho$ and $\delta$ small enough implies (\ref{eq:reducedVdelta}).
\end{proof}

\begin{Lemma} \label{Lem:redvolumelarge}
For any $A, E < \infty$, $\mathbf{p} > 3$ and $\delta > 0$ there are constants $0 < \nu = \nu (  \delta), \rho = \rho (A, E, \mathbf{p}, \delta) < 1$ such that the following holds:

Let $(M, (g_t)_{t \in [-2,0]} )$ be a Ricci flow on a compact, $n$-dimensional and orientable manifold $M$ and $x_0 \in M$ and assume that
\begin{enumerate}[label=(\roman*)]
\item $\nu [ g_{-2}, 4] \geq - A$.
\item $|R| \leq \rho$ on $M \times [-2,0]$.
\item For all $(x,t) \in M \times [-1,0]$ and $0 < r, s < 1$ we have
\[ | \{ \rrm (\cdot, t) < sr \} \cap B(x,t, r) |_t \leq E s^{3.9} r^n. \]
\item We have
\[ |B(x_0, 0, \nu^{-1} )|_0 > ( \omega_n - \nu) (\nu^{-1})^n. \]
\end{enumerate}
Then for all $x \in B(x_0, 0, 1)$, the reduced volume satisfies.
\[ \td{V}_{(x,0)} (1/2) > 1 - \delta. \]
\end{Lemma}

Note that $\nu$ does not depend on $A$, $E$ or $\mathbf{p}$.

\begin{proof}
Let us first fix the constants.
Choose $\nu = \nu(\delta) > 0$ small enough such that
\begin{multline*}
 \int_{\IR^n \setminus B(0^n, \nu^{-1} )} (2 \pi )^{-n/2} e^{- |z|^2 / 2} dz < \delta / 3, \qquad \nu < 0.01, \qquad 4 \nu < \delta / 3, \\
 \text{and} \qquad (\omega_n - 2 \nu) (\nu^{-1})^n   > (\omega_n - 3 \nu) (\nu^{-1} + 1)^n.
\end{multline*}
Next, choose $\eta = \eta ( \delta) > 0$ small such that
\begin{multline*}
 \eta < \nu/2, \qquad  \frac{1}{1+ \eta}(\omega_n - \nu) (\nu^{-1})^n - \frac{\eta}{1+\eta} > (\omega_n - 2 \nu) (\nu^{-1})^n, \\
 \qquad (1- \eta) (\omega_n - 3\nu) > \omega_n - 4 \nu  \textQQqq{and}  (2\pi)^{-n/2} (\nu^{-2} +1) \eta + \eta < \delta / 3. 
 \end{multline*}

We will now use Proposition \ref{Prop:effectivelimitXX} with $(q,t) \leftarrow (x, 0)$, assuming $\rho$ to be sufficiently small, depending on $A, E, \eta$.
Proposition \ref{Prop:effectivelimitXX} yields a pointed singular space $(\XX, q_\infty) = (X, d, \RR, g, q_\infty)$ with mild singularities for which $\Ric \equiv 0$ on $\RR$, subsets $U \subset \RR$ and $V \subset M$ and a diffeomorphism $\Phi : U \to V$ such that assertions \ref{(a)-Prop:effectivelimitXX}--\ref{(g)-Prop:effectivelimitXX} of this proposition hold.
Note that assertion \ref{(a)-Prop:effectivelimitXX} states that $q_\infty \in U$ and $d_{0} (\Phi(q_\infty), x) < \eta$.
Then, by assertion \ref{(f)-Prop:effectivelimitXX},
\begin{align*}
 \big| B^X (q_\infty, \nu^{-1} + 1 ) \cap \RR \big| &> \frac{1}{1+ \eta} \big| B^M(x, 0, \nu^{-1} + 1) \big|_0 - \frac{\eta}{1+\eta} \displaybreak[1] \\
 &> \frac{1}{1+ \eta} \big| B^M(x_0, 0, \nu^{-1} ) \big|_0 - \frac{\eta}{1+\eta} \displaybreak[1] \\
 &> \frac{1}{1+ \eta}(\omega_n - \nu) (\nu^{-1})^n - \frac{\eta}{1+\eta} \displaybreak[1] \\
&> (\omega_n - 2 \nu) (\nu^{-1})^n
> (\omega_n - 3 \nu ) (\nu^{-1} + 1)^n.
\end{align*}
Since $\XX$ has mild singularities and satisfies $\Ric = 0$ on $\RR$, we can apply Bishop-Gromov volume comparison on $\XX$ (see \cite[Proposition 4.1]{Bamler-CGN}) and obtain that
\[ \big| B^X (q_\infty, r) \cap \RR \big| > (\omega_n - 3 \nu) r^n \textQQqq{for all} 0 < r \leq \nu^{-1} + 1. \]
Using assertion \ref{(f)-Prop:effectivelimitXX} of Proposition \ref{Prop:effectivelimitXX} again, we get
\begin{multline*}
 \big| B^M (x, 0, r) \big|_0 > (1 - \eta) ( \omega_n - 3 \nu) r^n - \eta > (\omega_n - 4 \nu) r^n - \eta \\\textQQqq{for all} 0 < r < \nu^{-1} + 1. 
\end{multline*}
So by Lemma \ref{Lem:LLalmostperfect}, for $\delta \leftarrow \eta$, $\tau_0 \leftarrow 1/2$, $x_0 \leftarrow x$ and sufficiently small $\rho$,
\begin{align*}
 \td{V}_{(x,0)} (1/2) &> \int_{B^M (x, 0, \nu^{-1})} \big( 4 \pi \cdot \tfrac12 \big)^{-n/2} \exp \Big({-\frac{d_0^2 (x, y)}{4 \cdot \frac12}}\Big) dg_0 (y) - \eta \displaybreak[1] \\
& = \int_{B^M (x, 0, \nu^{-1})} \bigg( - \int_{d_0(x,y)}^{\nu^{-1}} \bigg( - (2 \pi )^{-n/2}  r \exp \Big({-\frac{r^2}{2}}\Big)  \bigg) dr \\
&\qquad\qquad\qquad\qquad + (2 \pi )^{-n/2} \exp \Big({-\frac{(\nu^{-1})^2}{2}}\Big) \bigg) dg_0 (y) - \eta \displaybreak[1]  \\
 &= \int_0^{\nu^{-1}}  (2 \pi )^{-n/2}  r e^{-r^2/2} \big| B^M (x, 0, r) \big|_0 dr \\
 &\qquad\qquad\qquad\qquad + (2 \pi)^{-n/2} e^{-(\nu^{-1})^2/2} \big| B^M(x, 0, \nu^{-1}) \big|_0 - \eta \displaybreak[1] \\
& > (\omega_n - 4 \nu) \int_0^{\nu^{-1}} (2 \pi)^{-n/2} r^{n+1}  e^{-  r^2 / 2} dr + (2 \pi)^{-n/2}  (\nu^{-1} )^n e^{- (\nu^{-1})^2/2} \\
&\qquad\qquad\qquad\qquad - \eta \int_0^{\nu^{-1}} (2 \pi)^{-n/2} r e^{- r^2 / 2} dr - (2 \pi)^{-n/2} e^{- (\nu^{-1})^2 / 2} \eta - \eta \displaybreak[1] \\
 &> (\omega_n - 4 \nu) \int_{B(0^n, \nu^{-1} ) \subset \IR^n} (2 \pi)^{-n/2} e^{- |z|^2/2} dz -  (2\pi)^{-n/2} (\nu^{-2} +1) \eta - \eta  \displaybreak[1] \\
 &> 1 - \int_{\IR^n \setminus B(0^n, \nu^{-1} )} (2 \pi)^{-n/2} e^{- |z|^2/2} dz  \\
 &\qquad\quad- 4 \nu  \int_{B(0^n, \nu^{-1} ) \subset \IR^n} (2 \pi)^{-n/2} e^{- |z|^2/2} dz -  (2\pi)^{-n/2} (\nu^{-2} +1) \eta - \eta \displaybreak[1] \\
& > 1 - \delta / 3 - 4 \nu  - \delta / 3 \\
 & > 1 - \delta.
\end{align*}
This finishes the proof.
\end{proof}

\subsection{Proof of the regularity theorems}
We first need to establish the following gap theorem for the reduced volume:

\begin{Lemma} \label{Lem:epsregredvolume}
For any $A < \infty$ there are constants $\delta = \delta(A), \sigma = \sigma (A) > 0$ such that the following holds:

Let $(M, (g_t)_{t \in [-2,0]} )$ be a Ricci flow on a compact, $n$-dimensional and orientable manifold $M$ and $x_0 \in M$ and assume that
\begin{enumerate}[label=(\roman*)]
\item \label{(i)-Lem:epsregredvolume} $\nu [ g_{-2}, 4] \geq - A$.
\item \label{(ii)-Lem:epsregredvolume} $|R| \leq 1$ on $M \times [-2,0]$.
\item \label{(iii)-Lem:epsregredvolume} For all $x \in B(x_0, 0,1)$ the reduced volume satisfies
\[ \td{V}_{(x,0)} (1/2) > 1 - \delta. \]
\end{enumerate}
Then $\rrm (x_0, 0) > \sigma$.
\end{Lemma}

\begin{proof}
Note that by the monotonicity of the reduced volume, we have
\[ 1 \geq \td{V}_{(x,0)} (\tau) > 1 - \delta \textQQqq{for all} \tau \in (0, 1/2]. \]

Assume now that the statement of the Lemma was false and pick arbitrary sequences $\delta_i, \sigma_i \to 0$.
Then we can find a sequence of flows $(M_i, (g^i_t)_{t \in [-2, 0]})$ and basepoints $x_0^i \in M_i$ such that conditions \ref{(i)-Lem:epsregredvolume}--\ref{(iii)-Lem:epsregredvolume} of the lemma hold, but for which $\rrm^{M_i} (x^i_0, 0) \leq \sigma_i$.
Choose $y_i \in B^{M_i} (x^i_0, 0, 1/2)$ such that
\[ a_i := |{\Rm}| (y_i, 0) \big( \tfrac12 - d_0^{M_i} (y_i, x^i_0) \big)^2 \]
is maximal and set $Q_i := |{\Rm}| (y_i, 0)$.
Then $a_i, Q_i \to \infty$ and $|{\Rm}| (\cdot, 0) \leq 4 Q_i$ on $B^{M_i} (y_i, 0, \frac12 (\frac12 - d_0^{M_i} (y_i, x^i_0))$. 
Let $( M_i, (g^{\prime i}_t)_{t \in [-2 Q_i, 0]})$ be the flows that arise from parabolic rescaling of $(M_i, (g^i_t)_{t \in [-2, 0]})$ by $Q_i$.
Then, in these rescaled flows we have $|{\Rm'}| (y_i, 0) = 1$ and $|{\Rm'}| (\cdot, 0) \leq 4$ on $B'(y_i, 0, \frac12 \sqrt{a_i})$ and $\td{V}_{(x, 0)}^{M_\infty, (g^{\prime i}_t)} (\tau) > 1- \delta$ for all $x \in B'(y_i, 0, \frac12 \sqrt{a_i})$ and $0 < \tau \leq Q_i / 2$.
Moreover, by Proposition \ref{Prop:Pseudoloc}, we have 
\begin{equation} \label{eq:Rmprime4}
|{\Rm'}| \leq 4 \eps^{-2} \textQQqq{on} P'_i := P' \big( y_i, 0, \tfrac12 \sqrt{a_i}, - \eps^2 / 4 \big)
\end{equation}for some uniform $\eps = \eps (A) > 0$.
Lastly, due to Proposition \ref{Prop:VolumeBound}, we have $|B' (x,0,1)|_0 > \kappa$ for all $x \in B'(y_i, 0, \frac14 \sqrt{a_i})$, where $\kappa = \kappa (A) > 0$ is some uniform constant.
So by passing to a subsequence, we may assume that the pointed flows $(M_i, (g^{\prime i}_t)_{t \in (-\eps^2 / 4, 0]}, (y_i, 0))$ smoothly converge to some Ricci flow $(M_\infty, (g^\infty_t)_{t \in (-\eps^2 / 4, 0]}, (y_\infty, 0))$ that has complete time-slices and bounded curvature on compact time-intervals.
Note that $|{\Rm}| (y_\infty, 0) = 1$.
Since we took a blow-up sequence, we have $R \equiv 0$ and hence $\Ric \equiv 0$ on $M_\infty \times (- \eps^2 / 4, 0]$.

We now claim that for any $x_\infty \in M_\infty$ and $0 < \tau < \eps^2 / 4$ we have
\begin{equation} \label{eq:tdVlargeinlimit}
  \td{V}^{M_\infty}_{(x_\infty, 0)} (\tau) = 1. 
\end{equation}
To see this, fix $x_\infty \in M_\infty$ and $0 < \tau < \eps^2 / 4$ and consider a sequence $x_i \in M_i$ that converges to $x_\infty$ under the smooth convergence of Ricci flows.
Note that $x_i \in B^{M_i} (x_0, 0, 1)$ for large $i$.

Consider the $\LL$-exponential maps on $(M_i, (g^{\prime i}_t)_{t \in [-2 Q_i, 0]})$ based at $(x_i, 0)$ for the parameter $\tau$
\[ \LL\exp_{(x_i, 0), \tau} : T_{x_i} M_i \longrightarrow M_i, \]
their Jacobians $J^\LL_{(x_i, 0)} (\cdot, \tau) : T_{x_i} M_i \to \IR$ and the subsets $\mathcal{D}^{\LL}_{(x_i, 0), \tau} \subset T_{x_i} M_i$, $\mathcal{G}_{(x_i, 0), \tau}^{\LL} \subset M_i$ as defined in subsection \ref{subsec:LLgeometry}.
Fix some constant $D < \infty$ and define the subsets
\[ S_{D,i} := \LL\exp_{(x_i, 0), \tau} \big( T_{x_i} M_i \setminus B ( 0_{x_i}, D) \big) \subset M_i. \]
Due to the curvature bound (\ref{eq:Rmprime4}), we find that for sufficiently large $i$ (depending on $D$) the following is true:
for all  $v \in B(0_{x_i}, D) \subset T_{x_i} M_i$, the image of the $\LL$-geodesic
\[ \gamma_{v} :  [0, \tau] \to M_i, \qquad \tau' \mapsto \LL\exp_{(x_i, 0), \tau'} (v) \]
lies in $B'(y_i, 0, \frac12 \sqrt{a_i})$.
So the maps $\LL\exp_{(x_i, 0), \tau}$ smoothly converge to $\LL\exp_{(x_\infty, 0), \tau}$ on $B(0_{x_i}, D) \subset T_{x_i} M_i$.
Moreover, there is some constant $D^* < \infty$, which does not depend on $i$, such that $M_i \setminus S_{D,i} \subset B' (x_i, 0, D^*)$ for all $i$.
Note that, if $i$ is large, then for all $z \in M_i \setminus S_{D,i}$, the $\LL$-distance $L_{(x_i, 0)} (z, -\tau)$ is given by the $\LL$-length of an $\LL$-geodesic $\gamma_v$ for some $v \in B(0_{z_i}, D) \subset T_{x_i} M_i$.
So for any $z \in M_\infty$ and $z_i \in M_i$ with $z_i \to z_\infty$ for which $z_i \in M_i \setminus S_{D,i}$ for infinitely many $i$, we have
\[ l_{(x_\infty, 0)} (z, -\tau) \leq \liminf_{i \to \infty} l_{(x_i, 0)} (z_i, - \tau). \]
It follows that
\begin{align*}
 \td{V}^{M_\infty}_{(x_\infty, 0)} (\tau) 
&= \int_M (4 \pi \tau)^{-n/2} e^{-l_{(x_\infty,0)} (z,- \tau)} dg^\infty_{-\tau} (z) \\
&\geq \liminf_{i \to \infty} \int_{M_i \setminus S_{D,i}} (4 \pi \tau)^{-n/2} e^{-l_{(x_i,0)} (z, - \tau)} dg^\infty_{-\tau} (z)  \displaybreak[1] \\
&= \liminf_{i \to \infty} \bigg( \td{V}^{M_i, (g^{\prime i}_t)} (\tau) - \int_{M_i \setminus S_{D,i}} (4 \pi \tau)^{-n/2} e^{-l_{(x_i,0)} (z, -\tau)} dg^\infty_{-\tau} (z) \bigg) \displaybreak[1] \\ 
&\geq  \liminf_{i \to \infty} \bigg( 1 - \delta_i - \int_{\mathcal{D}^\LL_{(x_i, 0), \tau} \setminus B(0_{x_i}, D)} (4 \pi \tau)^{-n/2}  \\
&\qquad\qquad\qquad\qquad\qquad\qquad \cdot \exp \big( {-l_{(x_i,0)} ( \LL\exp_{(x_i, 0), \tau} (v), -\tau)} \big) dv \bigg) \\ 
&\geq 1  - \limsup_{i \to \infty}  \int_{T_{x_i}M \setminus B(0_{x_i}, D)} (4 \pi \tau)^{-n/2} \displaybreak[1] \\
&\qquad\qquad\qquad\qquad\qquad\qquad \cdot \exp \big( {-l_{(x_i,0)} ( \LL\exp_{(x_i, 0), \tau} (v), -\tau)} \big) dv  \displaybreak[1] \\
&\geq 1 - \int_{\IR^n \setminus B(0^n, D)} (4 \pi)^{-n/2} e^{|v|^2 / 4} dv.
\end{align*}
Letting $D \to \infty$ yields that the left-hand side of (\ref{eq:tdVlargeinlimit}) is not smaller than the right-hand side.
The reverse inequality is always true by default.

Since $(M_\infty, (g^\infty_t)_{t \in (-\eps^2 / 4, 0]})$ is Ricci flat, we have
\[ L_{(x_\infty, 0)} (x, - \tau) = \frac{\big( d_0^{M_\infty} (x_0,x) \big)^2}{2 \sqrt{\tau}}. \]
So (\ref{eq:tdVlargeinlimit}) implies that for all $x \in M_\infty$ and $0 < \tau \leq \eps^2 / 4$
\[ \int_{M_\infty} (4 \pi \tau)^{-n/2} \exp \Big({- \frac{d_0^2 (x, z)}{4 \tau}}\Big) dg^\infty_0 (z) = \td{V}_{(x,0)} (\tau) = 1 . \]
So, by volume comparison
\begin{multline*}
 1 = \int_0^\infty (4 \pi \tau)^{-n/2} \frac{r}{2 \tau} e^{- r^2 / 4 \tau} \big| B^{M_\infty} (x, 0, r) \big|_0 dr \\ \leq \int_0^\infty (4 \pi \tau)^{-n/2} \frac{r}{2 \tau} e^{- r^2 / 4 \tau} \cdot \omega_n r^n dr = 1. 
\end{multline*}
It follows that $| B^{M_\infty} (x, 0, r) |_0 = \omega_n r^n$ for all $r > 0$ and hence that $(M_\infty, g^\infty_0)$ is isometric to Euclidean space.
This, however, contradicts $|{\Rm}|(y_\infty, 0) = 1$.
\end{proof}

We can finally prove Proposition \ref{Prop:epsreg} and Corollary \ref{Cor:checkregularityasspt}.

\begin{proof}[Proof of Proposition \ref{Prop:epsreg}]
Let $\delta = \delta(A), \sigma = \sigma (A) > 0$ be the constants from Lemma \ref{Lem:epsregredvolume} and determine $\nu = \nu (\delta)$ from Lemma \ref{Lem:redvolumelarge}.
Set $\eps_0 = \eps_0 (A) := \nu$.
By combining those lemmas we conclude that whenever $\rho$ is sufficiently small, depending on $A, E, \mathbf{p}$, and
\[ |B(x_0, 0, \nu^{-1} ) |_0 > (\omega_n - \nu) ( \nu^{-1} )^n, \]
then $\rrm (x_0, 0) > \sigma$.
The proposition now follows via parabolic rescaling by $(r_0 \nu)^2$.
\end{proof}

\begin{proof}[Proof of Corollary \ref{Cor:checkregularityasspt}]
The $Y$-regularity of the limit follows immediately from Proposition \ref{Prop:epsreg} and assertions \ref{(c)-Prop:convtosingspaceLpasspt} and \ref{(d)-Prop:convtosingspaceLpasspt} of Proposition \ref{Prop:convtosingspaceLpasspt}.
The constant $Y$ can be chosen only depending on $\eps_0$ and $\sigma_0$ of Proposition \ref{Prop:epsreg}, which in turn only depend on $A$.

Alternatively, observe that Proposition \ref{Prop:epsreg} implies property (G) from \cite[subsec 1.2]{Bamler-CGN} for an $A$ that only depends on the $A$ from Proposition~\ref{Prop:convtosingspaceLpasspt}.
The corollary follows now using \cite[Theorem 1.2(c)]{Bamler-CGN}.
\end{proof}

\section{Proof of the main theorems} \label{sec:mainproof}

\subsection{Proof of the integral curvature bound}
We first prove the following covering lemma, whose result we will iterate later.

\begin{Lemma} \label{Lem:rrmsmallcoverbyballs}
For any $A < \infty$ and $0 < \mathbf{p} < 4$ there is a constant $H = H(A, \mathbf{p}) < \infty$ such that:

For any $E' < \infty$ and $0 < \lambda < 1$ there is a constant $0 < \ov{r} = \ov{r} (A, \mathbf{p}, E', \lambda) < 1$ such that the following holds:

Let $(M, (g_t)_{t \in [-2,0]} )$ be a Ricci flow on a compact, $n$-dimensional and orientable manifold $M$ and $0 < r_0 \leq \ov{r}$ a scale with the property that
\begin{enumerate}[label=(\roman*)]
\item \label{(i)-Lem:rrmsmallcoverbyballs} $\nu [ g_{-2}, 4] \geq - A$.
\item \label{(ii)-Lem:rrmsmallcoverbyballs} $|R| \leq 1$ on $M \times [-2,0]$.
\item \label{(iii)-Lem:rrmsmallcoverbyballs} For all $(x,t) \in M \times [-1,0]$ and $0 < r < r_0$ and $0 < s < 1$ we have
\[ | \{ \rrm (\cdot, t) < sr \} \cap B(x,t, r) |_t \leq E' s^{3.1} r^n. \]
\end{enumerate}
Then for any $x \in M$ and $0 < r \leq 10 r_0$ we can find at most $H \lambda^{\mathbf{p} - n}$ many points $y_1, \ldots, y_m \in M$, $m < H \lambda^{ \mathbf{p} - n}$ such that
\begin{equation} \label{eq:lambdaballscover}
 \{ \rrm (\cdot, 0) < \lambda r \} \cap B(x,0,r) \subset \bigcup_{j =1}^m B(y_j, 0, \lambda r). 
\end{equation}
\end{Lemma}

\begin{proof}
Fix the constants $A$ and $\mathbf{p}$ for the rest of the proof and determine $Y = Y(A)$ as the maximum of the corresponding constants from Proposition \ref{Prop:convtosingspaceLpasspt}\ref{(b)-Prop:convtosingspaceLpasspt} and Corollary \ref{Cor:checkregularityasspt}.
So every blow-up limit $\XX$ of Ricci flows $(M_i, (g^i_t)_{t \in [-2,0]})$ that satisfy assumptions \ref{(i)-Lem:rrmsmallcoverbyballs}--\ref{(iii)-Lem:rrmsmallcoverbyballs}, as obtained via Proposition \ref{Prop:convtosingspaceLpasspt}, is a singular space $\XX = (X, d, \RR, g)$ with mild singularities of codimension $3.1$ that satisfies $\Ric \equiv 0$ on $\RR$, that is $Y$-tame and $Y$-regular at all scales.
Consider such a blow-up limit $\XX$ for a moment.
We can now apply \cite[Theorem 1.6]{Bamler-CGN} to $\XX$ and we obtain a constant $E = E(\mathbf{p}, Y(A)) < \infty$, which only depends on $\mathbf{p}$ and $Y$ and hence on $A$, such that $\XX$ satisfies the bound
\begin{equation} \label{eq:blowuplimitE}
 |\{ \rrm^\XX < s r \} \cap B^X (x_\infty , r) \cap \RR | \leq E s^{\mathbf{p}} r^n 
\end{equation}
for all $x_\infty \in X$, $r > 0$ and $0 < s < 1$.
Fix this constant $E = E(\mathbf{p}, Y(A)) $ for the rest of the proof and remember that $E$ only depends on $A$ and $\mathbf{p}$.

Next, we show that for any $E' < \infty$ and $0 < \lambda < 1$ there is a constant $\ov{r} = \ov{r} (A, \mathbf{p}, E', \lambda) > 0$ such that for any $x \in M$ and $0 < r \leq \ov{r}$
\begin{equation} \label{eq:rrmbound2Elambda}
 |\{ \rrm (\cdot, 0) < 2 \lambda r \} \cap B (x, 0, 2r) |_0 < 2 E (2\lambda)^{\mathbf{p}} (2r)^n. 
\end{equation}
Assume that there was no such $\ov{r}$ for some fixed $E', \lambda$.
Then we can find a sequence of Ricci flows $(M_i, (g^i_t)_{t \in [-2, 0]})$ that satisfy assumptions \ref{(i)-Lem:rrmsmallcoverbyballs}--\ref{(iii)-Lem:rrmsmallcoverbyballs} and points $x_i \in M_i$, $0 < r_i <10 r_{0,i}$ with $r_{0,i}, r_i \to 0$ such that for all $i$
\[ \big| \big\{ \rrm^{M_i} (\cdot, 0) < 2\lambda r_i \big\} \cap B^{M_i} (x_i, 0, 2r_i) \big|_0 \geq 2 E (2\lambda)^{\mathbf{p}} (2r_i)^n\]
Let $(M'_i = M_i, (g^{\prime i}_t)_{t \in [-2 (r_i/10)^{-2}, 0]})$ be the flows that arise from $(M_i, (g^i_t)_{t \in [-2, 0]})$ by parabolic rescaling by $(r_i/10)^{-2}$.
After this rescaling, the previous bound becomes
\begin{equation} \label{eq:rrmMprime2Elambda}
 \big| \big\{ \rrm^{M'_i} (\cdot, 0) < 20\lambda  \big\} \cap B^{M'_i} (x_i, 0, 20) \big|_0 \geq 2 E (2\lambda)^{\mathbf{p}} 20^n
\end{equation}
Note that $|R| \leq r_i^2 \to 0$ on $M'_i \times [- (r_i/10)^{-2}, 0]$.
So we can apply Proposition~\ref{Prop:convtosingspaceLpasspt}\ref{(b)-Prop:convtosingspaceLpasspt} to conclude that, after passing to a subsequence, we have convergence of the pointed Riemannian manifolds $(M'_i, g^{\prime i}_0, x_i)$ to some singular space $(\XX, x_\infty)$ with mild singularities of codimension $3.1$ that satisfies $\Ric \equiv 0$ on $\RR$ and that is $Y$-tame and $Y$-regular at all scales.
The convergence can be described by a convergence scheme $\{ (U_i, V_i, \Phi_i) \}_{i=1}^\infty$.
Moreover, the limit space $\XX$ satisfies (\ref{eq:blowuplimitE}).

We will now derive a contradiction by passing (\ref{eq:rrmMprime2Elambda}) to the limit.
Fix some $\eps > 0$ for the moment and observe that we have by Proposition \ref{Prop:convtosingspaceLpasspt}\ref{(e)-Prop:convtosingspaceLpasspt} for large $i$
\begin{multline} \label{eq_inclusions_lemma_covering}
 B^{M'_i} (x_i, 0, 20) \setminus \Phi_i \big( U_i \cap B^X (x_\infty, 20+ \eps)  \big) \subset B^{M'_i} (x_i, 0, 20) \setminus V_i  \\
 \subset \big\{ \rrm^{M'_i} (\cdot, 0) < \eps  \big\} \cap B^{M'_i} (x_i, 0, 20). 
 \end{multline}
Next note, that by a standard ball-packing argument there is a uniform constant $C < \infty$ such that every $20$-ball in $(M_i, g^{\prime i}_0)$ can be covered by at most $C$ many $1$-balls.
Combining this with (\ref{eq_inclusions_lemma_covering}) and the parabolically rescaled assumption \ref{(iii)-Lem:rrmsmallcoverbyballs} we get for large $i$
\[ \big| B^{M'_i} (x_0, 0, 20) \setminus \Phi_i \big( U_i \cap B^X (x_\infty, 20+ \eps) \big)  \big|_0 < CE' \eps^{3.1} 20^n . \]
Combining this with (\ref{eq:rrmMprime2Elambda}) and Proposition \ref{Prop:convtosingspaceLpasspt}\ref{(d)-Prop:convtosingspaceLpasspt}, we find
\[ \big| \{ \rrm^\XX (\cdot, 0) < 20\lambda + \eps \} \cap B^X(x_\infty, 20+\eps) \cap \RR \big| > 2 E \lambda^{\mathbf{p}} - CE' \eps^{3.1} 20^n. \]
For sufficiently small $\eps$ this contradicts (\ref{eq:blowuplimitE}) and hence shows (\ref{eq:rrmbound2Elambda}).

We will now use (\ref{eq:rrmbound2Elambda}) to show (\ref{eq:lambdaballscover}).
To do this, choose $m \in \IN$ maximal such that we can find points $y_1, \ldots, y_m \in \{ \rrm (\cdot, 0) < \lambda r \} \cap B(x, 0, r)$ with the property that the balls $B(y_1, 0, \lambda r / 2), \ldots, B(y_m, 0, \lambda r/2)$ are pairwise disjoint.
Then
\[ B(y_1, 0, \lambda r/2), \ldots, B(y_m, 0, \lambda r/2) \subset \{ \rrm (\cdot, 0) < 2 \lambda r \} \cap B(x, 0, 2r). \]
By Proposition \ref{Prop:VolumeBound} there is a constant $c = c(A) > 0$, which only depends on $A$, such that
\[ |B(y_j, 0, \lambda r / 2) |_0 > c (\lambda r / 2)^n. \]
It follows using (\ref{eq:rrmbound2Elambda}) that
\[ m < \frac{2 E (2 \lambda)^{\mathbf{p}} (2r)^n}{c (\lambda r / 2)^n} = \frac{2 E \cdot 2^{\mathbf{p}+n}}{c (\lambda  / 2)^n} \lambda^{\mathbf{p} - n} =: H \lambda^{\mathbf{p} - n}. \]
Note that $H = H(A, \mathbf{p})$ only depends on $A$ and $\mathbf{p}$.
By the maximal choice of $m$ we conclude (\ref{eq:lambdaballscover}), which finishes the proof.
\end{proof}

Applying Lemma \ref{Lem:rrmsmallcoverbyballs} successively for sufficiently small $\lambda$ yields:

\begin{Lemma} \label{Lem:rrmEpboundonecondition}
For any $A < \infty$ and $0 < \mathbf{p} < 4$ there is a constant $\mathbf{E}_{\mathbf{p}} (A) < \infty$ such that:

For any $E' < \infty$ there is a constant $0 < \ov{r} = \ov{r} ( A, \mathbf{p}, E') < 1$ such that the following holds:

Let $(M, (g_t)_{t \in [-2,0]} )$ be a Ricci flow on a compact, $n$-dimensional and orientable manifold $M$ and $0 < r_0 \leq \ov{r}$ a scale with the property that
\begin{enumerate}[label=(\roman*)]
\item \label{(i)-Lem:rrmEpboundonecondition} $\nu [ g_{-2}, 4] \geq - A$.
\item \label{(ii)-Lem:rrmEpboundonecondition} $|R| \leq 1$ on $M \times [-2,0]$.
\item \label{(iii)-Lem:rrmEpboundonecondition} For all $(x,t) \in M \times [-1,0]$ and $0 < r < r_0$ and $0 < s < 1$ we have
\[ | \{ \rrm (\cdot, t) < sr \} \cap B(x,t, r) |_t \leq E' s^{3.1} r^n. \]
\end{enumerate}
Then for any $x \in M$ and $0 < r \leq 10 r_0$ and $0 < s < 1$ we have
\begin{equation} \label{eq:rrmEpboundunderonecondition1}
 |\{ \rrm (\cdot, 0) < s r \} \cap B(x, 0, r) |_0 < \mathbf{E}_{\mathbf{p}} s^{\mathbf{p}} r^n. 
\end{equation}
\end{Lemma}

\begin{proof}
Fix $A < \infty$ and $0 < \mathbf{p} < 4$ for the rest of the proof and choose some $\mathbf{p}' = \mathbf{p} (\mathbf{p}')$ such that $\mathbf{p} < \mathbf{p}' < 4$.
Let $H = H(A, \mathbf{p}')$ be the constant from Lemma~\ref{Lem:rrmsmallcoverbyballs}.
Based on this constant choose $0 < \lambda = \lambda (A, \mathbf{p}, \mathbf{p'}) < 1$ small enough such that
\[ H \lambda^{\mathbf{p}' - \mathbf{p}} < 1. \]

Consider now the constant $E'$ and choose $\ov{r} = \ov{r} (A, \mathbf{p}', E', \lambda)$ according to Lemma \ref{Lem:rrmsmallcoverbyballs}.
Applying Lemma \ref{Lem:rrmsmallcoverbyballs} multiple times yields that for any integer $k \geq 1$ there are at most $(H \lambda^{\mathbf{p}' - n} )^k$ many points $y_1, \ldots, y_m \in M$ such that
\[ \{ \rrm (\cdot, 0) < \lambda^k r \} \cap B(x, 0, r) \subset \bigcup_{j=1}^m B(y_j, 0, \lambda^k r). \]
By Proposition \ref{Prop:VolumeBound} there is a constant $C = C(A) < \infty$, which only depends on $A$, such that the time-$0$ volume of the balls $B(y_j, 0, \lambda^k r)$ is bounded from above by $C (\lambda^k r)^n$.
Thus
\begin{multline*}
 |\{ \rrm (\cdot, 0) < \lambda^k r \} \cap B(x, 0, r) |_0 < C(\lambda^k r)^n (H \lambda^{\mathbf{p}' - n} )^k \\
 = C \lambda^{\mathbf{p} k} (H \lambda^{\mathbf{p}' - \mathbf{p} } )^k r^n < C \lambda^{\mathbf{p} k}  r^n . 
\end{multline*}
As $\lambda$ and $C$ only depended on $A$ and $\mathbf{p}$, this bound implies (\ref{eq:rrmEpboundunderonecondition1}) for some suitable $\mathbf{E}_{\mathbf{p}} = \mathbf{E}_{\mathbf{p}} (A) < \infty$.
\end{proof}

Next, we show that assumption \ref{(iii)-Lem:rrmEpboundonecondition} in Lemma \ref{Lem:rrmEpboundonecondition} always holds for a suitable $E' = E'(A)$.

\begin{Lemma} \label{Lem:findEprime}
For any $A < \infty$ there are constants $E' = E' (A) < \infty$ and $\ov{r} = \ov{r} (A)$ such that:

Let $(M, (g_t)_{t \in [-2,0]} )$ be a Ricci flow on a compact, $n$-dimensional and orientable manifold $M$ with the property that
\begin{enumerate}[label=(\roman*)]
\item \label{(i)-Lem:findEprime} $\nu [ g_{-2}, 4] \geq - A$.
\item \label{(ii)-Lem:findEprime} $|R| \leq 1$ on $M \times [-2,0]$.
\end{enumerate}
Then for any $(x,t) \in M \times [-1,0]$ and $0 < r < \ov{r}$ and $0 < s < 1$ we have
\begin{equation} \label{eq:rrmEpboundunderonecondition}
 |\{ \rrm (\cdot, t) < s r \} \cap B(x, t, r) |_t < E' s^{3.1} r^n. 
\end{equation}
\end{Lemma}

\begin{proof}
Fix $A$ and choose $E' := \mathbf{E}_{3.1} (A) < \infty$ according to Lemma \ref{Lem:rrmEpboundonecondition}.
Next choose $0 < \ov{r} = \ov{r} (A, 3.1, E') < 1$ according to Lemma \ref{Lem:rrmEpboundonecondition}.
It follows that whenever for all $(x,t) \in M \times [-1,0]$ and $0 < r < \ov{r}$ and $0 < s < 1$ the bound (\ref{eq:rrmEpboundunderonecondition}) holds, then it also holds for all $(x,t) \in M \times \{ 0 \}$ and $0 < r \leq 10 \ov{r}$ and $0 < s < 1$.

Let us now apply this conclusions to parabolic rescalings of $(g_t)_{t \in [-2,0]}$.
Let $0 < r_0 \leq 1$ and $t_0 \in [-2 + 2 r_0^2, 0]$ and consider the parabolic rescaling $g'_{t} := r_0^{-2} g_{r_0^{2} (t + t_0)}$, which is defined for at least all times $t \in [-2, 0]$.
The time-interval $[-2,0]$ for $(g'_t)_{t \in [-2, 0]}$ corresponds to the time-interval $[t_0 - 2 r_0^2 , t_0]$ for $(g_t)_{t \in [-2, 0]}$.
By the monotonicity of the $\nu$-functional we have
\begin{multline*}
 \nu [ g'_{-2} , 4] = \nu \big[ g_{t_0 - 2r_0^2}, 4 a^2 \big] \geq \nu [ g_{-2}, 4 r_0^2 + (t_0 - 2r_0^2  + 2) ] \\
= \nu [ g_{-2}, 2 r_0^2 + t_0 +2 ] \geq \nu [ g_{-2}, 4] \geq - A. 
\end{multline*}
So the rescaled flow $(g')_{t \in [-2, 0]}$ still satisfies assumptions \ref{(i)-Lem:findEprime} and \ref{(ii)-Lem:findEprime}.
Hence, applying our previous conclusion to $(g')_{t \in [-2, 0]}$ and rescaling back shows the following:

{\it Whenever $0 < r_0 \leq 1$ and $t_0 \in [ - 2 + 2 r_0^2, 0]$ and whenever the bound (\ref{eq:rrmEpboundunderonecondition}) holds for all $(x,t) \in M \times [ t_0 - r_0^2, t_0]$ and $0 < r < \ov{r} r_0$ and $0 < s < 1$, then it also holds for all $x \in M$ and $0 < r \leq 10 \ov{r} r_0$ and $0 < s < 1$.}

Assume now that the conclusion of the Lemma was wrong with our choice of $E'$.
Then we can find some $(x_1,t_1) \in M \times [-1, 0]$, $0 < r_1 < \ov{r}$ and $0 < s_1 < 1$ such that (\ref{eq:rrmEpboundunderonecondition}) fails for $x \leftarrow x_1, t \leftarrow t_1$, $r \leftarrow r_1$ and $s \leftarrow s_1$.
By the contrapositive of our previous conclusion for $t_0 \leftarrow t_1$, $r_0 \leftarrow \frac1{10} \ov{r}^{-1} r_1$ this implies that we can find some $x_2 \in M$, $t_2 \in [t_1 - 2 (\frac1{10} \ov{r}^{-1} r_1)^2, t_1]$, $0 < r_2 < \frac1{10}  r_1$ and $0 < s_2 < 1$ such that (\ref{eq:rrmEpboundunderonecondition}) fails for $x \leftarrow x_2, t \leftarrow t_2$, $r \leftarrow r_2$ and $s \leftarrow s_2$.
Repeating this argument yields a sequence $(x_1, t_1, r_1, s_1), (x_2, t_2, r_2, s_2), \ldots$ such that $r_{k+1} \leq \frac1{10} r_k$ and $| t_{k+1} - t_k | \leq 2 (\frac1{10} \ov{r}^{-1} r_k )^2$ and such that (\ref{eq:rrmEpboundunderonecondition}) fails for $x \leftarrow x_k, t \leftarrow t_k$, $r \leftarrow r_k$ and $s \leftarrow s_k$  for all $k = 1, 2, \ldots$.
As $r_k < (\frac1{10})^{k-1} \ov{r}$, we find that $t_k$ stays within $[-1.5, 0]$.
So the process can be continued indefinitely.
However, by the smoothness of $(g_t)_{t \in [-2,0]}$ there is some large $k$ for which the left-hand sided of (\ref{eq:rrmEpboundunderonecondition}) is zero for $x \leftarrow x_k, t \leftarrow t_k$, $r \leftarrow r_k$ and $s \leftarrow s_k$.
This gives us the necessary contradiction.
\end{proof}

We can finally state our most general bound on the sublevel sets of $\rrm$.

\begin{Proposition} \label{Prop:Mainsublevelrrmbound}
For any $A < \infty$ and $0 < \mathbf{p} < 4$ there is a constant $\mathbf{E} = \mathbf{E}_{\mathbf{p}} (A) < \infty$ such that:

Let $(M, (g_t)_{t \in [-2,0]} )$ be a Ricci flow on a compact, $n$-dimensional manifold $M$ with the property that
\begin{enumerate}[label=(\roman*)]
\item \label{(i)-Prop:Mainsublevelrrmbound} $\nu [ g_{-2}, 4] \geq - A$.
\item \label{(ii)-Prop:Mainsublevelrrmbound} $|R| \leq 1$ on $M \times [-2,0]$.
\end{enumerate}
Then for any $(x,t) \in M \times [-1,0]$ and $0 < r, s <1$ we have
\[
 |\{ \rrm (\cdot, t) < s r \} \cap B(x, t, r) |_t < \mathbf{E}_{\mathbf{p}} s^{\mathbf{p}} r^n. 
\]
\end{Proposition}

\begin{proof}
This follows from Lemmas \ref{Lem:rrmEpboundonecondition}, \ref{Lem:findEprime} and a covering argument using Proposition \ref{Prop:VolumeBound}.
In the non-orientable case, we need to pass to the orientable double cover.
\end{proof}

As a consequence, we obtain Theorem \ref{Thm:mainrrmbound}.

\begin{proof}[Proof of Theorem \ref{Thm:mainrrmbound}.]
The theorem is a consequence of Proposition~\ref{Prop:Mainsublevelrrmbound}.
More specifically, note that by the maximum principle applied to the evolution equation for the scalar curvature
\[ \partial_t R = \triangle R + 2 |{\Ric}|^2 \geq \triangle R + \frac2n R^2 \]
we have $R \geq - \frac{n}{2(2+t)}$ on $M \times (-2, 0]$.
So by assumption \ref{(i)-Prop:Mainsublevelrrmbound} we have $- \frac{n}{2} \leq R < A$ on $M \times [-1,0]$.
By parabolic rescaling and with the help of a covering argument we may therefore reduce our proof to the case $|R| \leq 1$ on $M \times [-2,0]$.
See also the discussion in subsection~\ref{subsec:terminology} on how parabolic rescaling affects the bound of the $\nu$-functional in assumption \ref{(ii)-Thm:mainrrmbound}.

We can now use Proposition~\ref{Prop:Mainsublevelrrmbound} for $\mathbf{p} = 4 - \eps$ to conclude
\begin{align*}
 \int_{B(x,t,r)} \big( &\rrm (\cdot, t) \big)^{-4+2\eps} dg_t \\
&= \int_{B(x,t,r)} \int_{r^{-4+2\eps}}^\infty \chi_{s < \rrm^{-4+2\eps} (\cdot, t)} ds dg_t + r^{-4+2\eps} |B(x,t,r)|_t \\
&= \int_{r^{-4+2\eps}}^\infty \big| \big\{ \rrm (\cdot, t) < s^{- \frac{1}{4 - 2 \eps}} \big\} \cap B(x,t,r) \big|_t ds + C(A) r^{n-4+2\eps} \\
&\leq \int_{r^{-4+2\eps}}^\infty \mathbf{E}_{\mathbf{p}} s^{- \frac{4 - \eps}{4 - 2\eps}} r^{n-4+\eps} ds+ C(A) r^{n-4+2\eps}
\leq C(A, \mathbf{E}_{\mathbf{p}}, \eps) r^{n-4+2\eps}.
\end{align*}
This finishes the proof.
\end{proof}

\subsection{Convergence of the flow away from a subset of codimension 4}
Theorem \ref{Thm:maincompactness}, except for the assertion involving the uniform convergence in the case $\rho_i \to 0$, now follows immediately:

\begin{proof}[Proof of Theorem \ref{Thm:maincompactness}.]
The theorem follows from Propositions \ref{Prop:convtosingspaceLpasspt} and \ref{Prop:Mainsublevelrrmbound}
\end{proof}

Next we present the proof of Theorems \ref{Thm:AnswerQuestion1} and \ref{Thm:AnswerQuestion2}.

\begin{proof}[Proof of Theorem \ref{Thm:AnswerQuestion1}.]
By \cite[Corollary 1.2]{Bamler-Zhang-Part2}, we have uniform convergence $d_t \to d_T$ for $t \nearrow T$.
The function $d_T : M \times M \to [0, \infty)$ is a pseudometric.
Write $x \sim y$ if $d(x, y) = 0$.
Then $(M, d_T)$ descends to a metric space $(M^* := M / \sim, d^*_T)$.

Let $q \in M$.
By the first part of Theorem \ref{Thm:maincompactness} (or Propositions \ref{Prop:convtosingspaceLpasspt} and \ref{Prop:Mainsublevelrrmbound}), we obtain that there is a sequence $t_i \nearrow T$ such that $(M, g_{t_i}, q)$ converges to a pointed singular space $(X, d, \RR^*, g, q_\infty)$ with codimension 4 singularities, according to some convergence scheme $\{ (U_i, V_i, \Phi_i) \}_{i = 1}^\infty$.
It follows that $(X,d)$ is isometric to $(M^*, d^*_T)$.
So identify in the following $(X,d)$ with $(M^*, d^*_T)$ and view $\RR^* \subset M^*$.
Let $\RR \subset M$ be the preimage of $\RR^*$ under the canonical projection $\pi : M \to M^* = M / \sim$.

By looking at the construction of the convergence scheme $\{ (U_i, V_i, \Phi_i) \}_{i = 1}^\infty$, we may assume that for all $x \in \RR^*$ we have $\Phi_i (\pi (x)) \to x$ as $i \to \infty$.
Thus, by Proposition \ref{Prop:convtosingspaceLpasspt}\ref{(d)-Prop:convtosingspaceLpasspt} we have $\lim_{i \to \infty} \rrm (x,t_i) = \rrm^\infty (x)$, where $\rrm^\infty$ denotes the curvature radius on $\XX$.
Using forward pseudolocality, Proposition \ref{Prop:Pseudoloc}, we obtain that the metric $g_t$ smoothly converges to some metric $g_T$ on $\RR$ as $t \nearrow T$ and $\rrm (x, T) := \lim_{t \nearrow T} \rrm (x,t)$ exists for all $x \in M$.

Assume now that $x \in \RR$ and $y \in M$ such that $x \sim y$.
Then $x = y$, because for $t$ sufficiently close to $T$ we have $d_t (x,y) < \rrm (x) / 2$ and the convergence $g_t \to g_T$ is smooth on $B(x, T, \rrm (x, T) / 2)$.
So $\pi |_{\RR} : \RR \to \RR^*$ is a bijective map and $\pi^* g = g_T$.
This finishes the proof of the theorem.
\end{proof}

\begin{proof}[Proof of Theorem \ref{Thm:AnswerQuestion2}.]
Recall the basepoint $q \in M$, choose a sequence $t^*_j \nearrow T$ and consider the kernels
\[ K_j (x,t) := K(q, t^*_j ; x,t) \]
to the conjugate heat operator $\partial_t + \triangle - R$.
So for every fixed $t \in [0, t^*_j)$ we have
\begin{equation} \label{eq:Kitotalmass1}
 \int_M K_j (x, t) dg_t (x) = 1 \textQQqq{and} - \partial_t K_j = \triangle K_j - R K_j. 
\end{equation}

By Perelman's Harnack inequality for the conjugate heat equation (cf \cite[9.5]{PerelmanI}), we have
\begin{equation} \label{eq_Ki_LL_geometry_Harnack}
 K_j (x,t) \geq \frac1{(4\pi (t^*_j - t))^{n/2}} \exp \big({ - l_{(q,t^*_j)} (x,t) }\big). 
\end{equation}
By comparison with the constant curve, we obtain the bound
\begin{multline*}
  l_{(q,t^*_j)} (q,t)
 \leq \frac1{2 \sqrt{t^*_j-t}} \int_0^{t^*_j-t} \sqrt{\tau} R(q, t^*_j - \tau) d\tau \\
 \leq \frac1{2 \sqrt{t^*_j-t}} \int_0^{t^*_j-t} \sqrt{\tau} \cdot \frac{C}{T- (t^*_j - \tau)} d\tau 
 \leq \frac1{2 \sqrt{t^*_j-t}} \int_0^{t^*_j-t} \frac{C}{\sqrt{\tau}} d\tau \leq C.
\end{multline*}
Combining this with (\ref{eq_Ki_LL_geometry_Harnack}) yields the following bound at $q$ for some $c_0 > 0$
\begin{equation} \label{eq_Ki_q_lower_bound}
 K_j (q, t) \geq \frac1{(4\pi (t^*_j - t))^{n/2}} e^{-C} \geq \frac{c_0}{(t^*_j - t)^{n/2}}. 
\end{equation}

Next, we use the reproduction formula
\begin{equation} \label{eq_reproduction_formula_Thm_proof}
  K_j(x,t) = \int_M K_j(y, \tfrac12 (T+t))  K(y, \tfrac12 (T + t) ; x,t) dg_{\frac12 (T+t)}(y),
\end{equation}
to derive further bounds on $K_j(x,t)$ for large $j$.
To see why (\ref{eq_reproduction_formula_Thm_proof}) holds, recall that both $K_j (x,t)$ and
\[ \td{K} (x, t') := \int_M K_j(y, \tfrac12 (T+t))  K(y, \tfrac12 (T + t) ; x,t') dg_{\frac12 (T+t)}(y) \]
satisfy the conjugate heat equation and agree for $t = t' = \tfrac12 (T + t)$.

We first note that we have the scalar curvature bound $R \leq 2C (T -t)^{-1}$ on $M \times [0, \frac12 (T+t)]$.
So by Proposition~\ref{Prop:GaussianHKbounds} we have an upper bound of the form
\begin{equation} \label{eq_upper_HK_bound_Thm_proof}
 K (y, \tfrac12 (T+t); x,t) \leq \frac{C'}{(T-t)^{n/2}} \exp \bigg({ - \frac{d^2_{\frac12 (T+t)} (x,y)}{C' (T - t)} }\bigg) \leq \frac{C'}{(T-t)^{n/2}}
\end{equation}
for some uniform $C' < \infty$.
Combining this with (\ref{eq:Kitotalmass1}) applied at time $\frac12 (T+t)$ yields that for large $j$
\begin{equation} \label{eq_Proof_Thm_Ki_upper_bound}
 K_j (x, t) \leq \frac{C'}{(T-t)^{n/2}}. 
\end{equation}

In order to deduce a lower bound on $K_j (\cdot, t)$, we observe that by (\ref{eq_upper_HK_bound_Thm_proof}) there is some uniform $D < \infty$ such that for
\[ B_{\frac12(T+t)} := B \big( q,\tfrac12 (T+t), D \sqrt{T-t} \big) \]
we have
\[ K (\cdot, \tfrac12 (T+t); q,t) \leq \frac{c_0 / 2}{(T-t)^{n/2}}  \textQQqq{on} M \setminus B_{\frac12(T+t)}, \]
where $c_0 > 0$ denotes the constant from (\ref{eq_Ki_q_lower_bound}).
So by (\ref{eq_Ki_q_lower_bound}), (\ref{eq_upper_HK_bound_Thm_proof}) and (\ref{eq:Kitotalmass1}) we obtain that for large $j$
\begin{align*}
 \frac{c_0}{(T - t)^{n/2}} &\leq \frac{c_0}{(t^*_j - t)^{n/2}} \leq K_j (q, t) \\
&\leq \int_{B_{\frac12 (T+t)}} K_j (y, \tfrac12 (T+t)) \cdot \frac{C'}{ (T-t)^{n/2}} dg_{\frac12 (T+t)} (y)  \\
&\qquad + \int_{M \setminus B_{\frac12 (T+t)}} K_j (y, \tfrac12 (T+t)) \cdot \frac{c_0/2}{(t^*_j - t)^{n/2}} dg_{\frac12 (T+t)} (y)  \\
&\leq \frac{C'}{ (T-t)^{n/2}} \int_{B_{\frac12 (T+t)}} K_j (y, \tfrac12 (T+t)) dg_{\frac12 (T+t)} (y) + \frac{c_0/2}{(T - t)^{n/2}}. 
\end{align*}
Thus
\[ \int_{B_{\frac12 (T+t)}} K_j (y, \tfrac12 (T+t)) dg_{\frac12 (T+t)} (y) \geq \frac{c_0}{2C'}. \]
So if we restrict the domain of the integral in (\ref{eq_reproduction_formula_Thm_proof}) to $B_{\frac12 (T+t)}$ and use the lower Gaussian bound on $K(y, \frac12 (T+t); x,t)$ from Proposition~\ref{Prop:GaussianHKbounds} and the distance distortion bound from Proposition~\ref{Prop:Distdistortion}, then we obtain that for large $j$
\[ K_j (x,t) \geq \frac{c_1}{(T-t)^{n/2}} \exp \bigg({ - \frac{d^2_t (q,x)}{c_1 (T-t)} } \bigg), \]
for some uniform $c_1 > 0$.

By (\ref{eq_Proof_Thm_Ki_upper_bound}) and local parabolic regularity, we obtain local bounds on higher covariant derivatives of $K_j$, which may depend on space and time, but which are independent of $j$.
So after passing to a subsequence, we have smooth convergence $K_j \to u \in C^\infty (M \times [0,T))$ on compact time-intervals.
The limit $u$ is a positive solution to the conjugate heat equation $- \partial_t u = \triangle u - R u$ and for all $t \in [0,T)$ we have
\[ \int_M u(x, t) dg_t (x) = 1 \]
and
\begin{equation} \label{eq:upperlowerubound}
 \frac{c_1}{(T-t)^{n/2}} \exp \bigg({ - \frac{d^2_t (q,x)}{c_1 (T-t)} } \bigg) \leq u (x, t) \leq \frac{C'}{(T-t)^{n/2}}. 
\end{equation}

Write now
\[ K_j (x,t) = (4 \pi (t^*_j - t))^{-n/2} e^{-f_j (x,t)} \textQq{and} u(x,t) = (4 \pi (T - t))^{-n/2} e^{-f(x,t)}. \]
Then $f_j \to f$ smoothly on compact time-intervals.
For any $t \in [0,T)$ we have
\[ w(t) := \mathcal{W} [ g_t , f(\cdot, t), T - t] = \lim_{j \to \infty} \mathcal{W} [ g_t , f_j(\cdot, t), t^*_j - t] \leq 0. \]
Recall also that $w (t)$ is non-decreasing in $t$.
More specifically,
\[ w'(t) = \int_M 2 (T-t) \Big| \Ric (\cdot, t) + \nabla^2 f (\cdot, t) - \frac1{2(T-t)} g_t \Big|^2 u(\cdot, t) dg_t. \]
So, given any sequence $t_i \nearrow T$, for any $\theta > 0$ we have
\begin{equation} \label{eq:inttisoliton}
 \lim_{i \to \infty} \int_{ t_i - \theta (T - t_i)}^{t_i} \int_M 2 (T-t) \Big| \Ric (\cdot, t) + \nabla^2 f (\cdot, t) - \frac1{2(T-t)} g_t \Big|^2 u(\cdot, t) dg_t dt = 0. 
\end{equation}

We now use Theorem \ref{Thm:maincompactness} (or Propositions \ref{Prop:convtosingspaceLpasspt} and \ref{Prop:Mainsublevelrrmbound}) to conclude that $(M, (T-t_i)^{-1} g_{t_i}, q)$ converges to a pointed singular space $(\XX, q_\infty) = (X, d, \RR, g, q_\infty)$ whose singularities have codimension $4$.
Let $\{ (U_i, V_i, \Phi_i ) \}_{i = 1}^\infty$ be a scheme for this convergence.
Consider now the functions $f^*_i (x) := f (\Phi_i (x), t_i)$, $f^*_i \in C^\infty (U_i)$.
We will show that, after passing to a subsequence, these functions limit to a smooth function $f_\infty \in C^\infty (\RR)$ that satisfies the gradient Ricci soliton equation.
\begin{equation} \label{eq:solitoninproof}
 \Ric_{g} + \nabla^2 f_\infty - \tfrac12 g = 0. 
\end{equation}
This question can be analyzed locally.
So let $x_0 \in \RR$ and choose $0 < 2 \sigma < \rrm^\infty (x)$.
Then for large $i$ we have $\rrm (\Phi_i (x_0), t_i) > \sigma (T - t_i)^{1/2}$.
Using backward and forward pseudolocality, we find that $\rrm (\Phi_i (x_0), t_i) > \eps \sigma (T - t_i)^{1/2}$ on $P(\Phi_i (x_0), t_i, \eps \sigma (T - t_i)^{1/2}, \pm \eps^2 \sigma^2 (T- t_i))$ for some uniform $\eps > 0$.
Using local parabolic derivative estimates and (\ref{eq:upperlowerubound}), we obtain that 
\begin{equation} \label{eq:nabmfuniformCm}
 |\nabla^m f | < C_m (T - t_i)^{-m/2} \textQQqq{for} m = 0, 1, \ldots 
\end{equation}
on $P(\Phi_i (x_0), t_i, \frac12 \eps \sigma (T - t_i)^{1/2}, \pm \frac12 \eps^2 \sigma^2 (T- t_i))$ for some uniform constants $C_m < \infty$.
Thus, after passing to a subsequence, these functions $f_\infty^*$ smoothly converge to a smooth function $f_\infty$ on $B^X (x_0, \frac14 \eps \sigma)$.
If $f_\infty$ did not satisfy the gradient Ricci soliton equation (\ref{eq:solitoninproof}) at $x_0$, then by the smooth convergence, for large $i$ we would have
\[  (T-t)^2 \Big| \Ric (\cdot, t) + \nabla^2 f (\cdot, t) - \frac1{2(T-t)} g_t \Big|^2 > c''  \]
at $\Phi_i(x_0)$ and $t = t_i$ for some uniform $c'' > 0$.
Due to the local derivative estimates (\ref{eq:nabmfuniformCm}) and Shi's derivative estimates on the curvature, this would imply that the same bound, possibly for a smaller $c''$, holds on $P(\Phi_i(x_0), t_i, \eps' (T- t_i)^{1/2}, - \eps' (T - t_i))$ for some uniform $\eps' > 0$.
This would however contradict (\ref{eq:inttisoliton}).
So $f_\infty$ satisfies indeed (\ref{eq:solitoninproof}).
\end{proof}

The Hamilton-Tian Conjecture follows now immediately.

\begin{proof}[Proof of Corollary \ref{Cor:AnswerHT}.]
The Hamilton-Tian Conjecture follows from Theorem \ref{Thm:AnswerQuestion2} and the upper bounds on the scalar curvature and the diameter as derived in \cite{Sesum-Tian:2008}.
\end{proof}

\bibliography{RF-bounded-scal}{}
\bibliographystyle{amsalpha}
\end{document}